\documentclass[ejs]{imsart}

\usepackage{mathtools}
\DeclarePairedDelimiter\abs{\lvert}{\rvert}
\DeclarePairedDelimiter\ceil{\lceil}{\rceil}
\DeclarePairedDelimiter\floor{\lfloor}{\rfloor}

\usepackage{amsmath,amsthm,amsfonts,amssymb,bm,bbm,dsfont}
\usepackage{stmaryrd}

\usepackage[authoryear, round]{natbib}
\RequirePackage[colorlinks,citecolor=blue,urlcolor=blue]{hyperref}
\RequirePackage{graphicx}

\usepackage{xspace}

\usepackage[table,dvipsnames]{xcolor}
\usepackage[normalem]{ulem}

\pubyear{2025}
\volume{19}
\doi{10.1214/25-EJS2380}

\startlocaldefs
\numberwithin{equation}{section}
\theoremstyle{plain}
\newtheorem{theorem}{Theorem}[section]

\newtheorem{prop}[theorem]{Proposition}
\newtheorem{lem}[theorem]{Lemma}
\newtheorem{remark}[theorem]{Remark}

\providecommand{\ns}[1]{\lVert#1\rVert}

\usepackage{notations}
\usepackage{pipeline_ejs}

\DeclareMathAlphabet{\pazocal}{OMS}{zplm}{m}{n}

\endlocaldefs

\begin{document}

\begin{frontmatter}
\title{\macrotitle}
\runtitle{\macrorunningtitle}

\begin{aug}
\author{\fnms{Sébastien~J.} \snm{Petit}\thanksref{t1,t2}\ead[label=e1]{sebastien.petit@lne.fr}}

\address{\printead{e1}}

\thankstext{t1}{Laboratoire National de Métrologie et d’Essais, 78197, Trappes Cedex, France}
\thankstext{t2}{Universit\'e Paris-Saclay, CNRS, CentraleSup\'elec, Laboratoire des signaux et syste\`emes, 91190, Gif-sur-Yvette, France}
\runauthor{S.J. Petit}

\end{aug}

\begin{abstract}
   This work considers parameter estimation for Gaussian process interpolation with a periodized version of the Matérn
   covariance function.
   Convergence rates are studied for the joint maximum likelihood estimation of the regularity
   and the amplitude parameters
   when the data are sampled according to the model.
   The mean integrated squared error is also analyzed with fixed and estimated
   parameters, showing that maximum likelihood estimation yields asymptotically the same
   error as if the ground truth was known.
   Finally, the case where the observed function is a fixed deterministic element of a
   Sobolev space of continuous functions is
   also considered, suggesting that a joint estimation
   does not select the regularity parameter as if the amplitude were fixed.
\end{abstract}



\tableofcontents
\end{frontmatter}

\section{Introduction}\label{sec:intro}

Gaussian process interpolation or kriging is a common technique for inferring an unknown function from noiseless data,
which has applications in geostatistics \citep{stein1999interpolation}, computer experiments \citep{Santner},
and machine learning \citep{Rasmussen}.
A covariance function fully characterizes a zero-mean Gaussian process model.
The need for tailoring this function to the task at hand is widely acknowledged in the literature.
The common practice consists in choosing it within a parametric family.
\citet{stein1999interpolation} promotes using
the~\citet{matern1986:_spatial_variations} family of stationary covariance functions.
Assuming isotropy and using the
parameterization
from \citet[p. 31]{stein1999interpolation},
this family is defined on $\RR^d$ by its spectrum:
\begin{equation}\label{eq:matern_general_case}
\tilde{k}\colon \omega \in \RR^d \mapsto
\frac{\phi}{\left( \alpha^2 + \ns{\omega}^2 \right)^{\nu + d/2}},
\end{equation}
which is indexed by three parameters: the \emph{regularity} parameter~$\nu$,
what we shall call the \emph{amplitude}
parameter~$\phi$, and the parameter~$\alpha$. See \citep{stein1999interpolation} for a comprehensive
description of the effect of these parameters.
In short, the parameter~$\nu$
is shown to be the key quantity governing the asymptotics of the prediction error.
The amplitude parameter $\phi$ does not impact the posterior mean
predictions but matters for uncertainty quantification, whereas~$\alpha$
is less important asymptotically. 

One can safely say that cross-validation
and maximum likelihood estimation are the most popular techniques for selecting
Gaussian process parameters from data. 
We shall focus on the latter for the rest of this article.

For observations from a Matérn process with
parameter~$\theta_0 = \left(\nu_0, \phi_0, \alpha_0\right)$,
a distinction is often made between increasing and fixed-domain asymptotic
frameworks \citep[see, e.g.,][for a review]{bachoc2021asymptotic}.
While several increasing-domain asymptotic frameworks have been exhaustively studied
\citep[see, e.g.,][]{mardia1984:_mle_incr, bachoc2014:_asympt},
a comprehensive asymptotic analysis of maximum likelihood estimation in fixed-domain
frameworks---i.e., on bounded domains---has
long been an open question.
Previous works mainly consider the estimation
of~$\phi$ and~$\alpha$
for a known~$\nu$
\citep[see, e.g.,][who often use alternative parametrizations]{ying1991:_ou, ying1993:_maximum,
vdv96:_ying_correction, zhang2004:_micro, Loh_2006, kaufman2013:_role, li2020bayesian}.

The asymptotics of~$\hat{\nu}_n$ seem to have been less studied.
\citet[][Section 6.7]{stein1999interpolation}
considers a periodized version of the Matérn
model~\eqref{eq:matern_general_case} %
with equispaced
observations on the torus and makes a conjecture about the asymptotic behavior of the joint maximum likelihood
estimate~$\hat{\theta}_n = (\hat{\nu}_n, \, \hat{\phi}_n, \, \hat{\alpha}_n)$ based
on the Fisher information matrix \citep[see also][who considers equispaced noisy observations on
the circle]{stein93:_splines_estimated_order}. %
This topic has only recently regained popularity.
Indeed, \citet{chen2021:_consistency} used the same framework as~\citeauthor{stein1999interpolation} to show
that~$\hat{\nu}_n$ is consistent if the other parameters remain fixed
(i.e., enforced to arbitrary values, which may not be~$\phi_0$ and~$\alpha_0$).
Continuing with fixed~$\phi$ and~$\alpha$,
\citet{karvonen2022:_asymptotic} has recently shown 
that~$\liminf \hat{\nu}_n \geq \nu_0$
in the case of quasi-uniform observations on
a ``nice'' bounded domain of~$\RR^d$. %
Similarly, \citet{korte2023smoothness} have shown that~$\hat{\nu}_n$ is consistent
for fixed~$\phi$ and~$\alpha$ and
quasi-uniform observations on closed Riemannian manifolds. %
The asymptotic analysis of the joint maximum likelihood estimation of~$\nu$ and~$\phi$
remains an open problem \citep[Section~3.1.2]{porcu2024matern}. %

Another long-standing open problem
(see notably \citealt{putter_young:_estimated} and \citealt[in the preface]{stein1999interpolation})
is that of predictions with estimated parameters:
how accurate and reliable are the predictions if one selects a
parameter~$\hat{\theta}_n$
from data and uses it to make subsequent
predictions?
The critical influence of~$\nu$ on the kriging error suggests that
the asymptotic behavior of~$\hat{\nu}_n$ is a key element
in answering this question.

Another research line consists in studying parameters estimation
assuming observations from a fixed deterministic function~$f$.
The definition of a ground truth~$\theta_0$ is not obvious in this setting.
Instead, the aim is to study which ``features'' of~$f$ are used
by the estimator to select a Gaussian process model and how this affects
predictions.
See \citet{karvonen2020maximum, karvonen2022:_maximum_lik_ill}
for analyses of maximum likelihood estimators of other parameters given a fixed regularity.
Regarding~$\hat{\nu}_n$, the tight lower bound shown by~\citet{karvonen2022:_asymptotic}
also covers the case of a continuous function from a Sobolev space.
The result shows an interesting connection with sample path properties.
More precisely, define the smoothness~$\nu_0(f)$ of~$f$ in a Sobolev sense so that~$\nu_0(\xi) = \nu_0$ holds almost surely
for any Matérn process~$\xi$ with regularity~$\nu_0$.
For fixed~$\phi$ and~$\alpha$,
\citet{karvonen2022:_asymptotic} showed
that~$\liminf \hat{\nu}_n \geq \nu_0(f)$ and,
under (essentially) a self-similarity hypothesis on the spectrum of~$f$,
that~$\hat{\nu}_n$
converges to~$\nu_0(f)$.
This means that, if the spectrum of~$f$ is well-behaved,
then maximum likelihood estimation fits~$\nu$
so that~$f$ and the sample paths have the same Sobolev smoothness.
It echoes similar findings in Bayesian nonparametric statistics
with noise-corrupted observations
(see notably \citealt{belitser_ghosal2003:_infinite_dim},
\citealt[p.~779]{knapik2016bayes}, and \citealt[pp.~1397 and~1404]{szabo2015frequentist}),
where, with our notations, similar conditions on
the truth imply that $\hat{\nu}_n \to \nu_0(f)$. 

This article focuses on the one-dimensional version of the framework used by
\citet[][Section 6.7]{stein1999interpolation} to analyze the joint maximum likelihood estimation
of~$\left(\nu, \phi, \alpha\right)$.
This kind of model has limited applicability but is usually studied because of its simplicity,
in the hope that conclusions can be transferred to more generally applicable models \citep[see][]{wahba1975smoothing, Wahba90a,
stein93:_splines_estimated_order, 
stein1997efficiency,
stein1999interpolation,
stein_2014:_limitations_low_rank_approximations_for_covariance_matrices_of_spatial_data,
chen2021:_consistency}.
In particular, this framework enables fairly explicit derivations, which could shed light
on potential steps to generalize the results
(see notably Section~\ref{sec:prior_works}). %

On the one hand,
a~$\sqrt{n}$-rate asymptotic normality result is shown for---a linear transform of---$(\hat{\nu}_n, \, \hat{\phi}_n)$
when observing a Matérn process.
Whether the (non-identifiable) parameter~$\alpha$ is known or estimated does not affect the limiting distribution.
Furthermore, one consequence is that the ratio between the mean squared error
with estimated parameters and the one with known parameters converges to unity.
On the other hand, it is shown that a joint estimation does not result
in the behavior discussed in the previous paragraph.
The key takeaway is that only the smaller asymptotic bound~$\liminf \hat{\nu}_n \geq \nu_0(f) - 1/2$ holds.
This means that the reproducing kernel Hilbert space is asymptotically too small to contain~$f$ but does not
say whether the Sobolev smoothness of the sample paths exceeds or converges to~$\nu_0(f)$.
To give a quantitative description of the behavior above~$\nu_0(f) - 1/2$, we derive the large sample limit of the (profile) likelihood on a class of functions that is
small but satisfies the usual spectrum conditions ensuring that~$\hat{\nu}_n \to \nu_0(f)$
for fixed~$\phi$ and~$\alpha$.
The minimizer of this limit has no closed-form expression~(see~\ref{eq:dev_nll_misp}),
but we show that a numerical approximation is not maximized by~$\nu_0(f)$.
A~strong consistency result on sample paths shows that the set of functions~$f$ such that $\hat{\nu}_n \to \nu_0(f)$
has probability one under a Matérn process. %
The findings are illustrated by numerical experiments. %

To summarize, the contributions of the present article are threefold. First, we
prove consistency and asymptotic normality results on the maximum likelihood estimates of
the parameters~$\nu$ and~$\phi$.
Then, we leverage these convergence rates to analyze the expected integrated error,
showing that estimating the parameters yields the same error asymptotically as if the ground truth was known. Finally, we investigate model selection by maximum likelihood estimation
on a deterministic function.

The article is organized as follows. Section~\ref{sec:context} introduces the
periodic framework and our notations and
Section~\ref{sec:prior_works} discusses how this framework helps for circumventing the challenges posed by the study of the profile likelihood.
Then, Section~\ref{sec:results} gives the main results. %
Section~\ref{sec:num_random} presents numerical experiments illustrating the findings. %
Section~\ref{sec:freq} provides our results on the deterministic case.
Finally, Section~\ref{sec:concl} presents our conclusions and discusses generalization.

\section{Gaussian process interpolation on the circle}\label{sec:context}

\subsection{Framework}\label{sec:framework}

Let $f : \left[0, \, 1\right] \to \Rset$ be a continuous periodic function observed on
a regular grid:
$\{ j/n, \ 0 \leq j \leq n-1 \}$.
Consider the periodic version
of the Matérn family of stationary covariance functions~\eqref{eq:matern_general_case}
introduced by \citet[Section 6.7]{stein1999interpolation} and
defined by
the uniformly absolutely convergent
Fourier series
$$k_{\theta}\colon x \in  \RR \mapsto \sum_{j \in \Zm} \underline{c}_j(\theta) e^{2\pi i x j}$$%
with coefficients:
\begin{equation}\label{eq:matern}
\underline{c}_j(\theta) = \frac{\phi}{(\alpha^2 + j^2)^{\nu + 1/2}}, \quad \mathrm{for} \, j \in \Zm \ \mathrm{and} \,  \theta = (\nu, \phi, \alpha) \in (0, +\infty)^3.
\end{equation}%
The function $k_\theta$ is continuous and strictly positive definite
\citep[see, e.g.,][Theorem 1]{Gneiting2011:_strictly_positive}.
The description of the parameters~$\nu$,~$\phi$, and~$\alpha$ from the \hyperref[sec:intro]{Introduction}
carries to this periodic one-dimensional version.
A specificity is that~$\alpha$ is
not identifiable as different values yield equivalent
probability measures. However, $\nu$ and $\phi$ are identifiable
\citep[see, e.g.,][Chapter 4 and Section 6.7]{stein1999interpolation}.

Assuming a centered process,
the usual task in Gaussian process interpolation is to use the model
$\xi \sim \mathrm{GP} \left(0, \, k_{\theta} \right)$ to infer the function~$f$ from the noiseless data
\begin{equation}\label{eq:obs}
Z = \left( f \left( 0 \right), \, f \left( 1/n \right), \dots, f \left( 1 - 1/n \right) \right)\tr.
\end{equation}%
The function $f$ is usually predicted using the posterior mean function
given by the kriging equations~\citep{matheron71}.
This predictor can be written simply in the framework presented above.
\begin{prop}\label{prop:blup}
Let $n \geq 1$ and $f\colon [0, \, 1] \to \RR$ be a continuous periodic function
with absolutely summable Fourier coefficients $c_j(f)$.
Writing~$\hat{f}_n$ for the posterior mean function given $Z$ and the parameter~$\theta$,
we have:
\begin{equation}\label{eq:blup}
\hat{f}_n(x) = \sum_{j \in \Zm} \left( \frac{
\sum_{q \in j + n \Zm} c_{q}(f)
}{
\sum_{q \in j + n \Zm} \underline{c}_{q}(\theta)
} \right) \underline{c}_{j}(\theta)  e^{2\pi i x j} \quad \mathrm{for} \ x \in [0, \, 1].
\end{equation}
The convergence of~\eqref{eq:blup} holds uniformly absolutely.
\end{prop}
The proof is deferred to Appendix~\ref{sec:proofs_framework}.

The expression~\eqref{eq:blup} shows how the posterior mean function
approximates~$f$: it transforms the Fourier coefficients
of $k_{\theta}$ into those of $f$ using the ratio of their discrete Fourier transforms.
Finally, we also define the integrated squared error:
\begin{equation}\label{eq:ISE}
\mathrm{ISE}_n \left( \nu, \, \alpha; \, f \right) = \int_0^1 \left( f -  \hat{f}_n \right)^2.
\end{equation}
Note that it does not depend on $\phi$.

\subsection{Maximum likelihood estimation}\label{sec:def_mle}

Given the observations $Z$ and $\Theta \subset \left(0, \, +\infty\right)^3$,
a maximum likelihood estimate is defined by
$\hat{\theta}_n = ( \hat{\nu}_n, \hat{\phi}_n, \hat{\alpha}_n )$
minimizing (a linear transform of) the negative log-likelihood:
\begin{equation}\label{eq:ALL}
\Lset_n\colon \theta \in \Theta \mapsto
n^{-1} \left( \ln \left( \det \left(  K_{\theta} \right) \right) + Z\tr K_{\theta}^{-1} Z \right),
\end{equation} 
with ties broken arbitrarily and $K_{\theta}$ the covariance matrix of $Z$ according to $k_{\theta}$.

The estimators $\hat{\nu}_n$ and $\hat{\alpha}_n$
are assumed bounded in this work, i.e.,
we take $\Theta = N \times \left(0, \, +\infty\right) \times A$ with
$N$ and $A$ compact intervals.
However, keeping~$\hat{\phi}_n$ unbounded is key to our main results and for discussing
the deterministic case in Section~\ref{sec:freq}.
Write $K_{\theta} = \phi R_{\nu, \alpha}$ 
for $\theta = \left(\nu, \, \phi, \alpha \right) \in \left(0, \, + \infty \right)^3$.
The following proposition gives an expression for the \emph{profile}
likelihood, i.e., the infimum of $\Lset_n(\nu, \, \phi, \, \alpha)$
with respect to~$\phi \in \left(0, \, + \infty\right)$ for fixed $\nu$ and $\alpha$.
\begin{prop}{\citep[see, e.g.,][Section 3.3.2]{Santner}}\label{prop:profile_lik}
Let $\nu, \alpha > 0$.
It holds that
\begin{equation}\label{eq:profile_lik}
\inf_{\phi > 0} \Lset_n(\nu, \, \phi, \, \alpha) = 1 + n^{-1} \ln(\det(R_{\nu, \alpha})) + \ln \left( \frac{Z\tr R_{\nu, \alpha}^{-1} Z}{n} \right).
\end{equation}
Moreover, if $Z$ is nonzero, then
the infimum is uniquely reached by $\hat{\phi}_n = Z\tr R_{\nu, \alpha}^{-1} Z/n$.
\end{prop}
(The case $Z = 0$ is covered since both sides of~\eqref{eq:profile_lik}
match.)

\section{Studying the profile likelihood using discrete Fourier transforms}\label{sec:prior_works}

\subsection{Linking the spectra of $k_\theta$ and $K_\theta$}

Consider temporarily the more general case of a strictly positive probability measure~$Q$
on a compact metric space~$\XX$
and a continuous kernel~$k \colon \mathbb{X} \times \mathbb{X} \to \mathbb{R}$ with
a constant diagonal~$\{ k(x, x), \, x \in \XX\}$.
This covers commonly used stationary kernels on tori and compact subsets of~$\mathbb{R}^d$.
Mercer's theorem \citep[see, e.g.,][]{steinwart2012mercer}
ensures the existence of a sequence of real eigenvalues~$\mu_0 \geq \mu_1 \geq \dots > 0$
and an $L^2(Q)$-orthonormal sequence~$(\phi_m)_m$ of eigenfunctions for the integral
operator~$T_k \colon g \in L^2(Q) \mapsto \int_{\XX} k(\cdot, x) g(x) \ddiff Q(x)$.
It also holds that:
\begin{equation}\label{eq:mercer_general}
k(x, y) = \sum_{m = 0}^{+ \infty} \mu_m \phi_m(x) \phi_m^{\ast}(y),
\quad
\mathrm{for} \, x, y \in \XX.
\end{equation}

Let~$x_1, \dots, x_n \in \XX$ and write~$\mu_{0, n} \geq \dots \geq \mu_{n - 1, n}$
for the eigenvalues of the covariance matrix~$K$ according to~$k$.
Since~$k$ has a constant diagonal, it holds that
\begin{equation}\label{eq:general_eigen_relation}
\sum_{m=0}^{n-1} \frac{\mu_{m, n}}{n}
= \frac{\mathrm{Tr}(K)}{n}
= \int_{\XX} k(x, x) \ddiff Q(x)
= \sum_{m = 0}^{+ \infty} \mu_m.
\end{equation}
Writing~$Q_n$ for the empirical measure obtained from~$x_1, \dots, x_n$,
the matrix~$K$ can be identified with the covariance
operator~$T_n \colon g \mapsto \int_{\XX} k(\cdot, x) g(x) \ddiff Q_n(x)$.
Therefore,
if~$Q_n $ approximates~$Q$ in some sense,
then~$T_n$ approximates~$T_k$,
and thus the normalized
eigenvalue~$\mu_{m, n}/n$
approximates~$\mu_m$,
at least for~$m$ not too large
\citep[see][for the case of samples from~$Q$]{2000:_kolt_gine_random_matrix, braun2006accurate}.
The equality~\eqref{eq:general_eigen_relation} suggests that~$\mu_{m, n}/n$ is biased upwards
by the mixing of all the~$\mu_p$.

Return to the framework and the notations from Section~\ref{sec:framework}.
In this case, the kernel~$k_\theta$ can be expanded as~\eqref{eq:mercer_general},
with~$\phi_m(x) = e^{2\pi i x m}$ and~$\mu_m = \underline{c}_m(\theta)$.
As \citet{Craven1978SMOOTHINGND}
and \citet[Section 6.7]{stein1999interpolation} point out,
the framework introduced in Section~\ref{sec:framework}
is convenient.
More precisely, it provides a natural link between~$K_{\theta}$ and~$k_{\theta}$ using discrete Fourier transforms
(see Appendix~\ref{sec:circulant} for details).
In particular, it gives a closed-form identity
\begin{equation}\label{eq:link_eigen_periodic}
n^{-1} \phi \lambda_{m, n} = \sum_{j \in \Zm} \underline{c}_{m + n j}(\theta)
\end{equation}
linking the
eigenvalues~$\phi \lambda_{0, n} , \dots, \phi \lambda_{n-1, n}$ of~$K_\theta$ to those
of~$k_\theta$. %
Furthermore, the matrices~$K_\theta$ share the same eigenvectors.
Considering~\eqref{eq:general_eigen_relation},
the equality~\eqref{eq:link_eigen_periodic}
shows how the eigenvalues of~$k_{\theta}$ are combined to obtain those of~$K_{\theta}$. %
It holds that $n^{-1} \phi \lambda_{m, n} \to \underline{c}_m (\theta)$ for a fixed~$m$
but the ratio~$n^{-1} \phi \lambda_{m, n} / \underline{c}_m (\theta)$ remains bounded away from
one for~$m$ close to~$n/2$. %
It will turn out that analyzing this ratio makes it possible to study the profile likelihood. %

\subsection{The consistency of $\hat{\nu}_n$ for fixed $\phi$ and $\alpha$}\label{sec:chen}

Assuming observations from $\xi \sim \mathrm{GP} \left(0, \, k_{\theta_0} \right)$
under a similar model
with $\theta_0 = \left( \nu_0, \, \phi_0, \, \alpha_0\right) \in \left(0, \, +\infty\right)^3$,
\citet{chen2021:_consistency} show the consistency of $\hat{\nu}_n$
for equispaced observations
on the $d$-dimensional torus for fixed parameters~$\phi$ and~$\alpha$. A sketch of
their reasoning for $d = 1$ is provided in this paragraph.
The spectrum of $K_\theta$ is studied by showing
that
\begin{equation}\label{eq:rough_bound_lambda}
n^{-1} \phi \lambda_{m, n} = e^{\Ocal(1)} \underline{c}_m (\theta) = e^{\Ocal(1)} m^{-2\nu - 1}
\end{equation}
uniformly in $\nu$ and $1 \leq m \leq n/2$.\footnote{The $\lambda_{m, n}$ satisfy
$\lambda_{m, n} = \lambda_{n - m, n}$.}
This approximation yields:
\begin{equation}\label{eq:dev_norm}
    \left\{
      \begin{array}{ll}
      \ln \left( \det \left( K_{\theta} \right) \right)  & \, = \, -2\nu n \ln(n) + n \ln (\phi ) + n\Ocal(1), \\[5pt]
        Z\tr K_{\theta}^{-1} Z & \,  = \,  \phi^{-1} \phi_0 \ \Ocal_{\PP} \left( \ln(n) \right) \ \mathrm{if} \ \nu \leq \nu_0 - 1/2, \\[5pt]
        Z\tr K_{\theta}^{-1} Z & \,  = \,  \phi^{-1} \phi_0 \ e^{\Ocal_{\PP}(1)} n^{1 + 2(\nu - \nu_0)} \ \mathrm{if} \ \nu > \nu_0 - 1/2, \\
      \end{array}
      \right.
\end{equation}
with uniform $\Ocal_{\PP}$-terms on some regularity ranges.
The consistency for fixed parameters~$\phi$ and~$\alpha$
follows by observing that~$\nu_0$ is the turning point
where the quadratic form starts dominating the log-determinant.
The latter claims are also true if~$\nu$ is estimated jointly with $\phi \in F$ for a set~$F$
bounded away
from zero and infinity.

\begin{remark}\label{remark:rem_korte}
Recently, \citet{korte2023smoothness} considered a similar model in the more general case
of quasi-uniform observations on closed Riemannian manifolds and gave a consistency result for~$\hat{\nu}_n$ with fixed~$\alpha$
and~$\phi$. They use different arguments to prove
(sufficient results slightly weaker than)~\eqref{eq:dev_norm} without establishing~\eqref{eq:rough_bound_lambda}.
In particular, bounds with matching rates for conditional
variance are used to analyze the log-determinant.
\end{remark}

\subsection{Profiling the likelihood}\label{sec:contrib}

Consider now the case $F = \left(0, \, +\infty\right)$ by 
plugging~\eqref{eq:dev_norm} into~\eqref{eq:profile_lik}, for~$\nu > \nu_0 - 1/2$, to get
\begin{equation}\label{eq:profiled_computation}
\inf_{\phi > 0} \Lset_n(\nu, \, \phi, \, \alpha) = \Ocal_{\PP}(1) - 2 \nu_0 \ln \left( n \right),
\end{equation}
which is not sharp enough.
Therefore, a more precise analysis of how the
spectrum of~$K_{\theta}$ fluctuates around the one of~$k_{\theta}$
is needed to study the profile likelihood.
The following section provides an ingredient for this purpose.
Coordination with tools for proving uniform central limit theorems makes
it possible to study convergence rates for parameter estimation and
prediction error in Section~\ref{sec:results}.
Developments for studying the profile likelihood are used
to provide insights on model selection in the case of a fixed deterministic function
from a Sobolev space in Section~\ref{sec:freq},
which also discusses related works in this setting.

\subsection{A symmetrized version of the Hurwitz zeta function}

\citet[Section 6.7]{stein1999interpolation} uses the function
$$
\gamma\colon \left(s; \, x\right) \in \left(1, \, +\infty\right)  \times \left(0, \, 1\right) \mapsto \sum_{j \in \Zm} \frac{1}{ \abs{j + x}^{s}},
$$%
for deriving the asymptotics of the Fisher information matrix of the model presented in Section~\ref{sec:framework}.
It will also play a major role in our analysis of the likelihood criterion.

The function~$\gamma$ is (jointly) smooth and related to the Hurwitz zeta function~$\zeta_H$ by:
\begin{equation}\label{eq:hurwitz_zeta_link}
\gamma\left(s; \, x\right)  = \zeta_H(s; \, x) + \zeta_H(s; \, 1 - x),
\quad \left(s, \, x\right) \in \left(1, \, +\infty\right) \times \left(0, \, 1\right).
\end{equation}
Moreover, the function $\gamma \left(s; \cdot \right)$ is symmetric with respect
to~$1/2$ for $s > 1$.

\section{Main results}\label{sec:results}

\subsection{Standing assumptions}\label{sec:assumptions}

Consider the framework presented in Section~\ref{sec:framework}
and suppose that the observations are sampled from a
Gaussian process~$\xi \sim \mathrm{GP} \left(0, \, k_{\theta_0} \right)$,
for a parameter~$\theta_0 = (\nu_0, \phi_0, \alpha_0) \in (0, +\infty)^3$.
The Fourier series representation of~$k$ yields the Karhunen-Loève expansion
\begin{equation}\label{eq:xi}
\xi\left( x \right) =
\frac{1}{\sqrt{2}} \lim_{J \to + \infty} \sum_{j = -J}^J \sqrt{\underline{c}_j(\theta_0)}\left( U_{1, \abs{j}} + i  U_{2, \abs{j}} \sign(j) \right) e^{2\pi i x j}
\end{equation}
for $x \in \left[0, \, 1\right]$,
with $( U_{q, j} )_{q \in \{ 1, 2  \}, j \geq 0}$
independent Gaussian variables such that $U_{2, 0} = 0$,
$ U_{1, 0} \sim \Ncal \left(0, 2\right)$,
and $U_{q, j} \sim \Ncal \left(0, 1\right)$ for $q \in \{ 1, 2  \}$ and $j \geq 1$.
The convergence of the expansion~\eqref{eq:xi} is meant
pointwise both in $L^2\left( \PP \right)$ and $\PP$-almost surely.
We will sometimes assume $\nu_0 > 1/2$
to ensure unconditional convergence.

Let $\hat{\theta}_n = ( \hat{\nu}_n, \hat{\phi}_n, \hat{\alpha}_n )$ be a maximum
likelihood estimate as defined in Section~\ref{sec:def_mle} for some
$\Theta = N \times \left(0, \, +\infty\right) \times A$ with
$A, N \subset \left(0, \, +\infty\right)$
compact intervals and $\nu_0 \in N$.
The following sections give convergence rates for parameter estimation and prediction error.

\subsection{Convergence rates of maximum likelihood estimation}

The following result states the strong consistency of~$\hat{\nu}_n$.
\begin{theorem}\label{thm:consistency_nu}
Let $\Theta = N \times \left(0, \, +\infty\right) \times A$ with
$N$ and $A$ compact intervals and $\nu_0 \in N$.
Then, the convergence $\hat{\nu}_n \to \nu_0$ holds almost surely.
\end{theorem}
The proof is deferred to Appendix~\ref{sec:proof_consistency}.
A key step is to show
that (a shift of) the profile likelihood converges almost surely to
\begin{equation}\label{eq:lik_crit_main}
\int_0^1 \ln \left( \gamma \left(2\nu + 1; \cdot \right)  \right)
+  \ln \left( \int_0^1
\frac{
\gamma \left(2\nu_0 + 1; \cdot \right)
}{
\gamma \left(2\nu + 1; \cdot \right)
}
 \right),
\end{equation}
for $\nu > \nu_0 - 1/2$. The first term is a refinement of the~$\Ocal(1)$
appearing in~\eqref{eq:dev_norm} for the log-determinant.
The second term is a refinement of the~$\Ocal_{\PP}(1)$ appearing for the quadratic form.
Jensen inequality shows that~\eqref{eq:lik_crit_main}
is minimized by taking $\nu = \nu_0$.

Furthermore, similarly to \citet[Section 6.7]{stein1999interpolation}, let us define
\begin{equation}\label{eq:def_psi}
\psi_{\nu}\colon x \in \left(0, \, 1\right) \mapsto
\frac{
 \sum_{j \in \Zm} \left| x + j \right|^{-2\nu - 1} \ln \left| x + j \right| 
}{
 \sum_{j \in \Zm} \left| x + j \right|^{-2\nu - 1}
},
\quad 
\mathrm{for} \ \nu > 0,
\end{equation}
which is square integrable on $\left(0, \, 1\right)$.
The following result proves
the conjecture made by \citet[p. 194]{stein1999interpolation} when $d=1$ and $\hat{\nu}_n$
and~$\hat{\alpha}_n$ are bounded. The proof is deferred to Appendix~\ref{sec:proof_asympt_norm}.
\begin{theorem}\label{thm:asympt_norm}
Let $\Theta = N \times \left(0, \, +\infty\right) \times A$ with
$N$ and $A$ compact intervals and $\nu_0 \in N$.
Then, we have the following convergence in distribution
\begin{equation*}
\sqrt{2n}
\begin{pmatrix}
 \frac{\hat{\phi}_n - \phi_0}{2\phi_0} - \left( \ln(n) + \EE \left( \psi_{\nu_0}( V ) \right)\right) \left( \hat{\nu}_n - \nu_0 \right) \\
 \sqrt{\var \left( \psi_{\nu_0}(V) \right)} \left( \hat{\nu}_n - \nu_0 \right)  \\
\end{pmatrix}  \leadsto \Ncal \left(0, \, I_2\right),
\end{equation*}%
where $V$ is a random variable distributed uniformly on $\left(0, \, 1\right)$.
\end{theorem}%
Observe that the asymptotic behavior of~$( \hat{\nu}_n, \, \hat{\phi}_n)$
is not influenced by whether the parameter~$\alpha$ is fixed, estimated, or even known. 

\subsection{Convergence rates of the integrated squared error}

This section states our results about the expectation of~\eqref{eq:ISE}
with fixed and estimated parameters.
The proofs are deferred to~\secproofserrorsuppmat.
We begin with the case of fixed parameters.

For $\nu, \nu_0 > 0$ and $x \in \left(0, \, 1\right)$, define
$$
\vartheta_{\nu; \nu_0}( x )
=
\frac{
\gamma \left(4\nu + 2; x\right) \gamma \left( 2\nu_0 + 1; x \right)
}{
\gamma^2 \left( 2\nu + 1; x \right)
}
+ 
\gamma \left(2 \nu_0 + 1; x \right)
- 2 \frac{
\gamma \left(2\nu + 2 \nu_0 + 2; x\right)
}{
\gamma \left(2\nu + 1; x \right)
}
$$
which is smooth and integrable when $\nu > (\nu_0 - 1)/2$. %
In this case, the notation~$\mathcal{C}_{\nu_0}(\nu) = \int_0^1 \vartheta_{\nu; \nu_0}$ will be used. %

The following result states the asymptotics of the prediction error
with fixed parameters.
\begin{theorem}\label{thm:error_nu_fixed}
Let $\left(\nu, \, \alpha\right) \in \left(0, \, +\infty\right)^2$ and~$\nu_0 > 1/2$.
Then,
\begin{equation*}
\EE \left( \mathrm{ISE}_n \left(\nu, \, \alpha; \, \xi\right) \right) \lesssim \frac{1}{n^{4 \nu + 2}}, \ \mathrm{for} \ \nu < (\nu_0 - 1)/2,
\end{equation*}
\
\begin{equation*}
\EE \left( \mathrm{ISE}_n \left(\nu, \, \alpha; \, \xi\right) \right) \lesssim \frac{\ln(n)}{n^{2\nu_0}}, \ \mathrm{for} \ \nu = (\nu_0 - 1)/2,
\end{equation*}
and
\begin{equation*}
n^{2\nu_0} \EE \left( \mathrm{ISE}_n \left(\nu, \, \alpha; \, \xi\right) \right) \to \phi_0 \mathcal{C}_{\nu_0}(\nu), \ \mathrm{otherwise}.
\end{equation*}
The symbol $\lesssim$ denotes an inequality up to a universal constant.
\end{theorem}
This result shows that half of the smoothness is sufficient for optimal convergence rates.
However, the constant $\mathcal{C}_{\nu_0}(\nu)$ is minimized by taking $\nu = \nu_0$, as illustrated in Figure~\ref{fig:c_curve}. %
These observations are in line with the results of \citet[][Theorem 3]{stein1999interpolation} and \citet[][Corollary~5.1]{kirchner2022necessary}. %
See \secerrorfixedloc for a corresponding result on prediction error at a fixed location. %

\begin{figure}
\centering
\includegraphics[scale=1.0]{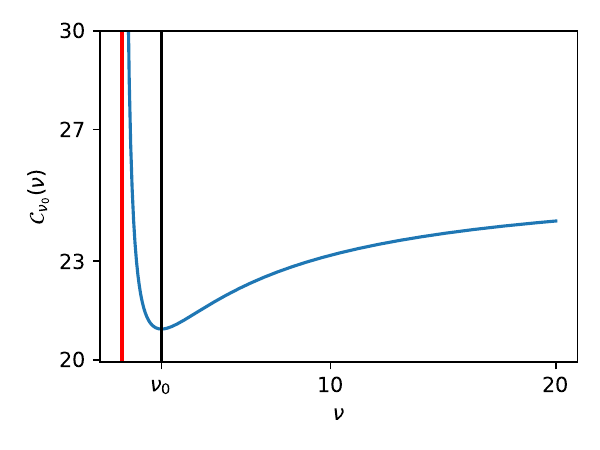}
\caption{%
Blue curve: numerical approximation of the function~$\nu \mapsto \mathcal{C}_{\nu_0}(\nu)$, for~$\nu_0 = 5/2$. %
Red vertical line:~$(\nu_0 - 1)/2$.} %
\label{fig:c_curve}
\end{figure}

Then, our last result gives the asymptotic behavior of the prediction error with estimated parameters.
\begin{theorem}\label{thm:error_nu_estimated}
Let $\nu_0 > 1/2$ and $\Theta = N \times \left(0, \, +\infty\right) \times A$ with
$N$ and $A$ compact intervals and $\nu_0 \in N$.
Then,
\begin{equation*}
n^{2\nu_0} \EE \left( \mathrm{ISE}_n \left(\hat{\nu}_n, \, \hat{\alpha}_n; \, \xi\right) \right) \to \phi_0 \mathcal{C}_{\nu_0}(\nu_0).
\end{equation*}
\end{theorem}
This last result shows that estimating the parameters
yields asymptotically the same error as if the ground truth was known.

\section{Numerical experiments}\label{sec:num_random}

The theoretical results are now illustrated by numerical experiments
in the periodic framework of Section~\ref{sec:framework}.

We pick a~$\theta_0$ and, for several values of~$n$, we sample~$\xi \sim \mathrm{GP} \left(0, \, k_{\theta_0} \right)$
to get observations~$Z = (\xi(0), \xi(1/n), \dots, \xi(1 - 1/n))\tr$.
The goal is to reconstruct~$\xi$ from~$Z$ without knowing any component of~$\theta_0$---including~$\nu_0$,
which is essential for obtaining asymptotically optimal predictions.

To assess good reconstruction of~$\xi$, we take a test grid~$(x_j)_{1 \leq j \leq N_{t}}$
on~$[0, \, 1]$ and we sample test data~$(\xi(x_1), \dots, \xi(x_{N_t}))$ jointly with~$Z$.
The test grid is sampled uniformly on~$[0, \, 1]$
so that the mean squared error
\begin{equation}\label{eq:empirical_ise}
\frac{1}{N_t} \sum_{j = 1}^{N_t} \left( \xi(x_j) - \hat{\xi}_n(x_j) \right)^2
\end{equation}
approximates~\eqref{eq:ISE}.

We consider several values of~$n$ between~$10$ and~$1000$ and~$N_t = 5000$.
The experiments are carried out for~$N_{\mathrm{rep}} = 400$ random repetitions.
(The series expansions of the covariance functions are approximated with~$10^5$ components,
and discrete Fourier transforms are used to speed up computations.)

The data are sampled with~$\theta_0 = (5/2, 10^6, 1)$.
For simplicity, we assume the parameter~$\alpha$ is fixed to a value~$\alpha_1$.
Preliminary experiments (not shown for brevity) have revealed that, as already observed
by \citet[Section~7]{korte2023smoothness}, different orders of magnitude between~$\alpha_1$ and~$\alpha_0$
result in finite-sample biases that can be slow to vanish. Consequently, the fixed value~$\alpha_1 = 1/2$
is used to bring out the efficiency of the estimators more clearly.
The smoothness parameter space~$N = [10^{-1}, 10]$ is used
and we consider~$\Theta = N \times (0, +\infty) \times \{ \alpha_1 \}$, i.e., the \emph{profile likelihood},
and several variants
of~$\Theta = N \times \{ \phi_1 \} \times \{ \alpha_1 \}$ with different values for~$\phi_1$.
The latter will be called likelihood hereafter, with specific mention of the value of~$\phi_1$.

Table~\ref{tab:nu} presents statistics of smoothness estimates. Observe that different orders
of magnitude between~$\phi_0$ and~$\phi_1$ lead to substantial biases, even for fairly large~$n$.
A look at~\eqref{eq:dev_norm} reveals that, for finite $n$, the ratio~$\phi_1/\phi_0$ can 
modify the turning point where the
quadratic form starts dominating the log-determinant.
In contrast, the profile likelihood leads to accurate estimates without requiring  prior guessing of~$\phi_0$.
Nevertheless, the most efficient estimation of~$\nu_0$ is obtained using the likelihood with known~$\phi_0$,
which is not surprising.

Figure~\ref{fig:random_experiments} shows the corresponding mean squared errors on the test data.
First, regarding the tendency with~$n$, the only substantial losses in accuracy correspond to the likelihood
with very low ratios~$\phi_1/\phi_0$. However, normalizing by the mean squared error with known
parameters shows small but significant discrepancies in the remaining cases.
Specifically, the fastest convergences to unity are
obtained using the profile likelihood and the likelihood with known~$\phi_0$.

To summarize, these experiments illustrate the practical benefits of smoothness estimation
in obtaining asymptotically optimal predictions,
as predicted by Theorem~\ref{thm:error_nu_estimated} for the profile likelihood.
In contrast, when the parameters are fixed,
known general conditions show that optimal
prediction is only possible
if~$\nu_0$ is known \citep[see, e.g.,][]{Stein_Simple_Condition, kirchner2022necessary}.
Theorem~\ref{thm:error_nu_fixed} and Figure~\ref{fig:c_curve}
show the relationship between a wrong choice of smoothness parameter and the loss
of prediction efficiency. %

\begin{table}
\caption{Summary of the smoothness estimates.
Cells show averages of~$\hat{\nu}_n$ over the~$N_{\mathrm{rep}}$
repetitions. Standard deviations are reported in parentheses.
The second column stands for the profile likelihood, whereas the subsequent columns stand
for the likelihood (with the corresponding value of~$\phi_1$ reported in the first row).}
\label{tab:nu}
\begin{tabular}{|c|c|c|c|c|c|}
\hline
$n$ & profile & $10^{-4} \phi_0$ & $10^{-3} \phi_0$ & $10^{-2} \phi_0$ & $10^{-1} \phi_0$\\
\hline
$10$ & $1.773 \ (0.417)$ & $0.364 \ (0.413)$ & $0.382 \ (0.420)$ & $0.646 \ (0.445)$ & $1.663 \ (0.238)$ \\
\hline
$20$ & $2.091 \ (0.241)$ & $0.410 \ (0.391)$ & $0.474 \ (0.396)$ & $1.089 \ (0.207)$ & $1.872 \ (0.110)$ \\
\hline
$30$ & $2.225 \ (0.189)$ & $0.431 \ (0.382)$ & $0.560 \ (0.380)$ & $1.271 \ (0.143)$ & $1.959 \ (0.076)$ \\
\hline
$40$ & $2.294 \ (0.152)$ & $0.402 \ (0.358)$ & $0.618 \ (0.342)$ & $1.382 \ (0.108)$ & $2.008 \ (0.058)$ \\
\hline
$50$ & $2.319 \ (0.130)$ & $0.442 \ (0.366)$ & $0.743 \ (0.297)$ & $1.478 \ (0.087)$ & $2.049 \ (0.045)$ \\
\hline
$100$ & $2.413 \ (0.084)$ & $0.519 \ (0.348)$ & $1.032 \ (0.182)$ & $1.674 \ (0.054)$ & $2.130 \ (0.027)$ \\
\hline
$200$ & $2.448 \ (0.057)$ & $0.643 \ (0.316)$ & $1.311 \ (0.098)$ & $1.828 \ (0.028)$ & $2.192 \ (0.015)$ \\
\hline
$300$ & $2.463 \ (0.048)$ & $0.782 \ (0.259)$ & $1.444 \ (0.070)$ & $1.894 \ (0.021)$ & $2.219 \ (0.012)$ \\
\hline
$400$ & $2.469 \ (0.040)$ & $0.873 \ (0.205)$ & $1.525 \ (0.055)$ & $1.937 \ (0.015)$ & $2.237 \ (0.009)$ \\
\hline
$500$ & $2.479 \ (0.035)$ & $0.949 \ (0.195)$ & $1.582 \ (0.043)$ & $1.965 \ (0.012)$ & $2.250 \ (0.007)$ \\
\hline
$1000$ & $2.487 \ (0.023)$ & $1.211 \ (0.111)$ & $1.728 \ (0.024)$ & $2.039 \ (0.007)$ & $2.281 \ (0.004)$ \\
\hline
\end{tabular}
\begin{tabular}{|c|c|c|c|c|c|}
\hline
$n$ & $ \phi_0$ & $10 \phi_0$ & $10^{2} \phi_0$ & $10^{3} \phi_0$ & $10^{4} \phi_0$\\
\hline
$10$ & $2.625 \ (0.232)$ & $3.501 \ (0.257)$ & $4.328 \ (0.287)$ & $5.128 \ (0.318)$ & $5.910 \ (0.347)$ \\
\hline
$20$ & $2.538 \ (0.102)$ & $3.137 \ (0.110)$ & $3.703 \ (0.120)$ & $4.254 \ (0.130)$ & $4.795 \ (0.139)$ \\
\hline
$30$ & $2.523 \ (0.070)$ & $3.031 \ (0.074)$ & $3.512 \ (0.079)$ & $3.980 \ (0.084)$ & $4.439 \ (0.089)$ \\
\hline
$40$ & $2.514 \ (0.050)$ & $2.970 \ (0.053)$ & $3.404 \ (0.057)$ & $3.825 \ (0.062)$ & $4.238 \ (0.067)$ \\
\hline
$50$ & $2.513 \ (0.042)$ & $2.934 \ (0.045)$ & $3.336 \ (0.049)$ & $3.727 \ (0.053)$ & $4.112 \ (0.057)$ \\
\hline
$100$ & $2.504 \ (0.024)$ & $2.848 \ (0.025)$ & $3.177 \ (0.027)$ & $3.497 \ (0.028)$ & $3.813 \ (0.030)$ \\
\hline
$200$ & $2.501 \ (0.013)$ & $2.789 \ (0.014)$ & $3.066 \ (0.015)$ & $3.337 \ (0.016)$ & $3.605 \ (0.017)$ \\
\hline
$300$ & $2.500 \ (0.011)$ & $2.763 \ (0.011)$ & $3.016 \ (0.012)$ & $3.265 \ (0.012)$ & $3.510 \ (0.013)$ \\
\hline
$400$ & $2.500 \ (0.008)$ & $2.748 \ (0.009)$ & $2.987 \ (0.009)$ & $3.222 \ (0.010)$ & $3.453 \ (0.010)$ \\
\hline
$500$ & $2.501 \ (0.007)$ & $2.738 \ (0.007)$ & $2.967 \ (0.007)$ & $3.192 \ (0.008)$ & $3.414 \ (0.008)$ \\
\hline
$1000$ & $2.500 \ (0.004)$ & $2.709 \ (0.004)$ & $2.911 \ (0.005)$ & $3.110 \ (0.005)$ & $3.307 \ (0.005)$ \\
\hline
\end{tabular}
\end{table}

\begin{figure}
\centering
\includegraphics[scale=1.0]{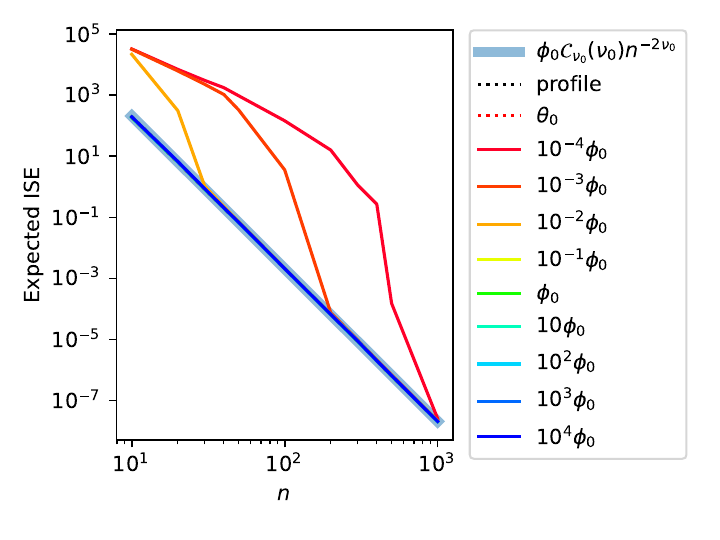}
\centering
\includegraphics[scale=1.0]{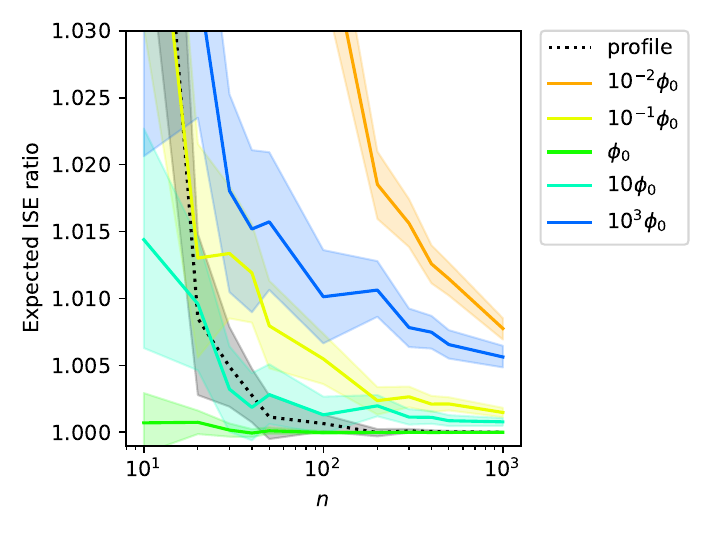}
\caption{The upper plot shows the
evolution of averages of~\eqref{eq:empirical_ise}.
The black dotted line stands for the profile likelihood, and the other thin solid colored lines
for the likelihood with the value of~$\phi_1$ indicated in the legend.
The red dotted line stands for~\eqref{eq:empirical_ise} with known~$\theta_0$.
The thick transparent line stands for the asymptotics
of~$\EE ( \mathrm{ISE}_n (\nu_0, \, \alpha_0; \, \xi) )$
predicted by Theorem~\ref{thm:error_nu_fixed}.
All but three of the thin lines mentioned above closely follow this thick line.
The lower plot shows means and~$95\%$ bootstrap confidence intervals of ratios
between averages of~\eqref{eq:empirical_ise} with
estimated and known parameters.
Results are reported for fewer values of~$\phi_1$ for clarity.} %
\label{fig:random_experiments}
\end{figure}

\section{The deterministic case}\label{sec:freq}
Let $\beta > 1/2$ and define the Sobolev space
\begin{equation*}
H^{\beta}\left[0, \, 1\right] = \left\lbrace g \in L^2 \left[0, \, 1\right], \ \ns{g}_{H^{\beta} \left[0, \, 1\right] }^2 =  \sum_{j \in \Zm} (1 + j^2)^{\beta} \abs{c_j(g)}^2 < +\infty \right\rbrace.
\end{equation*}
The usual identification with continuous representers given by the Sobolev embedding theorem makes it possible to interpret
this space as a set of continuous periodic functions.
This section studies maximum likelihood estimation with equispaced observations~\eqref{eq:obs}
from a fixed deterministic periodic function $f\colon \left[0, \, 1\right] \to \RR$ lying in a Sobolev space.
Define the (Sobolev) smoothness
$$
\nu_0(f) = \inf \left\lbrace \beta > 1/2, \ f \notin H^{\beta} \left[0, \, 1\right]  \right\rbrace
$$%
of~$f$ as \citet{karvonen2022:_asymptotic} and \citet{2022_wang:_rigde}.
We will assume that $\nu_0(f) \in \left(1, \, +\infty\right)$.
The restriction~$\nu_0(f) > 1$ is imposed for convenience as it
ensures that~$f$ has absolutely summable Fourier coefficients.
\secproofdeterministic contains the proofs for this section.

For quasi-uniform observations on ``nice'' bounded regions of~$\RR^d$, \citet{karvonen2022:_asymptotic}
shows that $\liminf \hat{\nu}_n \geq \nu_0(f)$ if~$\alpha$ and~$\phi$ are fixed.
Karvonen also shows that~$\hat{\nu}_n \to \nu_0(f)$ for a class of compactly supported
self-similar functions.
It is not hard to check that~$\nu_0(\xi) = \nu$ holds almost surely for any Matérn process with regularity parameter~$\nu$.
With that in mind, one can interpret the previous results the following way.
Maximum likelihood estimation chooses the parameter~$\nu$
so that the sample paths are asymptotically smoother than~$f$ and,
under more assumptions, so that the (Sobolev) smoothnesses match.
Interestingly, the proof is based on results similar to~\eqref{eq:dev_norm},
established by proceeding like~\citet{korte2023smoothness}.
A sketch is briefly provided with the notations of the framework from Section~\ref{sec:framework}.
The log-determinant is
studied using bounds with matching rates for conditional variance.
Then, for $\nu > \nu_0(f) - 1/2$, the uniform inequality
\begin{equation}\label{eq:approx_norm_det}
Z\tr K_{\theta}^{-1} Z \lesssim  \phi^{-1} n^{1 + 2(\nu - \nu_0(f))}%
\end{equation}%
is
(essentially)
shown. However, establishing (sufficient results slightly weaker than)
the reverse inequality requires additional assumptions
on~$f$, such as membership in a class of functions with self-similar spectra
(see \citealt[Definition~3.1]{karvonen2022:_asymptotic} and also \citealt[p.~1398]{szabo2015frequentist},
in the context of the inverse signal-in-white-noise model).
For the present purposes, it suffices to consider the ``prototypical'' subclass 
\citep[see][p. 14]{karvonen2022:_asymptotic} of functions~$f$ such that
\begin{equation}\label{eq:spectrum_sandwich}
C_1 \left| j \right|^{-\nu_0(f) - 1/2} \leq \left| c_j(f) \right| \leq C_2 \left| j \right|^{-\nu_0(f) - 1/2} \quad \mathrm{when} \ \abs{j} \geq N,
\end{equation}%
for some $N \geq 0$ and $C_2 \geq C_1 > 0$.
This notation is compatible with the definition of~$\nu_0(f)$.
For instance, the periodic function~$g$ which is symmetric with respect to zero and such that~$g(x) = 4 \pi^2 x^2$, for~$x \in [0, \, 1/2]$, has Fourier coefficients
\begin{equation}
  \left\{\begin{array}{rl}
    c_j(g) & = 2 (-1)^{\lvert j \lvert} j^{-2} \ \mathrm{for \ non\mbox{-}zero} \ j,\\
    c_0(g) & = \pi^2/3,
  \end{array}\right.\,
\end{equation}
and therefore satisfies~\eqref{eq:spectrum_sandwich} with~$\nu_0(f) = 3/2$. See also the function in Figure~\ref{fig:exp_deter} for another example.

As in previous works, the following property holds for the class~\eqref{eq:spectrum_sandwich} of functions with well-behaved spectra.
\begin{prop}\label{prop:convergence_det_nll}
Let $\Theta = N \times F \times A$ with
$N$, $F$, and $A$ compact intervals and~$N$ containing~$\nu_0(f) \in (1, +\infty)$.
Assume that~$f$ satisfies~\eqref{eq:spectrum_sandwich}.
Then, the convergence $\hat{\nu}_n \to \nu_0(f)$ holds.
\end{prop}
Having~$\phi$ and~$\alpha$ estimated on compact intervals jointly with~$\nu$ is somewhat anecdotal,
so nothing is new in this result. The details of the proof sketched in the previous paragraph
are therefore omitted.
However, since the proof roughly follows the lines from Section~\ref{sec:chen},
the observation from Section~\ref{sec:contrib} applies also in this setting.
Beforehand, the following preliminary step is required.
\begin{prop}\label{prop:excursion}
Suppose that~$\nu_0(f) \in (1, + \infty)$ and
let $\Theta = N \times \left(0, \, +\infty\right) \times A$ with
$N$ and $A$ compact intervals and $\max N \geq \nu_0(f) - 1/2$.
Then, it holds that $\liminf \hat{\nu}_n \geq \nu_0(f) - 1/2$.
\end{prop}

\begin{remark}\label{rem:escape_extension}
Inspection of the proof of Proposition~\ref{prop:excursion}
reveals that it is also valid for quasi-uniform observations on
nice bounded open regions of~$\mathbb{R}^d$,
using tools from \citet[Proposition~3.6--3.7]{karvonen2022:_asymptotic}.
The resulting asymptotic lower bound is~$\nu_0(f) - d/2$.
\end{remark}

Note the difference with the previous~$\liminf \hat{\nu}_n \geq \nu_0(f)$ for fixed~$\phi$.
A smoothness estimate larger than~$\nu_0(f)$ means that~$f$ is
rougher than the sample paths. The weaker inequality~$\hat{\nu}_n \geq \nu_0(f) - 1/2$ only
means that the function~$f$ is rougher than the elements of the reproducing kernel Hilbert space.
A~computation similar to~\eqref{eq:profiled_computation} shows that the behavior
above~$\nu_0(f) - 1/2$ is, roughly speaking, governed by~$\Ocal(1)$-terms.
It is possible to give a quantitative description of what happens for a class smaller than~\eqref{eq:spectrum_sandwich}.
For~$\nu > \nu_0(f) - 1/2$ and~$\alpha > 0$, define
$$
\Mset_n^f \left(\nu, \alpha \right) =
\inf_{\phi > 0} \Lset_n \left( \nu, \phi, \alpha \right)+ 2\nu_0(f) \ln(n) - 1.
$$
\begin{prop}\label{prop:dev_nll_misp}
Suppose that $\nu_0(f) \in (1, +\infty)$ and
\begin{equation}\label{eq:small_class}
c_j(f) = \left(1 + \Ocal ( \left| j \right|^{-1} ) \right) \left| j \right|^{-\nu_0(f) - 1/2}
\end{equation}
for nonzero~$j$.
Then, we have~$\Mset_n^f \left(\nu, \, \alpha \right) \to \Mset_{\infty}^f \left(\nu\right)$
uniformly on compact subsets of $\left(\nu_0(f) - 1/2, \, + \infty \right) \times \left(0, \, + \infty \right)$, where
\begin{equation}\label{eq:dev_nll_misp}
\Mset_{\infty}^f \left(\nu\right) =
\int_0^1 \ln \left( \gamma \left(2 \nu  + 1; \cdot \right) \right)
+ \ln \left(
\int_0^1 \frac{ \gamma^2 \left(\nu_0(f) + 1/2; \cdot\right) }{\gamma \left(2\nu + 1; \cdot\right) } \right).
\end{equation}
\end{prop}
We could not identify the minimizer(s) of the limit analytically.
Figure~\ref{fig:criterion} shows a numerical approximation of $\Mset_{\infty}^f$.

After inspection of the proof of Proposition~\ref{prop:dev_nll_misp}, it does not
seem obvious to exhibit a function~$f$ such that~$\hat{\nu}_n \to \nu_0(f)$ holds when the amplitude parameter~$\phi \in \left(0, \, +\infty\right)$
is jointly estimated. However, Theorem~\ref{thm:consistency_nu} shows that the set of such functions
has probability one under a Matérn process with regularity~$\nu_0$ belonging to~$N$.

\begin{figure}
\centering
\includegraphics[scale=0.90]{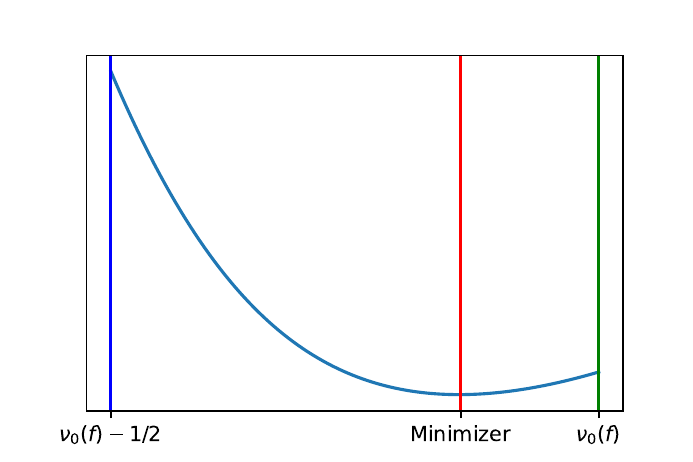}
\caption{
Numerical approximation of the function $\Mset_{\infty}^f$, for $\nu_0(f) = 3/2$.
A numerical approximation
of the minimizer is about
$1.354$.}
\label{fig:criterion}
\end{figure}

The previous results are illustrated numerically.
We consider a periodic function~$f = c B_{2m} + g$, where~$B_{2m}$
is a Bernoulli polynomial, $c$ is a constant, $m$ is an integer, and~$g$
is a band-limited periodic function.
The Fourier coefficients of Bernoulli polynomials satisfy~$c_j(B_{2m}) \propto j^{-2m}$,
for nonzero~$j$ \citep{abramowitz1968handbook}.
Thus, we have~$\nu_0(f) = 2m - 1/2$, and by choosing an appropriate value for~$c$,
the function $f$ satisfies~\eqref{eq:small_class}.
Specifically, we choose~$m = 1$ and an arbitrary function~$g$ with ten frequencies.
Figure~\ref{fig:exp_deter} illustrates the resulting~$f$, for which~$\nu_0(f) = 3/2$.

As in Section~\ref{sec:num_random}, we consider the profile likelihood and the likelihood
with several fixed values~$\phi_1$. The smoothness parameter space is
again~$N = [10^{-1}, \, 10]$ and the parameter~$\alpha$ is fixed to one.
We evaluate several sample sizes~$n$, ranging from~$10$ to~$10^6$.
To speed up computations, we use discrete Fourier transforms and
finite approximations of covariance function expansions.

The behavior of the smoothness estimates is shown in Figure~\ref{fig:exp_deter}.
The estimators seem to converge to the limits predicted by Proposition~\ref{prop:convergence_det_nll}
and Proposition~\ref{prop:dev_nll_misp}. However, convergence is quite slow (especially
for the likelihood with certain fixed values~$\phi_1$). %
Nevertheless, it appears clearly that the profile likelihood does not fit~$\nu$
so that~$f$ and the sample paths have the same Sobolev smoothness.

\begin{figure}
\centering
\includegraphics[scale=1.0]{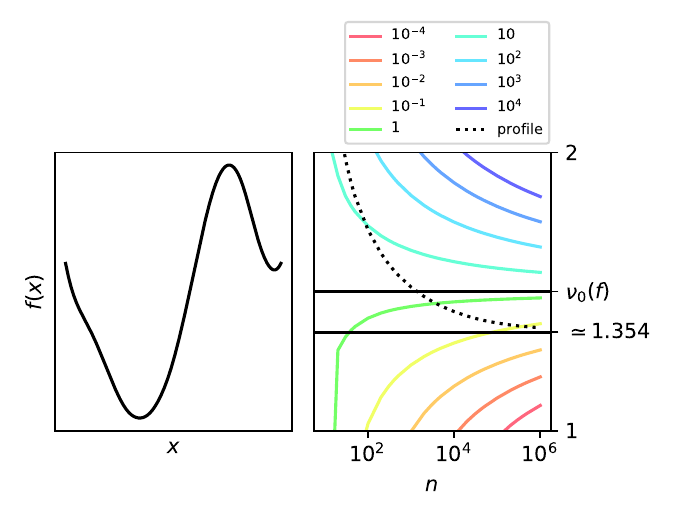}
\caption{Left: the function~$f$. Right: smoothness estimates as functions of~$n$.
The black dotted line stands for the profile likelihood, and the other solid colored lines
for the likelihood with the fixed value~$\phi_1$ indicated in the legend.
The lowest black horizontal line represents the approximate minimizer of~$\Mset_{\infty}^f$
shown in Figure~\ref{fig:criterion}.}
\label{fig:exp_deter}
\end{figure}

\section{Conclusion and perspectives}\label{sec:concl}

The joint maximum likelihood estimation of the
regularity and the amplitude parameters was studied theoretically for
a Mat\'ern model on the circle with equispaced observations. More precisely, strong consistency
and asymptotic normality results were established. As a consequence,
it was also shown that maximum likelihood estimation yields asymptotically the same
error as if the ground truth parameter was known. To our knowledge, this is the first rigorous
result of this kind for Gaussian process interpolation using a Matérn covariance function with
unknown smoothness parameter.

We also examined the scenario involving observations of a function from a continuous Sobolev space.
Typical results indicate that, under certain spectral assumptions,
maximum likelihood estimation of the regularity parameter alone
asymptotically finds the Sobolev smoothness of the target function.
It was shown that such a phenomenon does not appear when~$\nu$ is estimated
jointly with~$\phi$, i.e., using the profile likelihood.

The main limitation of these results is that they deal with a
limited applicability model. However, given the similarity
between periodic Matérn kernels and those on Euclidean domains,
it is reasonable to conjecture that similar results are also valid for the latter.

Section~\ref{sec:prior_works} shows that one proof strategy would be to
derive sharp estimates of tails of covariance matrix spectra,
which was made possible by the restrictive assumptions of the present article.
To our knowledge, no available result is sufficient to prove, for instance,
an analogous statement
of Lemma~\lemlogdet in a more applicable setting.
Multiplicative bounds for empirical approximation of spectra with
random points \citep{braun2006accurate} are a relevant approach.
However, we need to find asymptotic equivalents rather than upper bounds.

Another proof strategy would be to use scattered data approximation results for
quasi-uniform designs, as done by \citet{karvonen2022:_asymptotic},
who, for instance, studies the log-determinant using bounds with matching
rates for conditional variance.
The proof of Proposition~\ref{prop:excursion} can be adapted with these tools, as mentioned in Remark~\ref{rem:escape_extension}.
However, a refined analysis is required to adapt the other results of the present article.
For instance,
pursuing this argument to prove an analogous result of
Lemma~\lemlogdet would require finding an asymptotic equivalent of the conditional variance,
instead of bounds with matching rates.
However, the sampling inequalities 
\citep[e.g.,][]{narcowich2005:_sobolev_bounds, arcangeli2007:_error_bounds}
used to derive the upper bounds are notoriously challenging to prove.

\clearpage
\appendix
\section{Asymptotic analysis of prediction error at a fixed location}\label{sec:pred_fixed_loc}

Theorem~\thmerrornufixed considers the integrated squared prediction error on~$[0, \, 1]$
in the framework introduced in Section~\secframework, i.e.
with~$n$ equispaced observation locations~$\{j/n, \, 0 \leq j \leq n - 1\}$.
Asymptotic results on prediction error are often formulated for a fixed single location.
The closed-form expression for leave-one-out
prediction~\citep{Craven1978SMOOTHINGND, Dubrule1983CrossVO}
is a convenient way of giving a corresponding version of Theorem~\thmerrornufixed for predicting~$\xi(0)$.
\begin{theorem}\label{thm:error_nu_fixed_zero}
Let $\left(\nu, \, \alpha\right) \in \left(0, \, +\infty\right)^2$.
For~$n \geq 1$, let~$\hat{\xi}_n^{(0)}$ be the kriging predictor
given~$\{\xi(j/n), \, 1 \leq j \leq n - 1\}$ and~$(\nu, \, \alpha)$.
Then,
\begin{equation*}
\EE \left( \left( \hat{\xi}_n^{(0)}(0) - \xi(0) \right)^2 \right) \lesssim \frac{1}{n^{4 \nu + 2}}, \ \mathrm{for} \ \nu < (\nu_0 - 1)/2,
\end{equation*}
\
\begin{equation*}
\EE \left( \left( \hat{\xi}_n^{(0)}(0) - \xi(0) \right)^2 \right) \lesssim \frac{\ln(n)}{n^{2\nu_0}}, \ \mathrm{for} \ \nu = (\nu_0 - 1)/2,
\end{equation*}
and
\begin{equation*}
n^{2\nu_0} \EE \left( \left( \hat{\xi}_n^{(0)}(0) - \xi(0) \right)^2 \right) \to \phi_0 \mathcal{C}_{\nu_0}^{(0)}(\nu), \ \mathrm{otherwise},
\end{equation*}
where
$$
\mathcal{C}_{\nu_0}^{(0)}(\nu) = \frac{
\int_0^1 \gamma \left( 2\nu_0 + 1; \, \cdot \right) / \gamma^2 \left( 2\nu + 1; \, \cdot \right)
}{
\left( \int_0^1 \gamma^{-1} \left( 2\nu + 1; \, \cdot \right) \right)^2
}.
$$
\end{theorem}%
The proof is deferred to~\secproofserrorsuppmat.
Figure~\ref{fig:c_zero_curve} illustrates how~$\mathcal{C}_{\nu_0}^{(0)}(\nu)$ varies with~$\nu$. %

\begin{figure}
\centering
\includegraphics[scale=1.0]{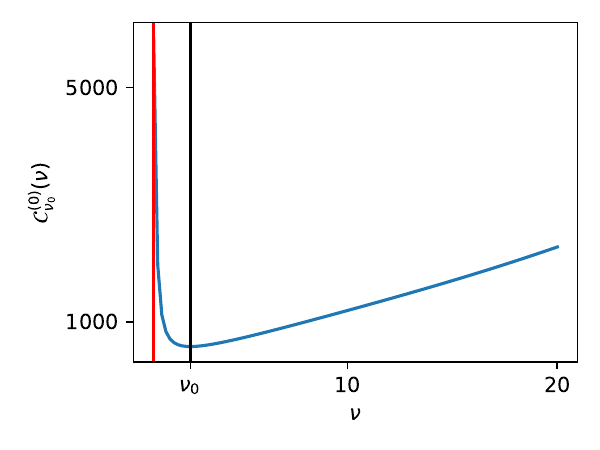}
\caption{
Blue curve: numerical approximation of the function~$\nu \mapsto \mathcal{C}_{\nu_0}^{(0)}(\nu)$, for~$\nu_0 = 5/2$. %
Red vertical line:~$(\nu_0 - 1)/2$.} %
\label{fig:c_zero_curve}
\end{figure}

\section{Proofs}\label{app:proofs}

\subsection{Notations}\label{sec:notations}

The symbol $\lesssim$ denotes an inequality up to a universal constant.
For compactness, the symbol $\approx$ is used when the two-way inequality $\lesssim$ holds.

Write $K_{\theta} = \phi R_{\nu, \alpha}$ and $\underline{c}_j(\theta) = \phi c_j (\nu, \alpha)$,
for $\theta = \left(\nu, \, \phi, \alpha \right) \in \left(0, \, + \infty \right)^3$ and $j \in \Zm$.
All results suppose that $\Theta = N \times \left(0, \, +\infty\right) \times A$
with $N = \left[\nu_{\min}, \, \nu_{\max}\right]$, $A = \left[\alpha_{\min}, \, \alpha_{\max} \right]$,
$0 < \nu_{\min} < \nu_0 < \nu_{\max} < + \infty$, and $0 < \alpha_{\min} \leq \alpha_{\max} < + \infty$ unless explicitly stated otherwise.
The proofs sometimes take care to ensure that~$\hat{\nu}_n$ is asymptotically
larger than some lower bounds. These steps can be ignored if the bound is
less than zero.
We define $N_{\epsilon} = \left[\nu_0 - 1/2 + \epsilon, \, +\infty \right) \cap N$
for~$\epsilon > 0$.
The notation $l = \floor*{(n-1) / 2}$ will often be used throughout the following.

\subsection{Circulant matrices and useful facts}\label{sec:circulant}

The framework introduced in Section~\ref{sec:framework} is convenient
for analyzing kernel-based regression methods
\cite[see, e.g.,][]{Craven1978SMOOTHINGND}.
This section reviews the properties needed for our purposes.

Let $W$ be the $n \times n$ matrix with entries $W_{j, m} = n^{-1/2} e^{ 2\pi i j m / n}$, for $0 \leq j, m \leq n-1$.
For every $\theta = (\nu, \phi, \alpha) \in (0, +\infty)^3$, the periodicity of~$k_\theta$
implies that
$$
K_\theta =
\begin{pmatrix}
k_{\theta} \left(0\right) \, & k_{\theta} \left( \frac{1}{n}\right) \, & \dots \, & k_{\theta} \left( \frac{n-1}{n}\right) \\[6pt]
k_{\theta} \left( \frac{n-1}{n}\right)  & \, k_{\theta} \left(0\right) & \, \dots & \, k_{\theta} \left( \frac{n-2}{n}\right)\\[6pt]
\dots  & \, \dots  & \, \dots & \, \dots \\[6pt]
k_{\theta} \left( \frac{1}{n}\right)  & \, k_{\theta} \left( \frac{2}{n}\right) & \, \dots & k_{\theta} \left(0\right) 
\end{pmatrix}
$$%
is a circulant matrix
and so is $R_{\nu, \alpha}$.
Consequently \citep[see, e.g.,][p. 130]{brockwell1987time}, it holds that
$R_{\nu, \alpha} = W \Delta_{\nu, \alpha} W^{\ast}$
with
$\Delta_{\nu, \alpha} = \mathrm{diag} \left( \lambda_{0, n}, \dots, \lambda_{n-1, n} \right)$
and
\begin{equation}\label{eq:eigen_values}
\lambda_{m, n} = \sum_{j = 0}^{n-1} e^{-2\pi i j m / n} k_{\nu, 1, \alpha}(j / n) = n \sum_{j \in \Zm} c_{m + n j}(\nu, \, \alpha), \ 0 \leq m \leq n - 1.
\end{equation}
Note that $\lambda_{m, n}$ depends on $\nu$ and $\alpha$
but the symbols are dropped to avoid cumbersome expressions.
These coefficients verify
\begin{equation}\label{eq:eigen_sym}
\lambda_{m, n} = \lambda_{n - m, n}, \ \mathrm{for} \ 1 \leq m \leq n - 1.
\end{equation}
The eigenvalue $\lambda_{0, n}$ is simple and there are~$l$ pairs
$\left( \lambda_{m, n} , \, \lambda_{n - m, n}\right)$,
for~$m \in \llbracket 1, \, l\rrbracket$, where~$l$ is
the shortcut defined in Appendix~\ref{sec:notations}. If $n$ is
even, then the eigenvalue~$\lambda_{n/2, n}$ is also simple.

Furthermore, combining each pair of eigenvectors of $W$
shows that
$R_{\nu, \alpha} = P \Delta_{\nu, \alpha}  P\tr$
for a unitary matrix~$P$ written using sines and cosines functions.
Then, with $\theta_0 = \left(\nu_0, \phi_0, \alpha_0\right)$ the ground truth
introduced in Section~\ref{sec:assumptions}, write
$$P\tr Z = \sqrt{\phi_0} \left(\sqrt{\lambda_{0, n}^{(0)}} U_{0, n}, \dots, \sqrt{\lambda_{n-1, n}^{(0)}} U_{n-1, n} \right),$$%
with $\lambda_{0, n}^{(0)}, \dots, \lambda_{n-1, n}^{(0)}$
the eigenvalues of $R_{\nu_0, \alpha_0}$
and $U_{0, n}, \dots, U_{n - 1, n}$ drawn independently from a standard Gaussian. We have
\begin{equation*}
Z\tr R_{\nu, \alpha}^{-1} Z
= \phi_0 \sum_{m = 0}^{n - 1} \frac{U_{m, n}^2 \lambda_{m, n}^{(0)} }{\lambda_{m, n}}.
\end{equation*}
Our strategy to analyze this kind of expression
will often consist of: 1) studying the sum
for~$m \in \llbracket 1, \, l\rrbracket$; 2) using the equality~\eqref{eq:eigen_sym};
and 3) treating the remaining terms for~$m = 0$ and possibly~$m = n/2$
separately.

The following approximation discussed in Section~\ref{sec:chen} will sometimes be used.
\begin{lem}\label{lem:rough_bound_c}
One has $n^{-1} \lambda_{0, n} \approx c_{0}(\nu, \alpha) \approx 1$ and $n^{-1} \lambda_{m, n} \approx c_{m}(\nu, \alpha) \approx m^{-2\nu - 1}$ uniformly in
$\nu \in N$, $\alpha \in A$, $n$ and $1 \leq m \leq \floor*{n / 2}$.
\end{lem}
\begin{proof}
Let $0 \leq m \leq \floor*{n / 2}$,
we have
using \eqref{eq:eigen_values}
\begin{equation*}\label{eq:bounds_eigen_log_det}
c_{m}(\nu, \alpha) \leq \lambda_{m, n} / n  \leq 2c_{m}(\nu, \alpha) + 2\sum_{j = 1}^{+\infty} c_{m + n j}(\nu, \alpha).
\end{equation*}
Moreover
\begin{equation*}
\sum_{j = 1}^{+\infty} c_{m + n j}(\nu, \alpha) / c_{m}(\nu, \alpha)
\leq \sum_{j = 1}^{+\infty} (\alpha_{\max}^2 + 1/4)^{\nu + 1/2} / j^{2\nu + 1}
\lesssim 1,
\end{equation*}
uniformly using the monotonicity of the zeta function.
This shows $n^{-1} \lambda_{m, n} \approx c_{m}(\nu, \alpha)$
and finishing the proof makes no difficulty.
\end{proof}
Nevertheless, our results will require refined approximations, as explained in Section~\ref{sec:contrib}.

\subsection{More notations and properties}\label{sec:more_notations}

For each $n$, it is straightforward to prove that the~$\lambda_{m, n}$s
are smooth functions of $\left( \nu, \, \alpha\right) \in \left(0, \, +\infty\right)^2$
by bounding the derivatives of the $c_j$s
uniformly on compacta
(up to third-order derivatives suffice for our purposes). 
Using the formulas from Appendix~\ref{sec:circulant} then shows
that $\Lset_n$ is also smooth for any realization. 

Furthermore, define:
$$
\Mset_n\colon \left(\nu, \alpha \right) \in N \times A \mapsto
\inf_{\phi > 0} \Lset_n \left( \nu, \phi, \alpha \right)+ 2\nu_0 \ln(n) - \ln \left( \phi_0 \right) - 1,
$$
with $\nu_0$ the ground truth introduced in Section~\ref{sec:assumptions}.
Its expression is given by Proposition~\ref{prop:profile_lik}
so it is a stochastic process which is smooth on the almost sure event $Z \neq 0$.
The proofs mostly consist in studying~$\Mset_n$.

For a compact interval $A \subset \left(0, \, +\infty \right)$, define now
$$\Uset_n\colon \nu \in N \mapsto \inf_{\alpha \in A} \Mset_n \left(\nu, \alpha \right).$$
The object~$\Uset_n$ is a stochastic process since the infima
can be replaced by countable ones.
Its almost sure continuity follows from the almost sure
smoothness of~$\Mset_n$ and the compacity of~$A$.

Also, write
$g_\nu = \ln \left( \gamma \left(2\nu + 1; \, \cdot\right) \right)$
for $\nu > 0$ and
$$h_{\nu;\nu_0} = \frac{\gamma\left(2\nu_0 + 1; \,  \cdot\right)}{\gamma  \left(2\nu + 1; \, \cdot\right)}$$
for $\nu > \nu_0 - 1/2$.
These functions are smooth and integrable
and we will write
$$
H\colon \nu \in \left( \nu_0 - 1/2, \, +\infty\right)  \mapsto \int_0^1 h_{\nu;\nu_0},
\quad
G\colon \nu \in \left( 0, \, +\infty\right)  \mapsto \int_0^1 g_{\nu},
$$
and $\Uset\colon  \nu \in \left( \nu_0 - 1/2, \, +\infty\right) \mapsto  G(\nu) + \ln\left( H (\nu) \right)$.
The smoothness of these functions is ensured by
dominated convergence arguments (three derivatives suffice for our purposes).

\subsection{Proofs of Section~\ref{sec:framework}}\label{sec:proofs_framework}

\begin{proof}[Proof of Proposition~\ref{prop:blup}]
For $x \in \left[0, \, 1\right]$, the kriging equations yield
$\hat{f}_n(x) = k_{\theta, \, x }\tr \, K_{\theta}^{-1} Z$,
with $k_{\theta, \, x} = \left( k_{\theta} \left(  m/n - x \right) \right)_{0 \leq m \leq n -1}$.
The assumptions guarantee that $f$
equals the limit of its Fourier series everywhere.
Then, using the matrix~$W$ defined in Appendix~\ref{sec:circulant},
it is straightforward to show that
\begin{equation}\label{eq:estimation_coef_f}
W^{\ast} Z = \sqrt{n} \left( \sum_{j \in m + n\Zm}  c_j(f) \right)_{0 \leq m \leq n - 1}
\end{equation}
and
$$
W^{\ast} k_{\theta, \, x}  = \sqrt{n} \left( \sum_{j \in m + n\Zm}  \underline{c}_j( \theta) e^{- 2 \pi i x j } \right)_{0 \leq m \leq n - 1},
$$
where the sums converge absolutely.
Then, the uniform absolute-convergence of~\eqref{eq:blup}
follows from elementary manipulations.
\end{proof}

\subsection{Proof of Theorem~\ref{thm:consistency_nu}}\label{sec:proof_consistency}

\subsubsection{Proof of the theorem}

\begin{proof}[Proof of Theorem~\ref{thm:consistency_nu}]
For $0 < \epsilon < 1/2$,
the sequence~$\Uset_n$ converges almost surely uniformly to~$\Uset$
on~$N_{\epsilon}$
by Lemma~\ref{lem:dev_pnll_full}.
Also, the function~$\Uset$
is continuous and strictly minimized by taking~$\nu = \nu_0$ thanks to Jensen inequality.

The rest of the proof is dedicated to showing that~$\liminf \hat{\nu}_n \geq \nu_0 - 1/2 + \epsilon$
for some~$\epsilon > 0$.
First for $\nu \in N$ and $\alpha \in A$, we have
\begin{align*}
\Mset_n(\nu, \alpha)
& \, = \,
G(\nu) + \mathcal{O} \left( \ln(n) / n \right) + \ln \left( \frac{\phi_0^{-1} Z\tr R_{\nu, \alpha}^{-1} Z}{n^{1 + 2(\nu - \nu_0)}} \right)\\
& \, = \,
\mathcal{O} \left(1 \right) + \ln \left( \frac{\phi_0^{-1} Z\tr R_{\nu, \alpha}^{-1} Z}{n^{1 + 2(\nu - \nu_0)}} \right)
\end{align*}
uniformly in $\nu \in N$ and $\alpha \in A$ thanks to Lemma~\ref{lem:log_det} and the continuity of~$G$.

Now, let $0 < \epsilon < 1/4$, $\nu \in N \setminus N_{\epsilon} = \left[\nu_{\min}, \, \nu_0 - 1/2 + \epsilon \right)$ and~$\alpha \in A$.
It holds that:
\begin{align*}
\frac{\phi_0^{-1} Z\tr R_{\nu, \alpha}^{-1} Z}{n^{1 + 2(\nu - \nu_0)}} 
& \, \geq \, \frac{C}{n} \sum_{m = 1}^{n-1} U_{m, n}^2 \min \left( \frac{m}{n}, 1 - \frac{m}{n} \right)^{2(\nu - \nu_0)} \\ & \qquad \qquad \left(C > 0, \ \mathrm{by \ Lemma~}\ref{lem:rough_bound_c} \ \mathrm{and} \ \eqref{eq:eigen_sym} \right) \\
& \, \geq \, \frac{C}{n} \sum_{m = 1}^{n-1} U_{m, n}^2 \min\left( \frac{m}{n}, 1 - \frac{m}{n} \right)^{- 1 + 2\epsilon} \\ & \qquad \qquad \left(\nu \leq \nu_0 - 1/2 + \epsilon\right) \\
& \, = \,
o(1) +
\frac{C}{n} \sum_{m = 1}^{n-1} \min\left( \frac{m}{n}, 1 - \frac{m}{n} \right)^{- 1 + 2\epsilon}
 \\ & \qquad \qquad \left(\mathrm{a.s., \ using \ Lemma~}\ref{lem:riemann_sum_bound}\right)\\
& \, \to \, \frac{C}{2^{2\epsilon} \epsilon}.
\end{align*}%
Lemma~\ref{lem:dev_pnll_full} gives $\Uset_n \left( \nu_0 \right) \to \Uset \left( \nu_0 \right)$ almost surely,
so we have
\begin{align*}
\inf_{\nu \in N \setminus N_{\epsilon} } \Uset_n \left( \nu \right) - \Uset_n \left( \nu_0 \right)
& \, = \,
\inf_{\nu \in N \setminus N_{\epsilon} , \alpha \in A} \Mset_n(\nu, \alpha) - \Uset_n \left( \nu_0 \right)
\\
& \, \geq \, \Ocal(1) + \ln \left( C \right) - \ln \left( 2^{2\epsilon} \epsilon \right) - \Uset\left( \nu_0 \right) + o \left( 1 \right).
\end{align*}%
Letting $\epsilon \to 0$ shows that the expression in display can be made almost surely ultimately strictly positive.
\end{proof}

\subsubsection{Approximating $\ln(\det(R_{\nu, \alpha}))$}

\begin{lem}\label{lem:bound_c_rho}
Let $\nu \in N$, $\alpha \in A$, $1 \leq m \leq \floor*{n/2}$, and $j \in \Zm$. We have:
\begin{equation}\label{eq:bound_c_h}
c_{m + n j}(\nu, \alpha) = 
\frac{1 + u_{n, m, j}(\nu, \alpha)}{ \abs{j n  + m}^{2 \nu + 1}},
\end{equation}
with $- 1 < v_m \leq u_{n, m, j}(\nu, \alpha) \leq 0$ and $v_m = \Ocal(m^{-2})$.
\end{lem}
\begin{proof}
Using~\eqref{eq:eigen_values}, we have
$$
c_{m + n j}(\nu, \alpha) = \frac{1}{(\alpha^2 + (j n + m)^2)^{\nu + 1/2}}
= \frac{1 + u_{n, m, j}(\nu, \alpha)}{ \abs{j n  + m}^{2 \nu + 1}},
$$
with $u_{n, m, j}(\nu, \alpha) = \left(1 + (\alpha/(jn + m))^2 \right)^{-\nu - 1/2} - 1$.
Elementary operations show that
$$
0 \geq u_{n, m, j}(\nu, \alpha) \geq \left(  \left( \frac{\alpha_{\max}}{m} \right)^2 + 1 \right)^{-\nu_{\max} - 1/2} - 1,
$$
which gives the desired result thanks to the Taylor inequality.
\end{proof}

\begin{lem}\label{lem:rough_gamma_behaviour}
Let $S \subset \left(1, \, + \infty \right)$ be a compact interval. It holds that
\begin{equation*}
\gamma \left( s; \, x \right) = \frac{1}{x^{s}} + \frac{1}{(1 - x)^{s}}
+ \Ocal \left( 1 \right),
\end{equation*}
uniformly in $s \in S$ and $x \in \left(0, \, 1\right)$. In particular, we have
$$\gamma \left( s; \, x \right) \approx \min \left(x, \, 1-x\right)^{-s}.$$
\end{lem}
\begin{proof}
Let $s_{\min} = \min S$. Then,
$0 \leq
\gamma \left( s; \, x \right) - x^{-s} - (1 - x)^{-s}
\leq 2 \zeta(s_{\min})$.
\end{proof}

\begin{lem}\label{lem:rough_gamma_der_behaviour}
Let $S \subset \left(1, \, + \infty \right)$ be a compact interval. It holds that
\begin{equation*}
\frac{\partial \gamma}{\partial s} \left( s; \, x \right) = - \frac{\ln(x)}{x^{s}} - \frac{\ln(1-x)}{(1 - x)^{s}}
+ \Ocal \left( 1 \right),
\end{equation*}
uniformly in $s \in S$ and $x \in \left(0, \, 1\right)$.
\end{lem}
\begin{proof}
Similar to the proof of Lemma~\ref{lem:rough_gamma_behaviour}.
\end{proof}

\begin{lem}\label{lem:log_det}
Uniformly in $\nu \in N$ and $\alpha \in A$, we have
$$
\ln(\det(R_{\nu, \alpha}))
=
- 2\nu n \ln(n)
+ n \int_{0}^{1} g_{\nu} + \mathcal{O}(\ln(n)).
$$
\end{lem}

\begin{proof}
Let $\nu \in N$ and $\alpha \in A$.
Using~\eqref{eq:eigen_values} and Lemma~\ref{lem:bound_c_rho}, we have
$$
\lambda_{m, n}/n = \sum_{j \in \Zm} c_{m + n j}(\nu, \alpha)
=  \sum_{j \in \Zm} \frac{1 + u_{n, m, j}(\nu, \alpha)}{ \abs{j n  + m}^{2 \nu + 1}}.
$$
Therefore, using the notation $l = \floor*{(n-1) / 2}$, we have
$$
\sum_{m = 1}^{l} \ln(\lambda_{m, n}/n)
=
- (2\nu + 1) l \ln(n)
+ a_n \left(\nu, \, \alpha\right)
+ \sum_{m = 1}^{l} g_{\nu} (m/n),
$$
with
$$\left| a_n\left(\nu, \, \alpha\right) \right| \leq \left|   \sum_{m = 1}^{l} \ln(1 +  v_m) \right| = \mathcal{O}(1)$$
uniformly in $\nu \in N$ and $\alpha \in A$.

The function $g_\nu$ is symmetric with respect to $1/2$.
Moreover, a direct consequence of Lemma~\ref{lem:rough_gamma_behaviour} is that
\begin{equation}\label{eq:dev_lim_g}
g_{\nu}(x) =-(2\nu + 1) \ln(x) + \Ocal \left( 1 \right),
\end{equation}
uniformly in $\nu \in N$ and $0 < x \leq 1/2$.
For~$\nu \in N$, the function
$g_{\nu}$ is thus integrable on~$\left(0, \, 1\right)$.
Furthermore,
verifying that it is non-increasing on $\left(0, \, 1/2\right]$
is straightforward using the derivative of $\gamma \left( 2\nu + 1; \, \cdot\right)$, so we have:
$$
\int_{1/n}^{(l + 1)/n} g_{\nu} \leq \frac{1}{n} \sum_{m = 1}^l g_{\nu}(m/n) \leq \int_{0}^{l/n} g_{\nu}.
$$
Use then~\eqref{eq:dev_lim_g} to get
$
\int_{0}^{1/n} g_{\nu} = \mathcal{O}(\ln(n)/n)
$,
uniformly in $\nu \in N$.
The remainders~$\int_{l/n}^{1/2} g_\nu$
and~$\int_{1/2}^{(l+1)/n} g_\nu$ are
$\Ocal ( n^{-1} )$ uniformly in $\nu \in N$
by a compacity argument using the continuity of~$\gamma$.

Therefore, we have
$$
\sum_{m = 1}^l g_{\nu}(m/n) = n \int_{0}^{1/2} g_{\nu} + \mathcal{O}(\ln(n)),
$$
uniformly in $\nu \in N$.
Moreover, Lemma~\ref{lem:rough_bound_c}
shows that $\ln \left( \lambda_{0, n}/n \right) = \Ocal(1)$ and
$\ln \left( \lambda_{n/2, n}/n \right) = \Ocal(\ln(n))$ uniformly for $n$ even.
One can then conclude using~\eqref{eq:eigen_sym}.
\end{proof}

\subsubsection{Approximating $Z\tr R_{\nu, \alpha} Z$}\label{sec:approx_quad}

Let us first give some definitions. For~$\epsilon > 0$, Lemma~\ref{lem:rough_gamma_behaviour} can be used
to show that there exists some
$C > 0$ such that
\begin{equation}\label{eq:envelope_def}
h_{\nu; \nu_0}(x) \leq F_{\epsilon}(x) =
C \min(x, 1 - x)^{- 1 + 2\epsilon}
,
\ \mathrm{for \ all} \ 0 < x < 1 \ \mathrm{and} \ \nu \in N_{\epsilon}.
\end{equation}
The function $F_{\epsilon}$ will be called the envelope of
the family $\mathcal{F}_{\epsilon} = \{h_{\nu; \nu_0}, \ \nu \in N_{\epsilon}\}$ of functions.

\begin{lem}\label{lem:dev_pnll_full}
For $0 < \epsilon < 1/2$, the sequence~$\Mset_n$
converges almost surely uniformly to $\left(\nu, \, \alpha\right) \mapsto \Uset \left( \nu \right)$ on~$N_\epsilon \times A$.
\end{lem}
\begin{proof}
For $\nu \in N_\epsilon$ and $\alpha \in A$, we have:
\begin{equation}\label{eq:cv_norm_unif}
\frac{\phi_0^{-1} Z\tr R_{\nu, \alpha}^{-1} Z}{n^{1 + 2(\nu - \nu_0)} } - \int_0^1 h_{\nu; \nu_0}
\end{equation}
$$
=
\frac{U_{0, n}^2  \lambda_{0, n}^{(0)} }{n^{1+2(\nu - \nu_0)} \lambda_{0, n}}
+
\frac{1}{n} \sum_{m=1}^{n-1} U_{m, n}^2 \left(\frac{\lambda_{m, n}^{(0)} }{n^{2(\nu - \nu_0)} \lambda_{m, n}} - h_{\nu; \nu_0} \left( m/n \right) \right) 
$$
$$
+ 
\frac{1}{n}  \sum_{m=1}^{n-1} B_{m, n} h_{\nu; \nu_0} \left( m/n \right)
+
\left( \frac{1}{n} \sum_{m=1}^{n-1}h_{\nu; \nu_0} \left( m/n \right)  - \int_0^1 h_{\nu; \nu_0} \right),
$$
with $B_{m, n} = U_{m, n}^2 - 1$.
First,
$
\sup_{\nu \in N_{\epsilon}} \left| n^{-1} \sum_{m = 1}^{n - 1} B_{m, n} h_{\nu ; \nu_0} \left( m/n \right) \right|
$
converges almost surely  to zero by
Lemma~\ref{lem:marginal_full_convergence},
Lemma~\ref{lem:stoch_eq}, and Arzel\`a-Ascoli theorem. Then, 
for all $\beta > 0$,
a Borel-Cantelli argument shows that $U_{0, n}^2 \lesssim n^{\beta}$ almost surely,
so the $m=0$-term converges almost surely uniformly to zero by Lemma~\ref{lem:rough_bound_c}.
Finally, Lemma~\ref{lem:int_h} and Lemma~\ref{lem:link_exp_norm_h} show that~\eqref{eq:cv_norm_unif}
converges almost surely uniformly.
Conclude using Proposition~\ref{prop:profile_lik}, Lemma~\ref{lem:log_det}, and the~$L^{\infty}$-continuity
at~$H\colon \left(\nu, \, \alpha\right) \in N_{\epsilon} \times A \mapsto \int_0^1 h_{\nu; \nu_0}$
of the mapping~$\psi$ used in the proof of Lemma~\ref{lem:dev_pnll}.
\end{proof}

\begin{lem}\label{lem:mono_h}
The function $h_{\nu; \nu_0}$
is non-decreasing (resp. non-increasing) on $(0, 1/2]$ when $\nu \geq \nu_0$ (resp. $\nu \leq \nu_0$).
\end{lem}
\begin{proof}
Suppose that $\nu \geq \nu_0$.
Use~\eqref{eq:hurwitz_zeta_link}
along with the fact that the Hurwitz Zeta function verifies
\begin{equation}\label{eq:hurwitz_der}
\frac{\partial \zeta_H}{\partial x}(s; \ x) = - s \zeta_H(s + 1; \ x),
\quad
\mathrm{for} \
x > 0,
\
\mathrm{and}
\
s > 1,
\end{equation}
and has the representation
$$
\zeta_H(s; \ x) = \frac{1}{\Gamma(s)} \int_{0}^{+\infty} \frac{t^{s - 1} e^{-tx}}{1 - e^{-t}} \ddiff t,
\quad
\mathrm{for}
\ x > 0,
\
\mathrm{and}
\
s > 1,
$$ 
where $\Gamma$ is the classical Gamma function \citep[see, e.g.,][]{postnikov1988:_introduction}.
So, for $x \in \left(0, \, 1\right)$, we have
$$
\gamma \left( 2\nu + 1; \, x\right) = \frac{1}{\Gamma(2\nu + 1)} \int_{0}^{+\infty} \frac{t^{2\nu} ( e^{-tx} + e^{-t(1-x)} ) }{1 - e^{-t}} \ddiff t,
$$
and
$$
\frac{\partial \gamma}{\partial x}\left( 2\nu + 1; \, x\right) = \frac{1}{\Gamma(2\nu + 1)} \int_{0}^{+\infty} \frac{t^{2\nu + 1} ( e^{-t(1-x)} - e^{-tx}) }{1 - e^{-t}} \ddiff t.
$$

Now let $x \in [1/2, 1)$, the derivative of $h_{\nu; \nu_0}$ at $x$ has the sign of
$$
\gamma \left( 2\nu+1; \, x \right)
\frac{\partial \gamma}{\partial x}\left( 2\nu_0 + 1; \, x\right)
-
\gamma \left( 2\nu_0 +1; \, x \right)
\frac{\partial \gamma}{\partial x}\left( 2\nu+1; \, x \right) 
$$
$$
= \frac{1}{\Gamma(2\nu + 1) \Gamma(2\nu_0 + 1)}
\int_{0}^{+\infty} \int_{0}^{+\infty}
\frac{t^{2\nu} s^{2\nu_0} ( \eta(s, t; x) - \eta(t, s; x)) }{\kappa(s, t)}
\ddiff t \ddiff s
$$
with $\eta(s, t; x) = s(e^{-tx} + e^{-t(1-x)})(e^{-s(1-x)} - e^{-sx})$ and $\kappa(s, t) = (1 - e^{-t}) (1 - e^{-s}) = \kappa(t, s)$
thanks to the Fubini-Lebesgue theorem.
Then, one can split the integral to have:
$$
\frac{1}{\Gamma(2\nu + 1) \Gamma(2\nu_0 + 1)}
\left(
\int_{0}^{+\infty} \int_{t}^{+\infty}
\frac{t^{2\nu} s^{2\nu_0} ( \eta(s, t; x) - \eta(t, s; x)) }{\kappa(s, t)}
\ddiff t \ddiff s
\right.
$$
$$
\left.
+
\int_{0}^{+\infty} \int_{t}^{+\infty}
\frac{s^{2\nu} t^{2\nu_0} ( \eta(t, s; x) - \eta(s, t; x)) }{\kappa(t, s)}
\ddiff t \ddiff s
\right)$$
\begin{align*}
& \, = \, \frac{1}{\Gamma(2\nu + 1) \Gamma(2\nu_0 + 1)} \\
& \, \cdot \, \int_{0}^{+\infty} \int_{t}^{+\infty}
\frac{(t^{2\nu} s^{2\nu_0} - s^{2\nu} t^{2\nu_0} ) ( \eta(s, t; x) - \eta(t, s; x)) }{\kappa(s, t)}
\ddiff t \ddiff s \leq 0
\end{align*}
since $t^{2\nu} s^{2\nu_0} \leq s^{2\nu} t^{2\nu_0}$ when $s \geq t$, $\kappa(s, t) \geq 0$
and $\eta(s, t; x) \geq \eta(t, s; x)$ when $s \geq t$ and $x \geq 1/2$.

So we proved that $h_{\nu; \nu_0}$ is non-increasing on $[1/2, 1)$
and the first claim is due to the symmetry with respect to $1/2$.
Observe
that $h_{\nu; \nu_0} = 1/h_{\nu_0; \nu}$ for the second claim.
\end{proof}

\begin{lem}\label{lem:int_h}
Let $\epsilon > 0$, we have
$$
\frac{1}{n} \sum_{m = 1}^{n - 1} h_{\nu; \nu_0}(m/n)
= \int_{0}^{1} h_{\nu; \nu_0} + \mathcal{O} \left( \frac{1}{n^{\min(1, \ 2\epsilon)}} \right),
$$
uniformly in $\nu \in N_{\epsilon}$. 
\end{lem}
\begin{proof}
The proof is similar to the treatment of~$n^{-1} \sum_{m = 1}^{n - 1} g_{\nu}(m/n)$
in the proof of Lemma~\ref{lem:log_det} using Lemma~\ref{lem:mono_h} and~\eqref{eq:envelope_def}
to get:
\begin{equation}\label{eq:int_h_zero}
\int_{0}^{1/n} h_{\nu; \nu_0} \leq \int_0^{1/n} F_\epsilon = \mathcal{O}(n^{-2\epsilon}), 
\quad
\mathrm{uniformly \ in} \ \nu \in N_{\epsilon}.
\end{equation}%
\end{proof}

\begin{lem}\label{lem:h_unif_alpha}
Let $1 \leq m \leq \floor*{n/2}$, we have
$$
\frac{\lambda_{m, n}^{(0)}}{n^{2(\nu - \nu_0)} \lambda_{m, n}}
= \left( 1 + \mathcal{O}(m^{-2}) \right) h_{\nu; \nu_0}(m/n)
$$
uniformly in $\nu \in N$ and $\alpha \in A$.
\end{lem}
\begin{proof}
A direct consequence from Lemma~\ref{lem:bound_c_rho}.
\end{proof}

\begin{lem}\label{lem:link_exp_norm_h}
Let $0 < \epsilon < 1/2$ and $0 < \delta < 2\epsilon$. There exists a constant $C$
such that
$$
\limsup n^{2\epsilon - \delta}
\sup_{(\nu, \alpha) \in N_{\epsilon} \times A} \frac{1}{n }
\sum_{m = 1}^{n - 1} 
U_{m, n}^2
\left|  \frac{\lambda_{m, n}^{(0)}}{n^{2(\nu - \nu_0)} \lambda_{m, n}} - h_{\nu; \nu_0} \left( \frac{m}{n} \right) \right| \leq C,
$$
almost surely.
\end{lem}
\begin{proof}
Let $\nu \in N_\epsilon$, $\alpha \in A$, and
$p = 1/(1 - 2\epsilon + \delta)$.
We have
$n^{-1} \sum_{m = 1}^{n-1} F_{\epsilon}^p (m/n) = \Ocal(1)$.
Then, 
Lemma~\ref{lem:h_unif_alpha}, the usual symmetry arguments, and H\"older inequality
yield:
\begin{align*}
 & \frac{1}{n }
\sum_{m = 1}^{n-1} U_{m, n}^2 \left|  \frac{\lambda_{m, n}^{(0)}}{n^{2(\nu - \nu_0)} \lambda_{m, n}} - h_{\nu; \nu_0} \left( \frac{m}{n} \right) \right|\\
& \, = \,
\frac{1}{n}
 \sum_{m = 1}^{n-1} U_{m, n}^2 \Ocal \left( m^{-2} \vee (n - m)^{-2} \right) h_{\nu; \nu_0}\left( \frac{m}{n} \right)\\
& \, \leq \,
\frac{1}{n^{1/q}} 
\cdot
\left(
\frac{1}{n} \sum_{m = 1}^{n-1} \left| U_{m, n} \right|^{2p} F_{\epsilon}^p(m/n)
\right)^{1/p}
\cdot
\underbrace{ \left(
\sum_{m = 1}^{n-1} \Ocal \left( m^{-2q} \vee (n - m)^{-2q} \right)
\right)^{1/q}}_{\Ocal(1) \ \mathrm{uniformly}}
\end{align*}
with $1/q = 2\epsilon - \delta$.
Conclude using
Lemma~\ref{lem:riemann_sum_bound} and $n^{-1} \sum_{m = 1}^{n-1} F_{\epsilon}^p (m/n) = \Ocal(1)$.
\end{proof}

For $n \geq 2$ and $1 \leq m \leq n - 1$, define $B_{m, n} = U_{m, n}^2 - 1$.
\begin{lem}\label{lem:marginal_full_convergence}
Let $\nu > \nu_0 - 1/2$. Then,
$n^{-1} \sum_{m = 1}^{n - 1} B_{m, n} h_{\nu ; \nu_0} \left( m/n \right)$
converges almost surely to zero.
\end{lem}
\begin{proof}
By Lemma~\ref{lem:riemann_sum_bound}, since
$0 \leq h_{\nu ; \nu_0}(x) \lesssim \min \left(x, \, 1 - x\right)^{2(\nu-\nu_0)}$.
\end{proof}

\begin{lem}\label{lem:riemann_sum_bound}
Let $\alpha > - 1$ and $g\colon \left(0, \, 1\right) \to \RR$ a function
such that the inequality $0 \leq g(x) \lesssim \min \left(x, \, 1 -x \right)^{\alpha}$
holds.
For each~$n$, let~$D_{1, n}, \dots, D_{n-1, n}$ be i.i.d. centered variables
such that~$\EE ( \left| D_{1, 2} \right|^q )$
is finite for all~$q \geq 0$.
Then, $n^{-1} \sum_{m = 1}^{n - 1} D_{m, n} g \left( m/n \right)$
converges almost surely to zero.
\end{lem}
\begin{proof}
If $\alpha \geq 0$, then~$g \left( m/n \right) = \Ocal(1)$, so
the result is given by \citep[Corollary~5]{taylor1987:_slln}.
Otherwise if $\alpha < 0$, then let $0 <\delta < 1/2$. It holds that:
$$
\left| \frac{1}{n} \sum_{m = 1}^{n - 1} D_{m, n} g \left( \frac{m}{n} \right) \right|\\
\leq
\left| \frac{1}{n} \sum_{m = 1}^{n - 1}
\underbrace{
g \left( \frac{m}{n} \right)
\one_{\floor*{\delta n} + 1 \leq m \leq n - \floor*{\delta n} - 1}
}_{\lesssim \delta^{\alpha}}
 D_{m, n}  \right|
$$
$$
+
\left|
\frac{1}{n} \sum_{m = 1}^{n - 1} \left(\one_{m \leq \floor*{\delta n}}
+ \one_{m \geq n - \floor*{\delta n} }
\right) D_{m, n} g\left( \frac{m}{n} \right) \right|.
$$
The first term converges almost surely to zero by \citep[Corollary~5]{taylor1987:_slln}. For the second term,
H\"older inequality gives (a multiple of) the bound:
$$
\left( \frac{1}{n} \sum_{m = 1}^{n - 1} 
\left| D_{m, n} \right|^q \right)^{1/q}
\cdot
\left( \frac{2}{n} \sum_{m = 1}^{\floor*{\delta n}} \left( \frac{m}{n} \right)^{p \alpha} \right)^{1/p}.
$$
The first term converges almost surely to the $q$-norm of the $D_{m, n}$ by the previous reference
and, for~$p$ close enough to one, the second is
$\Ocal ( \delta^{\alpha + 1/p} )$ with $\alpha + 1/p > 0$.
Take $\delta = 1/j$ and a countable intersection of almost sure events to conclude.
\end{proof}

\begin{lem}\label{lem:stoch_eq}
Let $0 < \epsilon < 1/2$ and define
$$g_n\colon
\nu \in N_{\epsilon} \mapsto \frac{1}{n }\sum_{m = 1}^{n - 1} B_{m, n} h_{\nu ; \nu_0} \left( \frac{m}{n} \right).$$
The sequence $(g_n)_{n \geq 2}$
is almost surely uniformly equicontinuous.
\end{lem}
\begin{proof}
Lemma~\ref{lem:rough_gamma_der_behaviour}
shows
that
$$
\left| \frac{\partial \gamma}{\partial s} \left(2\nu + 1; \, x\right) \right|
\lesssim - x^{-2\nu - 1} \ln(x)  
\lesssim
x^{-2(\nu + \delta) - 1}
\quad
(\mathrm{with \ the \ notation} \ \gamma \left(s; \, x\right)),
$$
holds uniformly in $x \in \left(0, \, 1/2\right]$ and $\nu \in N_\epsilon$,
for any~$\delta > 0$. 
With a slight abuse of notation, the latter fact and Lemma~\ref{lem:rough_gamma_behaviour}
yield:
\begin{equation}\label{eq:kappa_bound}
\left| \frac{\partial h_{\nu_0}}{\partial \nu}\left(\nu; \, m/n \right) \right|
\lesssim \left( \frac{
n
}{
m
} \right)^{1 - 2\epsilon + 2\delta},
\end{equation}
uniformly in $n$, $1 \leq m \leq \floor*{n/2}$, and $\nu \in N_{\epsilon}$.

Now let $\nu_1, \, \nu_2 \in N_{\epsilon}$. If one chooses $p > 1$ and $\delta > 0$ such that $p(1 - 2\epsilon + 2\delta) < 1$,
then we have by H\"older's inequality with $1/q + 1/p = 1$
$$
\left| g_n(\nu_{1}) - g_n(\nu_{2}) \right|
$$
$$
\leq
\left(\frac{1}{n} \sum_{m = 1}^{n - 1} \left| B_{m, n} \right|^q \right)^{1/q}  
\cdot 
\underbrace{
\left(\frac{1}{n} \sum_{m = 1}^{n - 1} \sup_{\nu \in N_{\epsilon} } \left| \frac{\partial h_{\nu_0}}{\partial \nu}\left(\nu; \, m/n \right)  \right|^p  \right)^{1/p}
}_{\Ocal(1) \ \mathrm{by} \ \eqref{eq:kappa_bound}}
\cdot
\abs{\nu_{1}  - \nu_{2} }.
$$
Use \citep[Corollary~5]{taylor1987:_slln} to conclude.
\end{proof}

\subsection{Proof of Theorem~\ref{thm:asympt_norm}}\label{sec:proof_asympt_norm}

\subsubsection{An upper bound of the rate}
\begin{lem}\label{lem:consitency_plus_rate}
Let $0 < \beta < 1/4$.
It holds that $\hat{\nu}_n - \nu_0 = o_{\PP} \left( n^{-\beta} \right)$ and
$\hat{\phi}_n - \phi_0 = o_{\PP} \left( n^{-\beta} \right)$.
\end{lem}
\begin{proof}
Let $0 <\beta < 1/4$, Proposition~\ref{prop:profile_lik} gives almost surely
$$
\ln \left( \hat{\phi}_n \right)
=  \ln \left( \phi_0 \right) + \ln \left( \frac{\phi_0^{-1} Z\tr R_{\hat{\nu}_n, \hat{\alpha}_n}^{-1} Z}{n^{1 + 2(\hat{\nu}_n - \nu_0)}} \right) + 2(\hat{\nu}_n - \nu_0) \ln(n).
$$
So
\begin{align*}
& \frac{n^{\beta}}{\ln(n)} \left( \ln \left( \hat{\phi}_n \right) - \ln \left( \phi_0 \right) \right) \\
& \, = \, \frac{n^{\beta}}{\ln(n)}  \ln \left( H \left( \hat{\nu}_n \right) \right) 
+ \frac{n^{\beta}}{\ln(n)} \left(\ln \left( \frac{\phi_0^{-1} Z\tr R_{\hat{\nu}_n, \hat{\alpha}_n}^{-1} Z}{n^{1 + 2(\hat{\nu}_n - \nu_0)}} \right) - \ln \left( H \left( \hat{\nu}_n \right)  \right) \right) \\
& \, + \, 2n^{\beta} (\hat{\nu}_n - \nu_0).
\end{align*}
The latter converges to zero in probability thanks to the coordination of~\eqref{eq:weak_conv_log}
with Slutsky's lemma in $L^{\infty}(N_{\epsilon} \times A)$ \citep[p. 32]{van1996:_weak},
Lemma~\ref{lem:bound_rate_nu},
and the univariate delta method since
the mapping $\ln \circ H$ is smooth.
This implies that
$\ln ( \hat{\phi}_n ) - \ln \left( \phi_0 \right) = o_{\PP} \left( n^{-\beta} \right)$
for all~$0 <\beta < 1/4$.
Conclude using again the delta method.
\end{proof}

\begin{lem}\label{lem:bound_rate_nu}
Let $0 < \beta < 1/4$. The bound $\hat{\nu}_n - \nu_0 = o_{\PP} \left( n^{-\beta} \right)$ holds in probability.
\end{lem}
\begin{proof}
Let $1/4 < \epsilon < 1/2$ and $0 < \beta < 1/2$ and use the notations from Appendix~\ref{sec:more_notations}.
Lemma~\ref{lem:dev_pnll} implies
$
\sup_{\nu \in N_{\epsilon}} \left| \Uset_n \left( \nu \right) - \Uset \left( \nu \right) \right|
= o_{\PP} \left( n^{-\beta}\right)
$.
Moreover, the function $\Uset$ is $C^{3}$-smooth and we have $\Uset^{\prime}(\nu_0) = 0$
and, with the notation given by~\eqref{eq:def_psi}:
$$
\Uset^{\prime \prime}(\nu_0) =
4 \left(
\int_0^1 \left( \psi_{\nu_0}  \right)^2
- \left( \int_0^1 \psi_{\nu_0} \right)^2
\right) > 0,
$$
thanks to Jensen inequality.
Finally, Theorem~\ref{thm:consistency_nu} and
a second-order Taylor expansion around~$\nu_0$
give the rate $n^{- \beta/2}$.
\end{proof}

\begin{lem}\label{lem:dev_pnll}
Let $1/4 < \epsilon < 1/2$. Then, the sequence
$$
\left(\nu, \, \alpha\right) \in N_{\epsilon} \times A \mapsto
\sqrt{n} \left( \Mset_n \left( \nu, \alpha \right)
- \int_{0}^{1} g_\nu
- \ln \left( \int_{0}^{1} h_{\nu; \nu_0}  \right) \right)
$$
of processes converges weakly in $L^{\infty}(N_{\epsilon} \times A)$ to
\begin{equation}\label{eq:lim_gp_delta_meth}
\mathrm{GP} \left( 0,
\left(\nu_1, \, \alpha_1; \, \nu_2, \, \alpha_2\right) \mapsto
\frac{2\int_0^1 h_{\nu_1; \nu_0} h_{\nu_2; \nu_0} }{\int_0^1 h_{\nu_1; \nu_0} \int_0^1 h_{\nu_2; \nu_0}} \right)
\end{equation}
which can be seen as a tight Borel probability measure.
In particular, for all $\beta < 1/2$, we have
$$
\sup_{\nu \in N_\epsilon, \alpha \in A} \left| \Mset_n \left( \nu, \alpha \right)
- \int_{0}^{1} g_\nu
- \ln \left( \int_{0}^{1} h_{\nu; \nu_0}  \right) \right| = o_{\PP}\left( n^{-\beta} \right).
$$
\end{lem}
\begin{proof}
Use the notation~$H \colon \left(\nu, \, \alpha\right) \in N_{\epsilon} \times A \mapsto \int_{0}^{1} h_{\nu; \nu_0}$
for this proof.

Let $\Dset_{\psi} \subset L^{\infty}(N_{\epsilon} \times A)$ be
the subset of positive functions bounded away from zero.
One has
$H \in \Dset_{\psi}$ and
$
\left(\nu, \, \alpha\right) \in N_{\epsilon} \times A \mapsto n^{- 1 - 2(\nu - \nu_0)} \phi_0^{-1} Z\tr R_{\nu, \alpha}^{-1} Z 
$
lying also in~$\Dset_{\psi}$
almost surely by continuity on the compact $N_{\epsilon} \times A$.

Furthermore, the mapping
$\psi: g \in \Dset_{\psi} \subset L^{\infty}(N_{\epsilon} \times A) \mapsto \ln \circ g \in L^{\infty}(N_{\epsilon} \times A)$ is Fr\'echet-differentiable
at~$H$ with $\psi^{\prime}(H): g \in L^{\infty}(N_{\epsilon} \times A) \mapsto g/H
\in L^{\infty}(N_{\epsilon} \times A)$.
The weak limit given by Lemma~\ref{lem:link_norm_h} is tight and hence separable,
so we can use Theorem~3.9.4 from \cite{van1996:_weak}
to show that
\begin{equation}\label{eq:weak_conv_log}
\sqrt{n} \left( \ln \left(  \frac{\phi_0^{-1} Z\tr R_{\nu, \alpha}^{-1} Z}{n^{1 + 2(\nu - \nu_0)}} \right)
- \ln \left( \int_0^1 h_{\nu; \nu_0}\right) \right) 
\end{equation}
converges weakly to~\eqref{eq:lim_gp_delta_meth}
in $L^{\infty}(N_{\epsilon} \times A)$.
The tightness of the limit follows from the continuity of~$\psi^{\prime}(H)$.
Conclude with
Proposition~\ref{prop:profile_lik},
Lemma~\ref{lem:log_det}, and Slutsky's lemma.
\end{proof}

\begin{lem}\label{lem:link_norm_h}
Let $1/4 < \epsilon < 1/2$. The sequence
$$
\left(\nu, \, \alpha\right) \in N_{\epsilon} \times A \mapsto
\sqrt{n} \left( \frac{\phi_0^{-1} Z\tr R_{\nu, \alpha}^{-1} Z}{n^{1 + 2(\nu - \nu_0)}} - \int_0^1 h_{\nu; \nu_0} \right)
$$
of processes converges weakly in $L^{\infty}(N_{\epsilon} \times A)$ to
$$
\mathrm{GP} \left( 0, \left(\nu_1, \, \alpha_1; \, \nu_2, \, \alpha_2\right) \mapsto 2 \int_0^1 h_{\nu_1; \nu_0} h_{\nu_2; \nu_0} \right),
$$
which can be seen as a tight Borel probability measure.
\end{lem}
\begin{proof}
Using the continuous mapping theorem for the isometry
$
\rho \colon L^{\infty}(N_{\epsilon})
\to L^{\infty}(N_{\epsilon} \times A)
$
mapping $g \in L^{\infty}(N_{\epsilon})$ to the function $\left(\nu, \, \alpha\right) \in N_{\epsilon} \times A \mapsto g(\nu)$
makes it possible
to rephrase the convergence given by Lemma~\ref{lem:ep}
in $L^{\infty}(N_{\epsilon} \times A)$.
(The limit~\eqref{eq:lim_sum_h} is a tight and hence separable measure.)
The rest of the proof is similar to the analysis of~\eqref{eq:cv_norm_unif} in the proof of Lemma~\ref{lem:dev_pnll_full}, but using~$\epsilon > 1/4$.
\end{proof}

\begin{lem}\label{lem:vc}
Let $1/4 < \epsilon < 1/2$. 
The family $\mathcal{F}_{\epsilon}$ of functions
equiped with the envelope~$F_{\epsilon}$ defined by~\eqref{eq:envelope_def}
verifies the uniform entropy condition \citep[][Section 2.5.1]{van1996:_weak}.
\end{lem}
\begin{proof}
For $x \in \left(0, \, 1\right)$ and $\nu \in N$,
write
$$\gamma \left( 2\nu + 1; \, x\right) = 
\gamma_{\uparrow} \left( 2\nu + 1; \, x\right)
+
\gamma_{\downarrow} \left( 2\nu + 1; \, x\right),$$
with
$$\gamma_{\downarrow} \left( 2\nu + 1; \, x\right)
= \sum_{j = 1}^{+\infty} (j + x)^{-2\nu - 1}
+ \sum_{j = 1}^{+\infty} (j + 1 - x)^{- 2\nu - 1},$$
and~$\gamma_{\uparrow} \left( 2\nu + 1; \, x\right)
= x^{-2\nu - 1}
+ (1 - x)^{- 2\nu - 1}$.
Let
$$h_{\uparrow} (\nu; \, x) =
\frac{\gamma \left(2\nu_0 + 1, \, x \right)}{
\gamma_{\downarrow} \left(2\nu + 1, \, x \right)}
\
\mathrm{and}
\
h_{\downarrow}(\nu; \, x) =
\frac{\gamma \left(2\nu_0 + 1, \, x \right)}{
\gamma_{\uparrow} \left(2\nu + 1, \, x \right)}.\footnote{
The symbols~$\downarrow$ and~$\uparrow$ account for the monotonicity
with respect to~$\nu$ for fixed~$x$.
}$$
The families
$\mathcal{F}_{\epsilon}^{\downarrow}
= \left\lbrace
h_{\downarrow} (\nu; \, \cdot) , \, \nu \in N_{\epsilon} \right\rbrace$ and
$\mathcal{F}_{\epsilon}^{\uparrow}
= \left\lbrace
1/h_{\uparrow} (\nu; \, \cdot)
, \, \nu \in N_{\epsilon} \right\rbrace$
of functions
are non-increasing with respect to the parameter~$\nu$ so
they are VC-subgraph classes.
Indeed, let $(x_1, y_1), (x_2, y_2) \in (0, 1) \times \RR$,
there cannot be two functions $f$ and $g$ in one of these families such that
$f(x_1) < y_1$, $f(x_2) \geq y_2$, $g(x_1) \geq y_1$, and $g(x_2) < y_2$,
since we have either $g \leq f$ or $f \leq g$.

Equip
$\mathcal{F}_{\epsilon}^{\downarrow}$ and
$\mathcal{F}_{\epsilon}^{\uparrow}$
respectively with the
envelopes $F_{\epsilon}$ (by increasing eventually the constant~$C$ in~\eqref{eq:envelope_def})
and $F_{\epsilon}^{\uparrow}\colon x \in \left(0, \, 1\right) \mapsto C_2 \min(x, \, 1 - x)^{1 + 2\nu_0}$, for some constant $C_2 > 0$.
Theorem 2.6.7 from \citep{van1996:_weak} shows that these families satisfy
the uniform entropy condition.

Consider
$\varsigma \colon x, y \in \left(0, + \infty\right) \mapsto \left( x^{-1} + y \right)^{-1}$.
It holds that $\left| \frac{\partial \varsigma}{\partial x}(x, \, y) \right| \leq 1$
and $\left| \frac{\partial \varsigma}{\partial y} (x, \, y) \right| = \varsigma^2 (x, \, y)$.
Observe that
$\varsigma \left(
h_{\downarrow} (\nu_1; \, \cdot), \,
1/h_{\uparrow} (\nu_2; \, \cdot) \right) \lesssim F_{\epsilon}$,
for $\nu_1, \, \nu_2 \in N_{\epsilon}$.
Consequently, for $\nu_1, \nu_2, \nu_3, \nu_4 \in N_{\epsilon}$ and $x \in \left(0, \, 1\right)$, we have:
$$
\left( \varsigma \left(
h_{\downarrow} (\nu_1; \, x), \,
1/h_{\uparrow} (\nu_3; \, x) \right) -
\varsigma \left(
h_{\downarrow} (\nu_2; \, x), \,
1/h_{\uparrow} (\nu_4; \, x) \right)
\right)^2$$
$$
\lesssim \left(
h_{\downarrow} (\nu_1; \, x)
- 
h_{\downarrow} (\nu_2; \, x)
 \right)^2
+ F_{\epsilon}^4 (x) \left(
\frac{1}{h_{\uparrow} (\nu_3; \, x)}
- 
\frac{1}{h_{\uparrow} (\nu_4; \, x)}
 \right)^2.
$$
Observe that
$\varsigma \left(
h_{\downarrow} (\nu; \, \cdot), \,
1/h_{\uparrow} (\nu; \, \cdot) \right) = h_{\nu; \nu_0}$
and use Theorem~2.10.20 from \citep{van1996:_weak}
to conclude that the family
$$\mathcal{F}_{\epsilon}^{(\pi)} = \left\lbrace \varsigma \left(
h_{\downarrow} (\nu_1; \, \cdot), \,
1/h_{\uparrow} (\nu_2; \, \cdot) \right)-1, \ \nu_1, \nu_2 \in N_{\epsilon}
\right\rbrace
\quad
\left( \mathrm{note \ that} \ 
h_{\nu_0;\nu_0} = 1
 \right)$$
with envelope $F_{\epsilon}^{(\pi)}  = 2\sqrt{F_{\epsilon}^2 + F_{\epsilon}^4 ( F_{\epsilon}^{\uparrow})^2}$ satisfy
the uniform entropy condition.
Concluding the proof is straightforward
since~$\mathcal{F}_\epsilon \subset \mathcal{F}_{\epsilon}^{(\pi)}  + 1$
and~$F_{\epsilon}^{(\pi)} \lesssim F_\epsilon$.
\end{proof}

\begin{lem}\label{lem:L2_dist}
For all $\epsilon > 1/4$, we have
$$
\frac{1}{n} \sum_{m = 1}^{n-1}
\left( h_{\nu_1; \nu_0} \left( m / n \right) - h_{\nu_2; \nu_0} \left( m / n \right) \right)^2
\to \int_0^1 \left( h_{\nu_1; \nu_0} - h_{\nu_2; \nu_0} \right)^2,
$$
uniformly in $\nu_1, \nu_2 \in N_{\epsilon}$.
\end{lem}
\begin{proof}
Let $\delta > 0$, there exists $\alpha > 0$ such that:
$$
\int_{0}^{\alpha}  \left( h_{\nu_1; \nu_0} - h_{\nu_2; \nu_0} \right)^2
\leq 4 \int_{0}^{\alpha}  F_{\epsilon}^2 \leq \delta/5
$$
and
$$
\frac{1}{n} \sum_{m = 1}^{\floor*{\alpha n}}  \left( h_{\nu_1; \nu_0} \left( m/n \right) - h_{\nu_2; \nu_0} \left( m/n \right) \right)^2
\leq \frac{4}{n} \sum_{m = 1}^{\floor*{\alpha n}}  F_{\epsilon}^2 \left( m/n \right)
\leq \delta/5,
$$
uniformly in $\nu_1, \nu_2 \in N_{\epsilon}$.
The same bounds also hold by symmetry for similar quantities related to $\left[1 - \alpha, \, 1\right]$.
Furthermore, a compacity argument using the smoothness of~$\gamma$ shows that
the mapping $x \in \left(0, \, 1 \right) \mapsto \left( h_{\nu_1; \nu_0}(x) - h_{\nu_2; \nu_0}(x) \right)^2$
and its derivative are
bounded on $[\alpha, \ 1 - \alpha]$ uniformly in $\nu_1, \nu_2 \in N_{\epsilon}$.
Consequently, the standard technique for bounding approximation errors of Riemann sums gives
$$
\left|
\frac{1}{n} \sum_{m = \floor*{\alpha n} + 1}^{\ceil*{(1 - \alpha) n} - 1}  \left( h_{\nu_1; \nu_0} \left( m/n \right) - h_{\nu_2; \nu_0} \left( m/n \right) \right)^2 
- \int_{\alpha}^{1 - \alpha}  \left( h_{\nu_1; \nu_0} - h_{\nu_2; \nu_0} \right)^2
\right| \leq \delta/5,
$$
uniformly in $\nu_1, \nu_2 \in N_{\epsilon}$, for sufficiently large $n$.
\end{proof}

For $n \geq 2$ and $1 \leq m \leq n - 1$, define $B_{m, n} = U_{m, n}^2 - 1$.
\begin{lem}\label{lem:ep}
Let $1/4 < \epsilon < 1/2$. Then, the sequence
$$
\nu \in N_{\epsilon} \mapsto \frac{1}{\sqrt{n}}\sum_{m = 1}^{n - 1} B_{m, n} h_{\nu ; \nu_0} \left( \frac{m}{n} \right)
$$
of processes converges weakly in $L^{\infty}(N_{\epsilon})$ to
\begin{equation}\label{eq:lim_sum_h}
\mathrm{GP} \left( 0, (\nu_1, \nu_2) \mapsto 2\int_0^1 h_{\nu_1; \nu_0} h_{\nu_2; \nu_0} \right),
\end{equation}
which can be seen as a tight Borel probability measure.
\end{lem}
\begin{proof}
Let $2 < \alpha < 1/(1 - 2\epsilon)$.
It holds that $F_{\epsilon} \in
L^\alpha \left( 0, \, 1 \right) \subset L^2 \left( 0, \, 1 \right)$.
Moreover,
Lemma~\ref{lem:vc} shows that
$\mathcal{F}_{\epsilon}$ satisfies the uniform entropy condition
\citep[Section 2.5.1]{van1996:_weak}.

Let us show that $(\mathcal{F}_{\epsilon}, \, \ns{\cdot}_{L^2 \left(0, \, 1\right)} )$
is totally bounded.
Use the shortcut $Q_n = n^{-1} \delta_{1/2} + n^{-1}\sum_{m=1}^{n-1} \delta_{m/n}$.
Since $\epsilon > 1/4$, then
$\int F_{\epsilon}^2 \ddiff Q_n$
is bounded uniformly in $n$ by, say,~$M^2$.
The uniform entropy condition implies that
$\mathcal{F}_{\epsilon}$
is totally bounded for the~$L^2 \left(Q_n\right)$-norm
for any~$n$.
Let $\mathcal{G}_n$ be an $(M \delta)$-internal covering, for~$\delta > 0$.
Lemma~\ref{lem:L2_dist} makes it possible to choose $n$
such that $$\sup_{g1, g2 \in \mathcal{F}_{\epsilon}} \left|
\int \left(g_1 - g_2\right)^2 \ddiff Q_n -
\int_0^1 \left(g_1 - g_2\right)^2 \right| \leq \delta^2.$$
Therefore, $\mathcal{G}_n$ is a $(\delta \sqrt{M^2 + 1})$-covering of
$(\mathcal{F}_{\epsilon}, \, \ns{\cdot}_{L^2 \left(0, \, 1\right)} )$.

With $Y_{m, n}\colon g \in (\mathcal{F}_{\epsilon}, \, \ns{\cdot}_{L^2 \left( 0, \, 1\right)}) \mapsto n^{-1/2} B_{m, n} g(m/n)$,
the usual measurability conditions
\citep[see][p. 205]{van1996:_weak}
are met since the suprema
can be replaced by ones on countable sets.
Indeed, using the surjection $\varrho\colon \nu \in N_{\epsilon} \mapsto h_{\nu; \nu_0} \in \mathcal{F}_{\epsilon}$, the suprema
on subsets of~$\mathcal{F}_{\epsilon} \times \mathcal{F}_{\epsilon}$
are suprema on subsets
of~$\left( N_{\epsilon} \times N_{\epsilon}, \, \ns{\cdot}_2 \right)$,
with~$\ns{\cdot}_2$ standing for the euclidean norm.
A subset of a separable metric space is separable.
The sample path continuity of
$\nu \in N_{\epsilon} \mapsto Y_{m, n} \left( \varrho \left( \nu \right) \right)$ is inherited from the continuity of $\nu \in N_{\epsilon} \mapsto h_{\nu; \nu_0}(x)$, for~$0 < x < 1$.

Since $2 < \alpha < 1/(1 - 2\epsilon)$, we have
$ n^{-1} \sum_{m=1}^{n-1} F_{\epsilon}^{\alpha} (m/n) = \Ocal(1)$
so
the Lyapunov condition on suprema holds:
\begin{align*}
\sum_{m=1}^{n-1} \EE \left(
\sup_{g \in \mathcal{F}_{\epsilon}} \left| Y_{m, n}(g)  \right|^{\alpha} 
\right)
    & \;\leq\; 
 \frac{\EE \left( \left| B_{1, 2} \right|^{\alpha} \right)}{n^{\alpha/2}} \sum_{m=1}^{n-1} F_{\epsilon}^{\alpha}(m/n)
 = o(1).
\end{align*}

Furthermore, for $\delta_n \to 0$, we have
\begin{align*}
&
\sup_{ \ns{g_1 - g_2}_{L^2 \left( 0, \, 1\right)} < \delta_n} 
\sum_{m = 1}^{n - 1}
\EE \left( \left( Y_{m, n} (g_1) - Y_{m, n} (g_2) \right)^2 \right)
\quad (\mathrm{with} \ g_1, g_2 \in \mathcal{F}_{\epsilon})
\\
  & \; = \; 
  \EE \left( B_{1, 2}^2
\right)
\sup_{\ns{g_1 - g_2}_{L^2 \left( 0, \, 1\right)} < \delta_n} 
\frac{1}{n} \sum_{m = 1}^{n - 1} \left( g_1(m/n) - g_2(m/n) \right)^2
 \\
  & \; = \; 
  o(1) +   \Ocal (\delta_n^2) \to 0
\end{align*}
thanks to Lemma~\ref{lem:L2_dist}.

Now, let us show the pointwise convergence of
the sequence of covariance functions.
For a fixed~$\nu \in N_\epsilon$,
the convergence
$n^{-1} \sum_{m = 1}^{n-1} h_{\nu; \nu_0}^2 \left( m / n \right) \to \int_0^1 h_{\nu; \nu_0}^2$
is ensured using Lemma~\ref{lem:mono_h} and the same reasoning as
in the proof of Lemma~\ref{lem:int_h}.
This fact and Lemma~\ref{lem:L2_dist} shows that
$$\cov \left(\sum_{m=1}^{n-1} Y_{m, n}(g_1), \, 
\sum_{m=1}^{n-1} Y_{m, n}(g_2)\right) \to 2 \int_0^1 g_1 g_2,$$
for fixed $g_1, g_2 \in \mathcal{F}_{\epsilon}$.

Finally, with $\mu_{n, m} = n^{-1} B_{m, n}^2 \delta_{m/n}$, one
has $0 < \mu_{n, m} F_{\epsilon}^2 < + \infty$ almost surely and $\sum_{m=1}^{n-1} \mu_{n, m} F_{\epsilon}^2 = \Ocal_{\PP}(1)$
using Markov's inequality.

We can then conclude using Lemma~2.11.6 and Theorem~2.11.1 from \citet{van1996:_weak},
which also imply the tightness of the limit
\citep[see][Lemma 1.3.8 and Theorem 1.5.7]{van1996:_weak}.
The reformulation
from $L^{\infty} \left( \mathcal{F}_\epsilon \right)$
to $L^{\infty} \left( N_\epsilon \right)$ is an application of the continuous mapping theorem.
\end{proof}

\subsubsection{A Taylor expansion}

The proof of Theorem~\ref{thm:asympt_norm} is finished using a standard third-order Taylor expansion around~$(\nu_0, \phi_0, \hat{\alpha}_n)$.
The following technical lemmata are required. Their proofs mostly consist in reproducing
the technique used by \citet[Section 6.7]{stein1999interpolation} to derive the asymptotics of the
Fisher information matrix. Some details are provided in~\secasymptnormsuppmat.
\begin{lem}\label{lem:score_function}
We have the following convergence in distribution
$$
\frac{ \sqrt{n}}{2 \sqrt{2}} A_n\tr \nabla \Lset_n(\nu_0, \, \phi_0, \, \hat{\alpha}_n) \leadsto \Ncal \left(0,  \, I_2 \right),
$$%
with $\nabla \Lset_n$ the gradient with respect to $\left(\nu, \, \phi\right)$ only and
\begin{equation}\label{eq:def_A_n}
A_n = 
\frac{2 \phi_0 }{\sqrt{\mathrm{Var} \left( \psi_{\nu_0} (V) \right)}}
\begin{pmatrix}
2^{-1} \phi_0^{-1}  &  \ 0   \\
\ln(n) + \EE \left( \psi_{\nu_0} (V) \right) & \ \sqrt{\mathrm{Var}  \left( \psi_{\nu_0} (V) \right)} \\
\end{pmatrix},
\end{equation}%
where $V$ is a random variable distributed uniformly on $\left(0, \, 1\right)$.
\end{lem}

\begin{lem}\label{lem:fisher_matrix}
It holds in probability that:
$$A_n\tr \nabla^2 \Lset_n \left( \nu_0, \, \phi_0, \, \hat{\alpha}_n \right) A_n \to 4 I_2,$$%
with $A_n$ given by~\eqref{eq:def_A_n} and $\nabla$ operating only on $\left(\nu, \, \phi\right)$.
\end{lem}

\begin{proof}[Proof of Theorem~\ref{thm:asympt_norm}]
Lemmata~\ref{lem:score_function} and~\ref{lem:fisher_matrix} give the asymptotics of the score
and the Hessian matrix, respectively.
We are now left to bound the third derivatives uniformly locally around $\left(\nu_0, \, \phi_0\right)$.
Cumbersome expressions are provided in~\secasymptnormsuppmat.
For~$\epsilon > 0$ small enough,
bounding the terms individually with
Lemma~\ref{lem:rough_bound_c} and
Lemma~\lemroughboundderivativesuppmat
makes it straightforward to show that
\begin{equation}\label{eq:majoration_third_der}
\EE \left(
\sup_{0 \leq p \leq 3, \abs{\nu - \nu_0} \leq \epsilon, \abs{\phi - \phi_0} \leq \epsilon, \alpha \in A}
\left|
\frac{
\partial^3 \Lset_n
}{
\left( \partial \nu \right)^p \left( \partial \phi \right)^{3 - p}
}
\left(\nu, \phi, \alpha \right)
\right|
\right) = \Ocal \left( n^{5 \epsilon} \right).
\end{equation}%
Lemma~\ref{lem:consitency_plus_rate} shows that
$(\hat{\nu}_n, \, \hat{\phi}_n ) 
\in \left[\nu_0 - \epsilon, \, \nu_0 + \epsilon\right]
\times \left[\phi_0 - \epsilon, \, \phi_0 + \epsilon\right]
$
with high probability.
Write $\nabla$ for taking derivatives with respect to $\left(\nu, \, \phi\right)$ only.
On this event, we have:
$$
0 = \nabla \Lset_n \left(\nu_0, \phi_0, \hat{\alpha}_n \right)
+ \nabla^2 \Lset_n \left(\nu_0, \phi_0, \hat{\alpha}_n \right) \begin{pmatrix}
\hat{\nu}_n - \nu_0   \\
\hat{\phi}_n - \phi_0 \\
\end{pmatrix}
$$
$$
+ \Ocal_{\PP} \left( n^{5 \epsilon}
\left| \left| 
\begin{pmatrix}
\hat{\nu}_n - \nu_0   \\
\hat{\phi}_n - \phi_0 \\
\end{pmatrix}
\right| \right|^2
\right),
$$
thanks to~\eqref{eq:majoration_third_der}.
Multiplying by~$A_n\tr$ (see~\eqref{eq:def_A_n}) and using Lemma~\ref{lem:consitency_plus_rate} again leads to
$$
0 = A_n\tr \nabla \Lset_n \left(\nu_0, \phi_0, \hat{\alpha}_n \right)
+ \left( A_n\tr \nabla^2 \Lset_n \left(\nu_0, \phi_0, \hat{\alpha}_n \right) A_n + o_{\PP}(1)\right)
A_n^{-1} \begin{pmatrix}
\hat{\nu}_n - \nu_0   \\
\hat{\phi}_n - \phi_0 \\
\end{pmatrix},
$$
where the preceding $\Ocal_{\PP}$-term has been reformulated using a few algebraic manipulations.
(Use the fact that~$\ns{A_n} \lesssim \ln(n)$.)
Multiply by $\sqrt{2n}$ and use Slutsky's lemma to conclude.
\end{proof}

\subsection{Proofs of technical lemmas for Theorem~\thmasymptnorm}\label{sec:asympt_norm_tech_lemmas}
Remember
(see Appendix~\seccirculant and Appendix~\secmorenotations)
that the~$\lambda_{m, n}$s
are smooth functions of $\nu$ and $\alpha$.
Thus, the function $\Lset_n$ is smooth for any realization
and can be written as:
$$
\Lset_n \left(\nu, \, \phi, \alpha\right) =
\ln (\phi)
+ \frac{1}{n}\sum_{m=0}^{n-1} \ln \left( \lambda_{m, n} \right) 
+ \frac{\phi_0}{n\phi} \sum_{m=0}^{n-1} \frac{\lambda_{m, n}^{(0)} U_{m, n}^2}{\lambda_{m, n}}.
$$
Expressions for some derivatives are given in the following. 
These expressions are cumbersome, but rough approximations will suffice:
we only need to ensure the $\partial^p \lambda_{m, n} / \partial \nu^p$s
do not grow too fast compared to $\lambda_{m, n}$.
The first-order derivative with respect to $\nu$ writes:
$$
\frac{\partial \Lset_n }{\partial \nu} \left(\nu, \, \phi, \alpha\right) = 
\frac{1}{n}\sum_{m=0}^{n-1} \frac{
 \partial \lambda_{m, n}  / \partial \nu
}{
\lambda_{m, n} 
}
- \frac{\phi_0}{n\phi} \sum_{m=0}^{n-1}
 \frac{U_{m, n}^2 \lambda_{m, n}^{(0)}  \partial \lambda_{m, n}  / \partial \nu }{\lambda_{m, n}^2}.
$$
Then, the second-order derivative with respect to $\nu$ writes:
$$
\frac{\partial^2 \Lset_n }{\partial \nu^2} \left(\nu, \, \phi, \alpha\right) = 
\frac{1}{n}\sum_{m=0}^{n-1} \frac{
\lambda_{m, n} 
\partial^2 \lambda_{m, n}  / \partial \nu^2
- \left( \partial \lambda_{m, n}  / \partial \nu
 \right)^2
}{
\lambda_{m, n}^2
} 
$$
$$
-\frac{\phi_0}{n\phi}
\sum_{m=0}^{n-1} \frac{U_{m, n}^2 \lambda_{m, n}^{(0)}
\left( \partial^2 \lambda_{m, n}  / \partial \nu^2
 \lambda_{m, n} - 2 \left( \partial \lambda_{m, n}  / \partial \nu
 \right)^2 \right)
}{\lambda_{m, n}^3}.
$$
Finally, the third-order derivative with respect to $\nu$ writes:
\begin{align*}
& \frac{\partial^3 \Lset_n}{\partial \nu^3} \left(\nu, \, \phi, \alpha\right) \\
& \, = \, \frac{1}{n} \sum_{m = 0}^{n - 1} \lambda_{m, n}^{-3} \left(
\frac{\partial^3 \lambda_{m, n}}{\partial \nu^3} \lambda_{m, n}^2
- 3 \frac{\partial^2 \lambda_{m, n}}{\partial \nu^2} 
\frac{\partial \lambda_{m, n}}{\partial \nu}
\lambda_{m, n} 
+ 2
\left( \frac{\partial \lambda_{m, n}}{\partial \nu} \right)^3
\right) \\
& \, - \, \frac{\phi_0}{n\phi}
\sum_{m=0}^{n-1} \lambda_{m, n}^{-4} \lambda_{m, n}^{(0)}
\left(
\frac{\partial^3 \lambda_{m, n}}{\partial \nu^3} \lambda_{m, n}^2 \right. \\
& \, - \, \left. 4 \frac{\partial^2 \lambda_{m, n}}{\partial \nu^2}  \frac{\partial \lambda_{m, n}}{\partial \nu}
\lambda_{m, n} 
+   6 
\left( \frac{\partial \lambda_{m, n}  }{\partial \nu }
 \right)^3
 \right)
U_{m, n}^2.
\end{align*}

Bounding all terms independently
will suffice for our purposes.
The necessary approximations are given by
Lemma~\lemroughboundc and the following.
Exceptionally, the arguments of the $\lambda_{m, n}$s are not dropped.
\begin{lem}\label{lem:rough_bound_derivative}
Let $ 0 < \delta < 2 \nu_{\min} $, $0 \leq m \leq \floor*{n/2}$,
$\nu \in N$, $\alpha \in A$ and $p \in \{1, 2, 3\}$. We have:
$$
\frac{1}{n} \left|
\frac{\partial^p \lambda_{m, n}  }{\partial \nu^p } \left(\nu, \, \alpha \right)
\right| \lesssim \frac{1}{m^{2\nu + 1 - \delta}}, \quad \mathrm{if} \ 1 \leq m \leq \floor*{n/2}
$$
and
$$
\frac{1}{n} \left|
\frac{\partial^p \lambda_{0, n}  }{\partial \nu^p } \left(\nu, \, \alpha \right)
\right| \lesssim 1,
$$
uniformly in $m$, $\nu$, and $\alpha$.
\end{lem}
\begin{proof}
We have
\begin{align*}
\frac{1}{n}
\left| \frac{\partial^p \lambda_{m, n}  }{\partial \nu^p } \left(\nu, \, \alpha \right) \right| 
& \leq
\sum_{j \in \Zm} \frac{
\left| \ln^p \left( \alpha^2 + (m + jn)^2 \right) \right|
}{
\left( \alpha^2 + (m + jn)^2 \right)^{\nu + 1/2}
} \\
& \lesssim \sum_{j \in \Zm} \frac{
1
}{
\left( \alpha^2 + (m + jn)^2 \right)^{\nu + 1/2 - \delta/2}
},
\end{align*}%
which equals $n^{-1} \lambda_{m, n} \left( \nu - \delta/2, \, \alpha\right)$, so
Lemma~\lemroughboundc gives the result.
(Adjust the lower bound of $N$ if needed.)
\end{proof}

\begin{lem}\label{lem:ratio_der}
Let $A \subset \left(0, \, + \infty\right)$ be a compact interval and $\nu_0 > 0$.
It holds that
$$
\frac{
\partial \lambda_{m, n} / \partial \nu  \left(\nu_0, \alpha \right)
}{
\lambda_{m, n} \left(\nu_0, \alpha \right)
}
=
- 2\ln(n)  - 2\psi_{\nu_0}(m/n)
+  \Ocal \left( m^{-2} \ln(n) \right),
$$
uniformly in $\alpha \in A$ and $1 \leq m \leq \floor*{n/2}$, with~$\psi_{\nu_0}$ given by~\eqdefpsi.
\end{lem}
\begin{proof}
We have:
\begin{align*}
& n^{-1} \frac{\partial \lambda_{m, n}} {\partial \nu} \left(\nu_0, \alpha \right)\\
& \, = \,
 - \sum_{j \in \Zm}  \frac{ \ln \left( \alpha^2 + (m + j n)^2 \right)}{\left( \alpha^2 + (m + j n)^2 \right)^{\nu_0 + 1/2}} \\
& \, = \,
- \sum_{j \in \Zm}  \frac{ 2\ln \left| m + j n \right| + \ln \left( \left( \frac{\alpha}{m + j n}\right)^2 + 1 \right)}{\left( \alpha^2 + (m + j n)^2 \right)^{\nu_0 + 1/2}} \\
& \, = \,
- \sum_{j \in \Zm}  \frac{ 2\ln \left| m + j n \right| + \Ocal \left( m^{-2} \right)}{\left( \alpha^2 + (m + j n)^2 \right)^{\nu_0 + 1/2}} \\
& \qquad \qquad \left( \mathrm{uniformly, \ since} \ m \leq n/2 \Rightarrow m \leq \abs{m + n j} \right)
\\
& \, = \,
- \left( 1 + \Ocal \left( m^{-2} \right)  \right) \sum_{j \in \Zm}  \frac{ 2 \ln(n) + 2\ln \left| m/n + j \right|
+ \Ocal \left( m^{-2} \right)
}{ \left| m + j n \right|^{2\nu_0 + 1}} \\
& \qquad \qquad
\left( \mathrm{Lemma \ \lemboundcrho} \right)
\\
& \, = \,
-
\left( 1 + \Ocal \left( m^{-2} \right)  \right)
n^{-2\nu_0 - 1} \gamma \left(2\nu_0 + 1; \, m/n \right) \\
& \, \cdot \, \left( 
2\ln(n)
+
2\psi_{\nu_0} \left( m/n \right)
+ \Ocal \left( m^{-2} \right)
\right).
\end{align*}
Thus, using Lemma~\lemboundcrho again yields
\begin{align*}
\frac{
\frac{\partial \lambda_{m, n} }{ \partial \nu } \left(\nu_0, \alpha \right)
}{
\lambda_{m, n} \left(\nu_0, \alpha \right)
}
& \, = \,
\frac{
- \left( 1 + \Ocal \left( m^{-2} \right)  \right)  \left( 2\ln(n)  + 2\psi_{\nu_0}(m/n)
+  \Ocal \left( m^{-2} \right) \right)
}{
1 + \Ocal \left( m^{-2} \right)
}
\\
& \, = \,
- \left( 1 + \Ocal \left( m^{-2} \right)  \right)  \left( 2\ln(n)  + 2\psi_{\nu_0}(m/n) \right)
+  \Ocal \left( m^{-2} \right).
\end{align*}
Lemmata~\lemroughgammabehaviour and~\lemroughgammaderbehaviour show that $\left| \psi_{\nu_0} \left(m/n\right) \right| \lesssim \ln(n)$.
\end{proof}

\begin{proof}[Proof of Lemma~\lemscorefunction]
Note that the $\lambda_{m,n}$s are random since they depend
on~$\left(\nu_0, \hat{\alpha}_n\right)$.
First, we have:
\begin{align*}
& \frac{\partial \Lset_n}{\partial \phi} (\nu_0, \, \phi_0, \, \hat{\alpha}_n) \\
& \, = \,
\frac{1}{\phi_0 n} \sum_{m=0}^{n-1}  1 - \frac{\lambda_{m, n}^{(0)} U_{m, n}^2}{\lambda_{m, n}}
\\
& \, = \,
\Ocal_{\PP} \left( \frac{1}{n} \right) +
\frac{1}{\phi_0 n} \sum_{m=1}^{n-1}  1 - \frac{\lambda_{m, n}^{(0)} U_{m, n}^2}{\lambda_{m, n}}
\quad \left( \mathrm{Lemma \ \lemroughboundc} \right)
\\
& \, = \,
\Ocal_{\PP} \left( \frac{1}{n^{\beta}} \right) +
\frac{1}{\phi_0 n} \sum_{m=1}^{n-1}  1 - U_{m,n}^2
\quad \left(\mathrm{for \ some} \ \beta > 1/2 \ \mathrm{by \ Lemma~\lemlinkexpnormh} \right).
\end{align*}
Furthermore, one has:
\begin{align*}
\frac{\partial \Lset_n }{\partial \nu} \left(\nu_0, \, \phi_0, \hat{\alpha}_n\right) 
& \, = \, 
\frac{1}{n} \sum_{m=0}^{n-1}
\frac{
 \partial \lambda_{m, n}  / \partial \nu
}{
\lambda_{m, n} 
}
\left( 1 - 
 \frac{U_{m, n}^2 \lambda_{m, n}^{(0)} }{\lambda_{m, n}} \right)
\\
& \, = \, 
\Ocal_{\PP} \left( \frac{1}{n} \right) +
\frac{1}{n} \sum_{m=1}^{n-1}
\frac{
 \partial \lambda_{m, n}  / \partial \nu
}{
\lambda_{m, n} 
}
\left( 1 - 
 \frac{U_{m, n}^2 \lambda_{m, n}^{(0)} }{\lambda_{m, n}} \right) \\
& \qquad \qquad \left( \mathrm{Lemmata \ \lemroughboundc \
and~\ref{lem:rough_bound_derivative}} \right) 
 \\
& \, = \,
\Ocal_{\PP} \left( \frac{1}{n} \right)
+\frac{1}{n} \sum_{m=1}^{n-1}
\frac{
 \partial \lambda_{m, n}  / \partial \nu
}{
\lambda_{m, n} 
}
\left( 1 - 
 U_{m, n}^2 \right)
 \\
 & \, + \,
 \frac{1}{n} \sum_{m=1}^{n-1}
\frac{
 \partial \lambda_{m, n}  / \partial \nu
}{
\lambda_{m, n} 
}
U_{m, n}^2  \Ocal \left( m^{-2} \vee \left(n - m\right)^{-2} \right) \\
& \qquad \qquad \left(\mathrm{essentially, \ by} \ \mathrm{Lemma~\lemhunifalpha} \right)
\\
& \, = \,
\Ocal_{\PP} \left( \frac{1}{n} \right)
+\frac{1}{n} \sum_{m=1}^{n-1}
\frac{
 \partial \lambda_{m, n}  / \partial \nu
}{
\lambda_{m, n} 
}
\left( 1 - 
 U_{m, n}^2 \right)
\end{align*}
since
$ \partial \lambda_{m, n}  / \partial \nu \left(\nu_0, \hat{\alpha}_n\right)
\lesssim m^{\delta} \wedge \left(n - m\right)^{\delta} \lambda_{m, n} \left(\nu_0, \hat{\alpha}_n\right)$
holds essentially,
thanks to Lemmata~\lemroughboundc
and~\ref{lem:rough_bound_derivative}.
(By ``essentially'', we mean that the constant does not depend on the sample path.)
Then, using Lemma~\ref{lem:ratio_der} leads to:
\begin{align*}
\frac{\partial \Lset_n }{\partial \nu} \left(\nu_0, \, \phi_0, \hat{\alpha}_n\right) &
=
- \frac{2\ln(n)}{n} \sum_{m=1}^{n-1}
\left( 1 - 
 U_{m, n}^2 \right) \\
&  - \frac{2}{n} \sum_{m=1}^{n-1}
\psi_{\nu_0}(m/n)
\left( 1 - 
 U_{m, n}^2 \right)
+ \Ocal_{\PP} \left( \frac{\ln(n)}{n} \right),
\end{align*}
and subsequent calculations show that
$A_n\tr
\nabla \Lset_n \left(\nu_0, \, \phi_0, \hat{\alpha}_n\right)$
equals:
$$
o_{\PP} \left( \frac{1}{\sqrt{n}} \right) +
 \frac{2}{n \sqrt{\mathrm{Var}  \left( \psi_{\nu_0} (V) \right)}} \sum_{m=1}^{n-1}   \left( 1 - 
 U_{m, n}^2 \right) 
\begin{pmatrix}
\EE \left(  \psi_{\nu_0}(V) \right) - \psi_{\nu_0}(m/n)\\
\sqrt{\mathrm{Var}  \left( \psi_{\nu_0} (V) \right)}
\end{pmatrix}.
$$
Conclude using a standard Lindeberg-Feller argument.
(Lemmata~\lemroughgammabehaviour
and~\lemroughgammaderbehaviour
give (a multiple of) the envelope~$x \mapsto -\ln(x)$
near zero for~$\psi_{\nu_0}$.
Proceed as for Lemma~\lemLtwodist to show
that $n^{-1} \sum_{m=1}^{n-1} \psi_{\nu_0}^{p} (m/n) \to \int_0^1 \psi_{\nu_0}^{p}$.)
\end{proof}

\begin{proof}[Proof of Lemma~\lemfishermatrix]
Observe that $A_n\tr C_n A_n = 2 I_2$, with the symmetric matrix:
$$
C_n = \begin{pmatrix}
2 \ln^2(n) + 4 \ln(n) \EE \left( \psi_{\nu_0} (V) \right)  + 2 \EE \left( \psi_{\nu_0}^2 (V) \right)  & \ \  -     \\[7pt]
- \ln(n) \phi_0^{-1} - \EE \left( \psi_{\nu_0} (V) \right) \phi_0^{-1}  & 2^{-1} \phi_0^{-2} \\
\end{pmatrix}.
$$%

The rest of the proof is left to the reader. It consists in showing that
$$\nabla^2 \Lset_n(\nu_0, \, \phi_0, \, \hat{\alpha}_n) / 2 = C_n + \Ocal_{\PP} \left( n^{-\epsilon} \right)$$
for some~$\epsilon > 0$ by proceeding as for the proof of Lemma~\lemscorefunction.
\end{proof}

\subsection{Proofs of Theorem~\thmerrornufixed, Theorem~\thmerrornuestimated, and Theorem~\thmerrornufixedzero}\label{sec:proof_error}
The posterior mean does not depend on~$\phi$,
so all derivations will be written with~$\phi = 1$.
Furthermore, we will use the notation~$c_j(\nu, \alpha)$
defined in Appendix~\secnotations.
Also, we assume that $\phi_0 = 1$ without loss of generality.

We avoid dealing with conditionally convergent series
since it is assumed that~$\nu_0 > 1/2$.
In this case, the coefficients of the expansion~\eqxi
are almost surely absolutely summable and
so the hypotheses of Proposition~\propblup are fulfilled.
The random Fourier series converges almost surely in~$L^2 [0, \, 1]$,
and the proofs will rely on using Parseval's identity.

Let $\left(\nu, \, \alpha\right) \in \left(0, \, + \infty\right)^2$
and $j \in \Zm$, we have
\begin{align}
2 \left| c_{j}(\xi - \hat{\xi}_n) \right|^2
= 
\left(
\frac{
c_{j}(\nu, \alpha) \sum_{j_1 \in \Zm \setminus \{ 0 \}} \sqrt{c_{j + j_1 n}(\nu_0, \alpha_0)} U_{1, \abs{j + j_1 n}} 
}{
\sum_{j_1 \in \Zm} c_{j + n j_1}(\nu, \alpha)
}
\right.
\nonumber \\
\left.
-
\frac{
 \sqrt{c_{j}(\nu_0, \alpha_0)} U_{1, \abs{j}} \sum_{j_1 \in \Zm \setminus \{ 0 \}} c_{j + n j_1}(\nu, \alpha)
}{
\sum_{j_1 \in \Zm} c_{j + n j_1}(\nu, \alpha)
}
\right)^2
\nonumber \\
+ \left(
\frac{
c_{j}(\nu, \alpha) \sum_{j_1 \in \Zm \setminus \{ 0 \}} \sqrt{c_{j + n j_1}(\nu_0, \alpha_0)} U_{2, \abs{j + n j_1}} \sign(j + n j_1) 
}{
\sum_{j_1 \in \Zm} c_{j + n j_1}(\nu, \alpha)
}
\right.
\nonumber \\
\left.
-
\frac{
\sqrt{c_{j}(\nu_0, \alpha_0)} U_{2, \abs{j}} \sign(j)  \sum_{j_1 \in \Zm \setminus \{ 0 \}} c_{j + n j_1}(\nu, \alpha)
}{
\sum_{j_1 \in \Zm} c_{j + n j_1}(\nu, \alpha)
}
\right)^2\label{eq:stoch_error_square_module}
\end{align}
after a few algebraic manipulations.
The expression~\eqref{eq:stoch_error_square_module} is a sum of
two independent terms.
Let $m \in \llbracket 0, \, n-1\rrbracket$.
If~$j \in m + n\Zm$ with $m \notin \{0, \, n/2 \}$, then
the two terms are
identically distributed and involve independent Gaussian variables.
Thus, there exists~$\chi_2^2$ distributed variables $A_{m, j, n}$ such that
\begin{equation}\label{eq:chi_square_notation}
\left| c_{m + j n}(\xi - \hat{\xi}_n) \right|^2
= 
a_{m, j, n} \left( \nu, \, \alpha \right) A_{m, j, n}/2
\end{equation}
with
\begin{align}
a_{m, j, n}\left( \nu, \, \alpha \right)
= 
c_{m + j n}^2(\nu, \alpha)
\frac{
\sum_{j_1 \in \Zm } c_{m + n j_1}(\nu_0, \alpha_0) - c_{m + j n}(\nu_0, \alpha_0)
}{
\left( \sum_{j_1 \in \Zm} c_{m + n j_1}(\nu, \alpha) \right)^2
}
\nonumber \\
+ 
c_{m + j n}(\nu_0, \alpha_0)
\left(
1 - \frac{
c_{m + j n}(\nu, \alpha)
}{
\sum_{j_1 \in \Zm} c_{m + n j_1}(\nu, \alpha)
}\right)^2.\label{eq:unsimplified_error_coef}
\end{align}

Lemma~\lemroughboundc
and Lemma~\ref{lem:maj_nj}
make it straightforward to prove Lemma~\ref{lem:error_coeff}.
\begin{lem}\label{lem:error_coeff}
Let $A, N \subset \left(0, \, +\infty\right)$ be compact intervals. It holds that
\begin{equation*}
a_{m, j, n}(\nu)
\lesssim \left( \left| j \right| n \right)^{-4\nu-2} m^{4\nu - 2 \nu_0 + 1} + \left( \left| j \right| n \right)^{-2\nu_0-1}, \quad \mathrm{for} \ j \neq 0,
\end{equation*}
and
\begin{equation*}
a_{m, 0, n}(\nu)
\lesssim n^{- 2 \nu_0 - 1} + m^{4 \nu - 2 \nu_0 + 1} n^{-4\nu-2},
\end{equation*}
uniformly in
$\nu \in N$, $\alpha \in A$,
$j \in \Zm$, and $1 \leq m \leq \floor*{(n-1) / 2}$. 
\end{lem}

\begin{lem}\label{lem:maj_nj}
Let $\nu, \alpha > 0$, $0 \leq m \leq \floor*{n/2}$, and $j \neq 0$. We have:
$$
c_{m + nj}(\nu, \alpha) \leq 2^{2\nu + 1}\left( n \left|j\right| \right)^{-2\nu-1}.
$$
\end{lem}
\begin{proof}
Using the fact that $m \leq n/2$ leads to:
\begin{equation*}
c_{m + nj}(\nu, \alpha)  \leq \left( n \left( \left| j  \right| -  1/2 \right) \right)^{-2\nu-1}
\leq 2^{2\nu + 1}\left( n \left|j\right| \right)^{-2\nu-1}.
\end{equation*}
\end{proof}

For $m \in \left\lbrace 0, \, n/2 \right\rbrace$ and $j \in \Zm$,
the two terms in~\eqref{eq:stoch_error_square_module} are not identically
distributed.
Moreover,
for $q \in \left\{ 1, \, 2 \right\}$ and $m \in \left\lbrace 0, \, n/2 \right\rbrace$,
there are duplicates among the variables $\left\{ U_{q, \left| m + nj \right|}, \, j \in \Zm \right\}$.
Nevertheless, the two terms
are sums of independent Gaussian variables, so
expressions like~\eqref{eq:chi_square_notation} hold.
However,
the presence of duplicates makes
the expressions more
complex than~\eqref{eq:unsimplified_error_coef}.
The upper bounds given by assuming
full redundancy among the variables appearing
in the two terms of~\eqref{eq:stoch_error_square_module}
suffice for our purposes.
The following lemmata are adaptations of Lemma~\ref{lem:error_coeff}.
The statements are
made uniform with respect to regularity ranges
to be used in the proof of Theorem~\thmerrornuestimated.
\begin{lem}\label{lem:error_zero_coef}
Let $N, A \subset \left(0, \, + \infty\right)$
be compact intervals, and write $\nu_{\min} = \min N$.
Then:
$$
\EE \left(
\sup_{\nu \in N, \alpha \in A}
\sum_{j \in \Zm} \left| c_{j n}(\xi - \hat{\xi}_n) \right|^{2} 
\right) \lesssim n^{-2\nu_0 - 1} + n^{- 4 \nu_{\min} - 2}.
$$
\end{lem}

\begin{lem}\label{lem:error_middle_coef}
Let $n \geq 2$ be even and $N, A \subset \left(0, \, + \infty\right)$
be compact intervals. Then:
$$
\EE \left( \sup_{\nu \in N, \alpha \in A} \sum_{j \in \Zm} \left| c_{n/2 + j n}(\xi - \hat{\xi}_n) \right|^2
\right) \lesssim n^{-2\nu_0-1}.
$$
\end{lem}

\begin{proof}[Proof of Theorem~\thmerrornufixed]
We prove the (more general) result with~$\hat{\alpha}_n \in A$,
for a compact interval $A$. This will be useful for proving
Theorem~\thmerrornuestimated.

Let $m \in \llbracket 0, \, n-1\rrbracket$ such that $m \notin \{ 0, \ n/2 \}$
and consider indexes~$m + n j$, with~$j \in \Zm$.
Lemma~\ref{lem:error_coeff} and~\eqref{eq:chi_square_notation} yields:
\begin{equation}\label{eq:sum_error_coeff}
\sum_{j \in \Zm} \EE \left( \left| c_{m + j n}(\xi - \hat{\xi}_n) \right|^2\right)
\lesssim n^{-2\nu_0-1} + n^{-4\nu-2} m^{4\nu - 2 \nu_0 + 1}.
\end{equation}
The first two statements then follow from
Lemmata~\ref{lem:error_zero_coef}
and~\ref{lem:error_middle_coef}, the identity
\begin{equation}\label{eq:sym_error_coefs}
\sum_{j \in \Zm}  \left| c_{m + j n}(\xi - \hat{\xi}_n) \right|^2 = \sum_{j \in \Zm}  \left| c_{n - m + j n}(\xi - \hat{\xi}_n) \right|^2,
\end{equation}
for every $0 \leq m \leq n - 1$, the Fubini-Tonelli thereom, and Parseval's identity.

For the last statement, let $\nu > (\nu_0 - 1)/2$ and $1 \leq m \leq l$ with $l = \floor*{(n-1) / 2}$.
Lemma~\lemboundcrho gives
$$
a_{m, j, n} \left( \nu, \, \hat{\alpha}_n \right) 
= 
\left( 1 + \Ocal\left( m^{-2} \right) \right) \left( \left| m + j n \right|^{-4\nu - 2}
\frac{
\sum_{j_1 \in \Zm \setminus \{j\} } \left| m + j_1 n \right|^{-2\nu_0 - 1}
}{
\left( \sum_{j_1 \in \Zm} \left| m + j_1 n \right|^{-2\nu -1} \right)^2
}
\right.
$$
$$
\left.
+ 
\left| m + j n \right|^{-2\nu_0 - 1}
\left(
\frac{
\sum_{j_1 \in \Zm \setminus \{j\}} \left| m + j_1 n \right|^{-2\nu-1}
}{
\sum_{j_1 \in \Zm} \left| m + j_1 n \right|^{-2\nu - 1}
}\right)^2 \right),
$$
for every $j \in \Zm$, essentially.
Consequently, it holds that:
$$
\sum_{j \in \Zm}  \EE \left( \left| c_{m + j n}(\xi - \hat{\xi}_n) \right|^2 \right)
 = \frac{ \left( 1 + \Ocal(m^{-2}) \right)}{
n^{2\nu_0 + 1}
} \vartheta_{\nu; \nu_0}(m/n)
$$
after a few algebraic manipulations.
Using the definition of~$\gamma$, it is straightforward
to show that
\begin{equation}\label{eq:dev_zero_vartheta}
\vartheta_{\nu; \nu_0}(x) \sim C_1 x^{4\nu -2\nu_0 + 1} + C_2
\end{equation}
for some nonzero constants $C_1, C_2$, when $x \to 0$.
Therefore, the function $\vartheta_{\nu; \nu_0}$ is integrable if~$\nu > (\nu_0 - 1)/2$
and\footnote{
Proceed as for Lemma~\lemLtwodist, using~\eqref{eq:dev_zero_vartheta}, 
if $(\nu_0 - 1)/2 < \nu < (\nu_0 - 1/2)/2$.
}
\begin{equation}\label{eq:int_bcal}
\frac{1}{n} \sum_{m = 1}^{l} \vartheta_{\nu; \nu_0}(m/n) \to \int_0^{1/2} \vartheta_{\nu; \nu_0}.
\end{equation}
Then,
Lemma~\ref{lem:error_middle_coef},
Lemma~\ref{lem:error_zero_coef}, the identity~\eqref{eq:sym_error_coefs},
the Fubini-Tonelli thereom, and Parseval's identity give
$$
n^{2\nu_0} \EE \left( \mathrm{ISE}_n \left(\nu, \, \hat{\alpha}_n ; \, \xi \right) \right) = 
o(1) + \frac{2}{n} \sum_{m = 1}^l \left( 1 + \Ocal(m^{-2}) \right) \vartheta_{\nu; \nu_0}(m/n)
\to \int_0^1 \vartheta_{\nu; \nu_0},
$$%
killing the $\Ocal(m^{-2})$-term
using Hölder inequality and~\eqref{eq:dev_zero_vartheta}
as in the proof of Lemma~\lemlinkexpnormh.
\end{proof}

\begin{proof}[Proof of Theorem~\thmerrornufixedzero]
We can assume that~$\phi = 1$ without loss of generality.
Then, using the framework introduced in Section~\secframework,
it can be seen that inferring~$\xi(j/n)$ given~$\{\xi(p/n), \, p \neq j\}$ is the same problem for all~$j$
due to symmetry. Consequently,
the expectation of the squared error at one location equals the expectation of
the leave-one-out (mean) squared error which is given by~\citep{Craven1978SMOOTHINGND}
$$
\EE \left( \left( \hat{\xi}_n^{(0)}(0) - \xi(0) \right)^2 \right)
= \EE \left( \frac{n^{-1} Z\tr R_{\nu, \alpha}^{-2} Z}{\mathrm{Tr}^2 \left( n^{-1} R_{\nu, \alpha}^{-1} \right)} \right)
= \phi_0 \frac{
n^{-1}\sum_{m=0}^{n-1} \lambda_{m, n}^{(0)} / \lambda_{m, n}^2
}{
\left( n^{-1}\sum_{m=0}^{n-1} \lambda_{m, n}^{-1} \right)^2
}.
$$
To conclude, study the two sums independently
using Lemmata~\lemroughboundc and~\lemboundcrho.
\end{proof}

The following lemma bounds the rate at which
$\hat{\nu}_n$ falls within the interval~$\left[\nu_0 - 1/2, \, \nu_{\max}\right]$ of values
giving reproducing kernel Hilbert spaces
almost surely not containing~$\xi$. It will be useful for proving Theorem~\thmerrornuestimated.
\begin{lem}\label{lem:rate_excursion_rkhs}
Let $\epsilon > 0$. With the notations of Theorem~\thmerrornuestimated, we have:
$$
\PP \left( \hat{\nu_n}  \leq \nu_0 - 1/2 - \epsilon \right) \lesssim e^{-C\sqrt{n}},
$$
for some $C > 0$.
\end{lem}
\begin{proof}
Let $\alpha_1$ be any element of $A$.
We proceed by bounding
\begin{equation*} 
\PP \left( 
\inf_{\nu_{\min} \leq \nu \leq \nu_0 - 1/2 - \epsilon, \alpha \in A} \Mset_n \left( \nu, \alpha \right) 
- \Mset_n \left( \nu_0, \, \alpha_1 \right) 
\leq
0
\right).
\end{equation*}
Then, let $\alpha \in A$ and $\nu_{\min} \leq \nu \leq \nu_0 - 1/2 - \epsilon$, we have:
\begin{align*}
\Mset_n \left( \nu, \alpha \right)   & \;=\; 
 \mathcal{O} \left( 1 \right) + \ln \left( \frac{Z\tr R_{\nu, \alpha}^{-1} Z}{n^{1 + 2(\nu - \nu_0)}} \right)
 \quad (\mathrm{Lemma~\lemlogdet})\\
  & \;\geq\;  
 \mathcal{O} \left( 1 \right) + \ln \left( \frac{\sum_{m=1}^{\floor*{\sqrt{n}}}
 U_{m, n}^2 \lambda_{m, n}^{(0)} / \lambda_{m, n} 
 }{n^{1 + 2(\nu - \nu_0)}} \right)\\
  & \;=\;  
 \mathcal{O} \left( 1 \right) + \ln \left( \frac{\sum_{m=1}^{\floor*{\sqrt{n}}} U_{m, n}^2 m^{2(\nu - \nu_0)}}{n^{1 + 2(\nu - \nu_0)}} \right)\quad
 (\mathrm{Lemma \ \lemroughboundc})
 \\
  & \;=\;  
 \mathcal{O} \left( 1 \right) + \ln \left(\frac{1}{n} \sum_{m=1}^{\floor*{\sqrt{n}}} U_{m, n}^2 \left(\frac{m}{n}\right)^{2(\nu - \nu_0)} \right)\\
  & \;\geq\;  
 \mathcal{O} \left( 1 \right) + \ln \left(\frac{1}{n} \sum_{m=1}^{\floor*{\sqrt{n}}} U_{m, n}^2 \left(\frac{m}{n}\right)^{-1 - 2\epsilon} \right)\\
  & \;=\;  
 \mathcal{O} \left( 1 \right) +2\epsilon \ln(n) + \ln \left(\sum_{m=1}^{\floor*{\sqrt{n}}} U_{m, n}^2 m^{-1 - 2\epsilon} \right)\\
  & \;\geq\;  
 \mathcal{O} \left( 1 \right) +2\epsilon \ln(n) + \ln \left(\sum_{m=1}^{\floor*{\sqrt{n}}} U_{m, n}^2 \floor*{\sqrt{n}}^{-1 - 2\epsilon} \right)\\
  & \;\geq\;  
 \mathcal{O} \left( 1 \right) + \epsilon \ln(n) + \ln \left( \frac{1}{\floor*{\sqrt{n}}} \sum_{m=1}^{\floor*{\sqrt{n}}} U_{m, n}^2 \right)\\
  & \;\geq\;  
 \mathcal{O} \left( 1 \right) + \epsilon \ln(n) + \frac{1}{\floor*{\sqrt{n}}} \sum_{m=1}^{\floor*{\sqrt{n}}} \ln \left( U_{m, n}^2 \right)\quad
 (\mathrm{Jensen \ inequality})
 \end{align*}
with a uniform big-$\Ocal$.
Let $\delta > 0$ and $t > 0$, we have
\begin{align*}
\PP \left( - \frac{1}{\floor*{\sqrt{n}}}  \sum_{m = 1}^{\floor*{\sqrt{n}}}  \ln(U_{m, n}^2) \geq \delta \right)
& \;=\; 
\PP \left( e^{- \frac{t}{\floor*{\sqrt{n}}}  \sum_{m = 1}^{\floor*{\sqrt{n}}}  \ln(U_{m, n}^2)}
\geq e^{t \delta}
 \right)\\
  & \;\leq\;  
e^{- \delta \floor*{\sqrt{n}}/4} \EE \left( \left| U_{1, 1} \right|^{-1/2} \right)^{\floor*{\sqrt{n}}},
\end{align*}
with $t = 1/4$ and $\EE \left( \left| U_{1, 1} \right|^{-1/2} \right) < + \infty$.
This gives the desired convergence rate if $\delta$ is high enough.

Furthermore, we have
$$
\Mset_n \left( \nu_0, \alpha_1 \right)
=
\Ocal \left( 1 \right) + \ln \left( n^{-1} \sum_{m = 0}^{n-1} U_{m,n}^2  \right),
$$
and
$$
\PP \left( \ln \left( n^{-1} \sum_{m = 0}^{n-1} U_{m,n}^2 \right) \geq \delta \right)
\leq e^{-C_2 n},
$$%
for some $C_2 > 0$ if $\delta > 0$ is high enough,
using also a Chernoff bound argument.
Now, putting all the pieces together yields:
$$
\inf_{\nu_{\min} \leq \nu \leq \nu_0 - 1/2 - \epsilon, \alpha \in A} \Mset_n \left( \nu, \alpha \right) 
- \Mset_n \left( \nu_0, \, \alpha_1 \right)
$$
$$
\geq 
\Ocal(1)
+
\epsilon \ln(n)
+ \frac{1}{\floor*{\sqrt{n}}} \sum_{m=1}^{\floor*{\sqrt{n}}} \ln \left( U_{m, n}^2 \right)
- \ln \left( n^{-1} \sum_{m = 0}^{n-1} U_{m,n}^2 \right) 
$$%
giving the result thanks to the pigeonhole principle.
\end{proof}

\begin{proof}[Proof of Theorem~\thmerrornuestimated]
The proof of Theorem~\thmerrornufixed already deals with an estimated parameter~$\hat{\alpha}_n \in A$.
It is extended to estimated~$\hat{\nu}_n \in N$ by bounding 
derivatives and using Lemma~\ref{lem:rate_excursion_rkhs}.

Let $\epsilon > 0$ and $1 \leq m \leq l = \floor*{(n-1) / 2}$ and
use the notation~\eqref{eq:unsimplified_error_coef}.
The functions~$a_{m, j, n}$ are smooth.
For
any (fixed) $0 < \delta < \nu_{\min}$,
it holds that 
$$
\left| \frac{\partial c_{j}}{\partial \nu} (\nu, \, \alpha) \right| \lesssim c_j(\nu-\delta, \, \alpha),
$$
uniformly in $\nu \in \left[\nu_0 - 1/2 - \epsilon, \, \nu_{\max} \right]$,
$\alpha \in A$, and $j \in \Zm$.
Coordination with
Lemmata~\lemroughboundc and~\ref{lem:maj_nj}
makes it possible to show that
\begin{align*}
\left| \frac{\partial a_{m, 0, n}}{\partial \nu} \left( \nu, \, \alpha \right) \right|
& \;\lesssim\; 
n^{-2\nu_0 - 1} m^{2\delta}
+
\frac{m^{- 2\nu_0 + 1}}{n^{2 - 2\delta}}
\left( \frac{m}{n} \right)^{4\nu} \\
& \;\leq\; 
n^{-2\nu_0 - 1} m^{2\delta}
+
\frac{m^{- 2\nu_0 + 1}}{n^{2 - 2\delta}}
\left( \frac{m}{n} \right)^{4\nu_0 - 2 - 4\epsilon} \\
& \;=\; 
n^{-2\nu_0 - 1} m^{2\delta}
+
\frac{m^{2\nu_0 - 1 - 4\epsilon}}{n^{4\nu_0 - 2\delta - 4\epsilon}}\\
& \;\leq\; 
n^{-2\nu_0 - 1} m^{2\delta}
+
\frac{m^{- 1 - 4\epsilon}}{n^{2\nu_0 - 2\delta - 4\epsilon}}
\end{align*}
and, for $j \neq 0$, that
\begin{align*}
& \left| \frac{\partial a_{m, j, n}}{\partial \nu} \left( \nu, \, \alpha \right) \right| \\
& \;\lesssim\; 
\left( \left| j \right| n \right)^{-4\nu - 2 + 2\delta} m^{4\nu - 2\nu_0 + 1 + 2\delta}
+ \left( \left| j \right| n \right)^{-2\nu_0 - 1 + 2\delta}\\
& \;=\; 
\left| j \right|^{-4\nu - 2 + 2\delta}
n^{- 2 + 2\delta}
\left( \frac{m}{n} \right)^{4\nu}
m^{- 2\nu_0 + 1 + 2\delta}
+ \left( \left| j \right| n \right)^{-2\nu_0 - 1 + 2\delta}\\
& \;\leq\; 
\left| j \right|^{- 2 + 2\delta}
n^{- 2 + 2\delta}
\left( \frac{m}{n} \right)^{4\nu_0 - 2 + 4\epsilon}
m^{- 2\nu_0 + 1 + 2\delta}
+ \left( \left| j \right| n \right)^{-2\nu_0 - 1 + 2\delta}\\
& \;=\; 
\left| j \right|^{- 2 + 2\delta}
n^{-4\nu_0 + 2\delta - 4\epsilon}
m^{2\nu_0 - 1 + 4\epsilon + 2\delta}
+ \left( \left| j \right| n \right)^{-2\nu_0 - 1 + 2\delta}
\end{align*}
uniformly in
$1 \leq m \leq l$, $j \neq 0$, $\alpha \in A$, and
$\nu \in \left[\nu_0 - 1/2 - \epsilon, \, \nu_{\max} \right]$.
Then,
\begin{align*}
& \sum_{m = 1}^l \sum_{j \in \Zm}
\EE \left( 
A_{m, j, n} \left|
a_{m, j, n} \left( \hat{\nu}_n, \, \hat{\alpha}_n \right)
-
a_{m, j, n} \left( \nu_0, \, \hat{\alpha}_n \right)
\right|
\one_{ \hat{\nu}_n \geq \nu_0 - 1/2 - \epsilon}
\right) \\
& \, \leq \, 
\sum_{m = 1}^l \sum_{j \in \Zm}
\EE \left( 
A_{m, j, n} 
\left| \hat{\nu}_n - \nu_0 \right|
\sup_{\nu_0 - 1/2 - \epsilon \leq \nu \leq \nu_{\max}, \alpha \in A}
\left| \frac{\partial a_{m, j, n}}{\partial \nu} \left( \nu, \, \alpha \right) \right|
\right) \\
& \, =\,
\sqrt{\EE \left( 
A_{1, 0, 1}^2 \right) }
\sqrt{ \EE \left( \left( \hat{\nu}_n - \nu_0
\right)^2 \right) }
\sum_{m = 1}^l \sum_{j \in \Zm}
\sup_{\nu_0 - 1/2 - \epsilon \leq \nu \leq \nu_{\max}, \alpha \in A}
\left| \frac{\partial a_{m, j, n}}{\partial \nu} \left( \nu, \, \alpha \right) \right| \\
& \, = \, o \left( n^{-2\nu_0} \right),
\end{align*}%
for $\delta$ and $\epsilon$ small enough and
using the above inequalities and Theorem~\thmasymptnorm.
Therefore, Lemmata~\ref{lem:error_zero_coef} and~\ref{lem:error_middle_coef},
the identity~\eqref{eq:sym_error_coefs}, and the Fubini-Tonelli theorem show that
$$
\EE \left(
\left|
\mathrm{ISE}_n \left(\hat{\nu}_n, \, \hat{\alpha}_n; \, \xi\right)
-
\mathrm{ISE}_n \left(\nu_0, \, \hat{\alpha}_n; \, \xi\right)
\right|
\one_{ \hat{\nu}_n \geq \nu_0 - 1/2 -\epsilon}
\right) = o \left( n^{-2\nu_0} \right).
$$

Furthermore, using again the Fubini-Tonelli theorem yields
\begin{align*}
  & \; \; \EE \left(
\sum_{m=1}^l \sum_{j \in \Zm} 
\left| c_{m + j n}(\xi - \hat{\xi}_n) \right|^2
\one_{ \hat{\nu}_n \leq \nu_0 - 1/2 - \epsilon}
\right)\\
  & \;=\; 
\sum_{m=1}^l \sum_{j \in \Zm}  \EE \left( 
a_{m, j, n} \left( \hat{\nu}_n, \, \hat{\alpha}_n \right) A_{m, j, n} \one_{ \hat{\nu}_n \leq \nu_0 - 1/2 -\epsilon}
\right)\\
  & \;\leq\;  
 \sum_{m=1}^l \sum_{j \in \Zm}  \sup_{ \nu_{\min} \leq \nu \leq \nu_0 - 1/2 -\epsilon, \alpha \in A} a_{m, j, n}\left( \nu, \, \alpha \right)  \EE \left( 
A_{m, j, n} \one_{ \hat{\nu}_n \leq \nu_0 - 1/2 - \epsilon}
\right)\\
  & \;\leq\;  
\sqrt{ \EE \left( 
A_{1, 0, 1}^2 \right) }  \sqrt{ \EE \left( \one_{ \hat{\nu}_n \leq \nu_0 - 1/2 - \epsilon} \right) } 
\sum_{m=1}^l \sum_{j \in \Zm}  \sup_{ \nu_{\min} \leq \nu \leq \nu_0 - 1/2 - \epsilon, \alpha \in A} a_{m, j, n}\left( \nu, \, \alpha \right)\\
  & \;\leq\;  
\sqrt{ \EE \left( 
A_{1, 0, 1}^2 \right) }  \sqrt{ \EE \left( \one_{ \hat{\nu}_n \leq \nu_0 - 1/2 - \epsilon} \right) }
\
n^{\beta} \quad \mathrm{for \ some} \ \beta \ \mathrm{given \ by \ Lemma}
\ \ref{lem:error_coeff}\\
  & \;=\;  
o \left( n^{-2\nu_0} \right),
\end{align*}
using Lemma~\ref{lem:rate_excursion_rkhs}.
Then,
the sum for $j \equiv 0 \pmod n$ can be bounded similarly using Lemma~\ref{lem:error_zero_coef} and
the sum for $j \equiv n/2 \pmod n$ is controlled by Lemma~\ref{lem:error_middle_coef} for $n$ even.

Finally, the previous reasoning is easily applied to bound 
$$\EE \left(
\mathrm{ISE}_n \left(\nu_0, \, \hat{\alpha}_n; \, \xi\right)
\one_{ \hat{\nu}_n \leq \nu_0 - 1/2 - \epsilon}
\right)$$%
and the desired result follows.
\end{proof}

\subsection{Proofs of Section~\secfreq}\label{sec:proof_deterministic}

Note that the finiteness of $\nu_0(f)$ is assumed so that $f$ is necessarily nonzero.
Consequently, the data vector $Z$ is ultimately nonzero
under the observation model~\eqobs since $f$ is continuous.
Furthermore, we assume that~$\nu_0(f) > 1$, so $f \in H^{\beta} \left[0, \, 1\right]$
for some $\beta > 1$. Consequently,
the Sobolev embedding theorem implies that~$f$ has Hölder regularity strictly greater than $1/2$.
Hence, $f$ has absolutely summable Fourier coefficients.

The proofs are based on the observation that
$$
Z\tr R_{\nu, \alpha}^{-1} Z
=
\sum_{m=0}^{n-1}
 \frac{\left| \sum_{j \in m + n\Zm} c_j(f)  \right|^2}{\sum_{j \in m + n\Zm}  c_j(\nu, \alpha) },
$$
using~\eqestimationcoeff and elements from Appendix~\seccirculant.

\begin{proof}[Proof of Proposition~\propexcursion]
Let $\epsilon > 0$, $\nu_{\min} \leq \nu \leq \nu_0(f) - 1/2 - \epsilon$, and $\alpha \in A$.
For~$\nu > 0$, $\phi = 1$, and $\alpha > 0$, the reproducing kernel Hilbert
space~$\mathcal{H}_{\nu, \alpha}$ attached to the covariance function is
\begin{equation*}
\mathcal{H}_{\nu, \alpha} = \left\lbrace g \in L^2 \left[0, \, 1\right], \ \ns{g}_{\mathcal{H}_{\nu, \alpha}}^2 =  \sum_{j \in \Zm} (\alpha^2 + j^2)^{\nu + 1/2} \abs{c_j(g)}^2 < +\infty \right\rbrace.
\end{equation*}
From this, it is easy to see that~$\mathcal{H}_{\nu, \alpha}$ is norm-equivalent to $H^{\nu + 1/2}\left[0, \, 1\right]$.
Furthermore, the quadratic form~$Z\tr R_{\nu, \alpha}^{-1} Z$
is the squared $\mathcal{H}_{\nu, \alpha}$-norm of the predictor.
Since~$f$ is continuous and non-zero, then, eventually, for each~$n$, there exists~$ 0 \leq j \leq n -1$
such that~$f(j/n) \neq 0$.
Then, Proposition~\propprofilelik and Lemma~\lemlogdet give
\begin{align*}
  \Mset_n^f \left(\nu, \alpha \right) & \;=\; 
2 (\nu_0(f) - \nu - 1/2) \ln(n) + \Ocal \left( 1 \right)
+ \ln \left( Z\tr R_{\nu, \alpha}^{-1} Z \right)
 \\
  & \;\geq\; 
2 \epsilon \ln(n) + \Ocal \left( 1 \right)
+ \ln \left(  \frac{ \max_{0 \leq j \leq n - 1} f^2(j/n)}{\sum_{j \in \Zm} c_j(\nu, \alpha)}\right),
\end{align*}
where we used the fact that the predictor is the minimum $\mathcal{H}_{\nu, \alpha}$-norm interpolating function from
the RKHS on~$\{ p/n, \, 0 \leq p \leq n - 1\}$,
which has larger norm than the minimum-norm interpolating function on the argmax.
The term inside the logarithm is ultimately uniformly bounded away from zero on~$N \times A$ by continuity.

Moreover, for $\nu = \nu_0(f) - 1/2 - \epsilon/2$ and any
fixed $\alpha \in A$, we have:
\begin{align*}
  \Mset_n^f \left(\nu_0(f) - 1/2 - \epsilon/2, \alpha \right) & \;=\; 
\epsilon \ln(n) + \Ocal \left( 1 \right)
+ \ln \left( Z\tr R_{\nu, \alpha}^{-1} Z\right).
\end{align*}
We have $f \in H^{\beta}\left[0, \, 1\right]$ for $\beta = \nu_0(f) - \epsilon/2$,
and thus~$f \in \mathcal{H}_{\nu, \alpha}$ by norm-equivalence.
In this case, the quadratic form~$Z\tr R_{\nu, \alpha}^{-1} Z$ is the squared norm of a projection of $f$
in~$\mathcal{H}_{\nu, \alpha}$ and is thus bounded.
This completes the proof.
\end{proof}

\begin{proof}[Proof of Proposition~\propdevnllmisp]
Without loss of generality,
consider a compact
subset of the form~$N \times A$ with $A = \left[ \alpha_{\min}, \, \alpha_{\max} \right]$ and
$N = \left[\nu_0(f) - 1/2 + \epsilon, \, \nu_{\max}\right]$,
for some~$\epsilon > 0$.
Then, Proposition~\propprofilelik and Lemma~\lemlogdet yield:
\begin{align*}
\Mset_n^f \left(\nu, \alpha \right) &
\;=\; 
\int_0^1 g_{\nu} + \Ocal \left( \frac{\ln(n)}{n} \right) \\
& \, + \, \ln \left( n^{2(\nu_0(f) - \nu - 1/2)} \sum_{m=0}^{n-1}
 \frac{\left| \sum_{j \in m + n\Zm} c_j(f)  \right|^2}{\sum_{j \in m + n\Zm}  c_j(\nu, \alpha) } \right),
\end{align*}%
with a uniform big-$\Ocal$.
Focus now on the term inside the logarithm.
For $1 \leq m \leq n-1$, Lemma~\lemroughgammabehaviour
shows that
$$\gamma \left(\nu_0(f) + 3/2; \, m/n \right) \approx
n \left( m^{-1} \vee (n - m)^{-1} \right) \gamma \left(\nu_0(f) + 1/2; \, m/n \right).$$
Thus, using the hypothesis on the $c_j(f)$ we have:
\begin{align*}
\sum_{j \in \Zm} c_{jn + m}(f) 
& \, = \, \sum_{j \in \Zm}  \left| jn + m \right|^{-\nu_0(f) - 1/2}
 + \Ocal \left( \sum_{j \in \Zm} \left| jn + m \right|^{-\nu_0(f) - 3/2} \right)\\
& \, = \,
n^{-\nu_0(f) - 1/2} \gamma \left( \nu_0(f) + 1/2; \, m/n\right) \\
& \, + \, \Ocal \left( n^{-\nu_0(f) - 3/2} \gamma \left( \nu_0(f) + 3/2; \, m/n\right) \right)\\
 & \, = \,
n^{-\nu_0(f) - 1/2} \gamma \left( \nu_0(f) + 1/2; \, m/n\right) \left( 1 + \Ocal \left( m^{-1} \vee (n - m)^{-1} \right) \right).
\end{align*}
(It holds that $\sum_{j \in n\Zm} c_j(f) \to c_0(f)$, so the term for $m = 0$ is a uniform big-$\Ocal$ thanks to Lemma~\lemroughboundc.)
Then, use Lemma~\lemboundcrho to get:
\begin{align*}
 & n^{2(\nu_0(f) - \nu - 1/2)} \sum_{m=0}^{n-1}
 \frac{\left| \sum_{j \in m + n\Zm} c_j(f)  \right|^2}{\sum_{j \in m + n\Zm}  c_j(\nu, \alpha) }  \\
 & = 
 \Ocal \left( n^{- 2 \epsilon }\right) + \frac{1}{n} \sum_{m=1}^{n-1}
 \frac{ \left(1 + \Ocal \left( m^{-1} \vee (n - m)^{-1} \right) \right) \gamma^2 \left(\nu_0(f) + 1/2; m/n\right)}{\gamma \left(2\nu + 1; m/n\right) } \\
 & = 
 \Ocal \left( n^{- 2 \epsilon}\right)
+ \Ocal \left( n^{- \epsilon}\right)
+ \frac{1}{n} \sum_{m=1}^{n-1}
 \frac{\gamma^2 \left(\nu_0(f) + 1/2; m/n\right)}{\gamma \left(2\nu + 1; m/n\right) } 
\end{align*}%
using H\"older inequality with $1/p = 1 - \epsilon$, similarly to the proof of Lemma~\lemlinkexpnormh.
The uniform convergence of the Riemann sum is proved similarly to Lemma~\lemLtwodist, using (a multiple of) the
envelope~$x \mapsto x^{2\epsilon} - 1$.
\end{proof}

\begin{acks}[Acknowledgments]
The author thanks Julien Bect and Emmanuel Vazquez for their patience and guidance.
The author is also grateful to two anonymous reviewers whose comments helped to improve
the content and the presentation of this article.
\end{acks}

\bibliographystyle{unsrtnat}
\bibliography{bib_file}

\begin{thebibliography}{46}
\providecommand{\natexlab}[1]{#1}
\providecommand{\url}[1]{\texttt{#1}}
\expandafter\ifx\csname urlstyle\endcsname\relax
  \providecommand{\doi}[1]{doi: #1}\else
  \providecommand{\doi}{doi: \begingroup \urlstyle{rm}\Url}\fi

\bibitem[Stein(1999)]{stein1999interpolation}
M.~L. Stein.
\newblock \emph{Interpolation of spatial data}.
\newblock Springer Series in Statistics. Springer-Verlag, New York, 1999.
\newblock Some theory for Kriging.

\bibitem[Santner et~al.(2003)Santner, Williams, and Notz]{Santner}
T.~J. Santner, B.~J. Williams, and W.~I. Notz.
\newblock \emph{The Design and Analysis of Computer Experiments}, volume~1.
\newblock Springer, 2003.

\bibitem[Rasmussen and Williams(2006)]{Rasmussen}
C.~E. Rasmussen and C.~K.~I. Williams.
\newblock \emph{{G}aussian Processes for Machine Learning}.
\newblock Adaptive Computation and Machine Learning. MIT Press, Cambridge, MA,
  USA, 2006.

\bibitem[Mat\'ern(1986)]{matern1986:_spatial_variations}
B.~Mat\'ern.
\newblock \emph{Spatial Variation}.
\newblock Springer-Verlag New York, 1986.

\bibitem[Bachoc(2021)]{bachoc2021asymptotic}
F.~Bachoc.
\newblock Asymptotic analysis of maximum likelihood estimation of covariance
  parameters for {G}aussian processes: an introduction with proofs.
\newblock In \emph{Advances in Contemporary Statistics and Econometrics}, pages
  283--303. Springer, 2021.

\bibitem[Mardia and Marshall(1984)]{mardia1984:_mle_incr}
K.~V. Mardia and R.~J. Marshall.
\newblock Maximum likelihood estimation of models for residual covariance in
  spatial regression.
\newblock \emph{Biometrika}, 71\penalty0 (1):\penalty0 135--146, 1984.

\bibitem[Bachoc(2014)]{bachoc2014:_asympt}
F.~Bachoc.
\newblock Asymptotic analysis of the role of spatial sampling for covariance
  parameter estimation of {G}aussian processes.
\newblock \emph{Journal of Multivariate Analysis}, 125:\penalty0 1--35, 2014.

\bibitem[Ying(1991)]{ying1991:_ou}
Z.~Ying.
\newblock Asymptotic properties of a maximum likelihood estimator with data
  from a {G}aussian process.
\newblock \emph{Journal of Multivariate analysis}, 36\penalty0 (2):\penalty0
  280--296, 1991.

\bibitem[Ying(1993)]{ying1993:_maximum}
Z.~Ying.
\newblock Maximum likelihood estimation of parameters under a spatial sampling
  scheme.
\newblock \emph{The Annals of Statistics}, pages 1567--1590, 1993.

\bibitem[van~der Vaart(1996)]{vdv96:_ying_correction}
A.~W. van~der Vaart.
\newblock Maximum likelihood estimation under a spatial sampling scheme.
\newblock \emph{The Annals of Statistics}, pages 2049 -- 2057, 1996.

\bibitem[Zhang(2004)]{zhang2004:_micro}
H.~Zhang.
\newblock Inconsistent estimation and asymptotically equal interpolations in
  model-based geostatistics.
\newblock \emph{Journal of the American Statistical Association}, 99\penalty0
  (465):\penalty0 250--261, 2004.

\bibitem[Loh(2006)]{Loh_2006}
W.-L. Loh.
\newblock Fixed-domain asymptotics for a subclass of matérn-type {G}aussian
  random fields.
\newblock \emph{The Annals of Statistics}, 33, 2006.

\bibitem[Kaufman and Shaby(2013)]{kaufman2013:_role}
C.~G. Kaufman and B.~Adam. Shaby.
\newblock The role of the range parameter for estimation and prediction in
  geostatistics.
\newblock \emph{Biometrika}, 100\penalty0 (2):\penalty0 473--484, 2013.

\bibitem[Li(2020)]{li2020bayesian}
C.~Li.
\newblock {B}ayesian fixed-domain asymptotics: {B}ernstein-von {M}ises theorem
  for covariance parameters in a {G}aussian process model.
\newblock \emph{arXiv preprint arXiv:2010.02126}, 2020.

\bibitem[Stein(1993{\natexlab{a}})]{stein93:_splines_estimated_order}
Michael~L. Stein.
\newblock Spline smoothing with an estimated order parameter.
\newblock \emph{The Annals of Statistics}, 21\penalty0 (3):\penalty0
  1522--1544, 1993{\natexlab{a}}.

\bibitem[Chen et~al.(2021)Chen, Owhadi, and Stuart]{chen2021:_consistency}
Y.~Chen, H.~Owhadi, and A.~Stuart.
\newblock {Consistency of empirical Bayes and kernel flow for hierarchical
  parameter estimation}.
\newblock \emph{Mathematics of Computation}, 90\penalty0 (332):\penalty0
  2527--2578, 2021.

\bibitem[Karvonen(2023)]{karvonen2022:_asymptotic}
T.~Karvonen.
\newblock {Asymptotic bounds for smoothness parameter estimates in Gaussian
  process interpolation}.
\newblock \emph{SIAM/ASA Journal on Uncertainty Quantification}, 11\penalty0
  (4):\penalty0 1225--1257, 2023.

\bibitem[Korte-Stapff et~al.(2024)Korte-Stapff, Karvonen, and
  Moulines]{korte2023smoothness}
M.~Korte-Stapff, T.~Karvonen, and E.~Moulines.
\newblock {Smoothness Estimation for Whittle-Matérn Processes on Closed
  Riemannian Manifolds}.
\newblock \emph{arXiv preprint arXiv:2401.00510}, 2024.

\bibitem[Porcu et~al.(2024)Porcu, Bevilacqua, Schaback, and
  Oates]{porcu2024matern}
E.~Porcu, M.~Bevilacqua, R.~Schaback, and C.~J. Oates.
\newblock The mat{\'e}rn model: A journey through statistics, numerical
  analysis and machine learning.
\newblock \emph{Statistical Science}, 39\penalty0 (3):\penalty0 469--492, 2024.

\bibitem[Putter and Young(2001)]{putter_young:_estimated}
H.~Putter and G.~A. Young.
\newblock {On the effect of covariance function estimation on the accuracy of
  kriging predictors}.
\newblock \emph{Bernoulli}, 7\penalty0 (3):\penalty0 421 -- 438, 2001.

\bibitem[Karvonen et~al.(2020)Karvonen, Wynne, Tronarp, Oates, and
  Sarkka]{karvonen2020maximum}
T.~Karvonen, G.~Wynne, F.~Tronarp, C.~Oates, and S.~Sarkka.
\newblock Maximum likelihood estimation and uncertainty quantification for
  {G}aussian process approximation of deterministic functions.
\newblock \emph{SIAM/ASA Journal on Uncertainty Quantification}, 8\penalty0
  (3):\penalty0 926--958, 2020.

\bibitem[Karvonen and Oates(2023)]{karvonen2022:_maximum_lik_ill}
T.~Karvonen and C.~J. Oates.
\newblock Maximum likelihood estimation in gaussian process regression is
  ill-posed.
\newblock \emph{Journal of Machine Learning Research}, 24\penalty0
  (120):\penalty0 1--47, 2023.

\bibitem[Belitser and Ghosal(2003)]{belitser_ghosal2003:_infinite_dim}
E.~Belitser and S.~Ghosal.
\newblock {Adaptive Bayesian inference on the mean of an infinite-dimensional
  normal distribution}.
\newblock \emph{The Annals of Statistics}, 31\penalty0 (2):\penalty0 536 --
  559, 2003.

\bibitem[Knapik et~al.(2016)Knapik, Szab{\'o}, van Der~Vaart, and van
  Zanten]{knapik2016bayes}
B.~T. Knapik, B.~T. Szab{\'o}, A.~W. van Der~Vaart, and J.~H. van Zanten.
\newblock Bayes procedures for adaptive inference in inverse problems for the
  white noise model.
\newblock \emph{Probability Theory and Related Fields}, 164\penalty0
  (3):\penalty0 771--813, 2016.

\bibitem[Szab{\'o} et~al.(2015)Szab{\'o}, van Der~Vaart, and van
  Zanten]{szabo2015frequentist}
B.~Szab{\'o}, A.~W. van Der~Vaart, and J.~H. van Zanten.
\newblock Frequentist coverage of adaptive nonparametric {B}ayesian credible
  sets.
\newblock \emph{The Annals of Statistics}, 43\penalty0 (4):\penalty0
  1391--1428, 2015.

\bibitem[Wahba(1975)]{wahba1975smoothing}
G.~Wahba.
\newblock {Smoothing noisy data with spline functions}.
\newblock \emph{Numerische mathematik}, 24\penalty0 (5):\penalty0 383--393,
  1975.

\bibitem[Wahba(1990)]{Wahba90a}
G.~Wahba.
\newblock \emph{Spline Models for Observational Data}.
\newblock Society for Industrial and Applied Mathematics, Philadelphia, 1990.

\bibitem[Stein(1997)]{stein1997efficiency}
M.~L. Stein.
\newblock {Efficiency of linear predictors for periodic processes using an
  incorrect covariance function}.
\newblock \emph{Journal of statistical planning and inference}, 58\penalty0
  (2):\penalty0 321--331, 1997.

\bibitem[Stein(2014)]{stein_2014:_limitations_low_rank_approximations_for_covariance_matrices_of_spatial_data}
M.~L. Stein.
\newblock Limitations on low rank approximations for covariance matrices of
  spatial data.
\newblock \emph{Spatial Statistics}, 8:\penalty0 1 -- 19, 2014.

\bibitem[Gneiting(2011)]{Gneiting2011:_strictly_positive}
T.~Gneiting.
\newblock Strictly and non-strictly positive definite functions on spheres.
\newblock \emph{Bernoulli}, 19:\penalty0 1327--1349, 2011.

\bibitem[Matheron(1971)]{matheron71}
G.~Matheron.
\newblock The theory of regionalized variables and its applications.
\newblock Technical Report Les cahiers du CMM de Fontainebleau, Fasc. 5, Ecole
  des Mines de Paris, 1971.

\bibitem[Steinwart and Scovel(2012)]{steinwart2012mercer}
I.~Steinwart and C.~Scovel.
\newblock {Mercer’s theorem on general domains: On the interaction between
  measures, kernels, and RKHSs}.
\newblock \emph{Constructive Approximation}, 35:\penalty0 363--417, 2012.

\bibitem[Koltchinskii and Gin{\'e}(2000)]{2000:_kolt_gine_random_matrix}
V.~Koltchinskii and E.~Gin{\'e}.
\newblock {Random matrix approximation of spectra of integral operators}.
\newblock \emph{Bernoulli}, 6\penalty0 (1):\penalty0 113 -- 167, 2000.

\bibitem[Braun(2006)]{braun2006accurate}
M.~L. Braun.
\newblock {Accurate error bounds for the eigenvalues of the kernel matrix}.
\newblock \emph{{The Journal of Machine Learning Research}}, 7:\penalty0
  2303--2328, 2006.

\bibitem[Craven and Wahba(1979)]{Craven1978SMOOTHINGND}
P.~Craven and G.~Wahba.
\newblock {Smoothing Noisy Data with Spline Functions. {E}stimating the Correct
  Degree of Smoothing by the Method of Generalized Cross-Validation.}
\newblock \emph{Numerische Mathematik}, 1979.

\bibitem[Kirchner and Bolin(2022)]{kirchner2022necessary}
K.~Kirchner and D.~Bolin.
\newblock Necessary and sufficient conditions for asymptotically optimal linear
  prediction of random fields on compact metric spaces.
\newblock \emph{The Annals of Statistics}, 50\penalty0 (2):\penalty0
  1038--1065, 2022.

\bibitem[Stein(1993{\natexlab{b}})]{Stein_Simple_Condition}
M.~L. Stein.
\newblock {A simple condition for asymptotic optimality of linear predictions
  of random fields}.
\newblock \emph{Statistics and Probability Letters}, 17\penalty0 (5):\penalty0
  399--404, August 1993{\natexlab{b}}.

\bibitem[Wang and Jing(2022)]{2022_wang:_rigde}
W.~Wang and B.-Y. Jing.
\newblock Gaussian process regression: Optimality, robustness, and relationship
  with kernel ridge regression.
\newblock \emph{Journal of Machine Learning Research}, 23\penalty0
  (193):\penalty0 1--67, 2022.

\bibitem[Abramowitz and Stegun(1968)]{abramowitz1968handbook}
M.~Abramowitz and I.~A. Stegun.
\newblock \emph{Handbook of mathematical functions with formulas, graphs, and
  mathematical tables}, volume~55.
\newblock US Government printing office, 1968.

\bibitem[Narcowich et~al.(2005)Narcowich, Ward, and
  Wendland]{narcowich2005:_sobolev_bounds}
F.~Narcowich, J.~Ward, and H.~Wendland.
\newblock {S}obolev bounds on functions with scattered zeros, with applications
  to radial basis function surface fitting.
\newblock \emph{Mathematics of Computation}, 74\penalty0 (250):\penalty0
  743--763, 2005.

\bibitem[Arcang{\'e}li et~al.(2007)Arcang{\'e}li, L{\'o}pez~de Silanes, and
  Torrens]{arcangeli2007:_error_bounds}
R.~Arcang{\'e}li, M.~C. L{\'o}pez~de Silanes, and J.~J. Torrens.
\newblock An extension of a bound for functions in {S}obolev spaces, with
  applications to (m, s)-spline interpolation and smoothing.
\newblock \emph{Numerische Mathematik}, 107\penalty0 (2):\penalty0 181--211,
  2007.

\bibitem[Dubrule(1983)]{Dubrule1983CrossVO}
O.~Dubrule.
\newblock {Cross validation of kriging in a unique neighborhood}.
\newblock \emph{Journal of the International Association for Mathematical
  Geology}, 15:\penalty0 687--699, 1983.

\bibitem[Brockwell and Davis(1987)]{brockwell1987time}
P.~J. Brockwell and R.~A. Davis.
\newblock \emph{Time series: theory and methods}.
\newblock Springer science \& business media, 1987.

\bibitem[Postnikov(1988)]{postnikov1988:_introduction}
A.~G. Postnikov.
\newblock \emph{Introduction to analytic number theory}, volume~68.
\newblock American Mathematical Soc., 1988.

\bibitem[Taylor and Hu(1987)]{taylor1987:_slln}
R.~L. Taylor and T.-C. Hu.
\newblock Strong laws of large numbers for arrays of rowwise independent random
  elements.
\newblock \emph{International Journal of Mathematics and Mathematical
  Sciences}, 10:\penalty0 805--814, 1987.

\bibitem[van Der~Vaart and Wellner(1996)]{van1996:_weak}
A.~W. van Der~Vaart and J.~A. Wellner.
\newblock \emph{Weak convergence and empirical processes: with applications to
  statistics}.
\newblock Springer Science \& Business Media, 1996.

\end{thebibliography}

\end{document}